%% file: Anderson_Thakur.tex
\documentclass[10pt,a4paper,reqno,oneside]{amsart}

\usepackage{amsrefs}

\usepackage{array}
\usepackage{url}
\usepackage{accents}
\usepackage{multicol} 
\usepackage[linktocpage = true]{hyperref}
\usepackage[pdftex]{graphicx}
\usepackage{epstopdf}

\usepackage{todonotes}

\hypersetup{
 colorlinks = true,
 citecolor = blue,
}

\usepackage{color}
\definecolor{red}{rgb}{1,0,0}
\definecolor{green}{rgb}{0,1,0}
\definecolor{blue}{rgb}{0,0,1}
\definecolor{refkey}{gray}{.625}
\definecolor{labelkey}{gray}{.625}
\usepackage[a4paper]{geometry}
\usepackage{amsmath,amssymb}

\usepackage{amssymb}
\usepackage{amsmath}
\usepackage{amsthm}
\usepackage{enumerate}
\usepackage{graphicx}

\usepackage{tikz-cd}
\allowdisplaybreaks

\def\GL{\mathrm{GL}}
\def\Gal{\mathrm{Gal}}

\def\dim{\mathrm{dim}}

\def\Cl{\mathrm{Cl}}

\def\id{\mathrm{Id}}

\makeatletter

\newtheorem{thmletter}{Theorem}

\theoremstyle{plain}
\newtheorem{thm}{\protect\theoremname}[section]
\newtheorem{prop}[thm]{\protect\propositionname}
\newtheorem{cor}
[thm]{\protect\corollaryname}
\newtheorem{lem}[thm]{\protect\lemmaname}

\theoremstyle{definition}
\newtheorem{example}[thm]{\protect\examplename} 
\newtheorem{defn}[thm]{\protect\definitionname} 
 
\newtheorem{notation}[thm]{\protect\notationname}

\AtBeginDocument{
	
}

\makeatother

  \providecommand{\corollaryname}{Corollary}
  \providecommand{\examplename}{Example}
  \providecommand{\lemmaname}{Lemma}
  \providecommand{\propositionname}{Proposition}
  \providecommand{\theoremname}{Theorem}
  \providecommand{\definitionname}{Definition}
  \providecommand{\remarkname}{Remark}
  \providecommand{\notationname}{Notation}

  \makeatletter
  \newcommand{\be}{%
  \begingroup
  \eqnarray%
   \@ifstar{\nonumber}{}%
  }

\newcommand{\A}{\mathcal{A}}
\newcommand{\K}{\mathcal{K}}
\newcommand{\C}{\mathcal{C}}

\newcommand{\ang}[1]{\langle #1 \rangle}

\newcommand{\floor}[1]{\lfloor #1 \rfloor}
\newcommand{\bi}[2]{\left[{{#1}\atop#2}\right]}

\newcommand{\Sgn}{\operatorname{Sgn}}
\newcommand{\Res}{\operatorname{Res}}

\newcommand{\Log}{\operatorname{Log}}

\newcommand{\ft}{\mathfrak{t}}

\newcommand{\bket}[1]{{\lfloor{#1}\rceil}}

\newcommand{\FS}[1]{\ensuremath{\left\{ #1 \right\}}}

\newcommand{\FT}[1]{\ensuremath{\left\langle #1 \right\rangle}}

\makeatother

\begin{document}

\title[Anderson generating function of rank-one Drinfeld Module ]{Anderson generating function of rank-one Drinfeld Module over rational function fields}
\thanks{ Research partially supported by NSFC grant 
	12071247(Hu),  NSFC grant 12501094(Huang).}
\thanks{$^*$ The corresponding author.}

\author{Chuangqiang Hu}
\address{School of Mathematics, Sun Yat-Sen University, Guangzhou, 510275,
P. R. China}

\email{\href{huchq5@mail.sysu.edu.cn}{huchq5@mail.sysu.edu.cn}}
 
\author{Xiao-Min Huang}
\address{School of Mathematics and statistics, 
Guangdong University of Technology, Guangzhou, 510006,  P. R. China}
\email{\href{mahuangxm@gdut.edu.cn}{mahuangxm@gdut.edu.cn}}

\author{Stephen S.-T. Yau$^*$}
\address{Department of Mathematical sciences, Tsinghua University, Beijing, 100084, P. R. China}
\email{\href{yau@uic.edu}{yau@uic.edu}}

\allowdisplaybreaks

\begin{abstract}
We establish a fundamental breakthrough in rank-one Drinfeld module arithmetic by deriving explicit formulas over the integral domain $\A = H^{0}(\mathbb{P}^1 - P_{\rho}, \mathcal{O}_{\mathbb{P}^1})$, which generalizes the classical polynomial ring ($N=1$) to the projective line associated with an infinite place of degree $N \geqslant 2$. This fills a longstanding gap by developing a comprehensive parallel to Carlitz module theory—foundational in positive-characteristic arithmetic—for the understudied case of infinite places of degree $>1$.
We construct Anderson generating functions for these modules and link them to the Carlitz period via Pellarin’s series, exponential torsion modules, and logarithmic deformations. These constructions provide powerful tools for studying such Drinfeld modules and their associated $L$-series, central to modern number theory.
A key result reveals a critical distinction from Carlitz theory: the standard Anderson generating function residue formula fails due to Galois group action. We resolve this obstruction by introducing an exponential action, enabling simultaneous study of all twisted exponential functions—a major methodological advance.
We further show that Anderson generating function computation involves the dual of Drinfeld modules, leading to an appropriate residue formula modification. Notably, our natural approach generalizes to arbitrary Dedekind domains, extending our results beyond $\A$ and opening new avenues in Drinfeld module theory.
\end{abstract}
\maketitle{}

\textbf{Key words:} ~Drinfeld module; Anderson Motive; Shtuka function; Anderson generating function; Carlitz period.

	{\textbf{AMS subject classification:} 11G09; 11R58; 11M38; 12H10. }

\tableofcontents

\section{Introduction}

\subsection{Drinfeld Module}
A Drinfeld module is an analogue of an elliptic curve in function field theory; it was introduced by Drinfeld \cite{DVG74} in the 1970s to prove the $\GL_2$-Langlands conjecture. Drinfeld modules encompass a rich arithmetic theory, and several important special cases, most notably the Carlitz module $\C$ had already been investigated by Carlitz (see \cite{CL35}) long before Drinfeld's general definition.

The work of many authors such as Drinfeld, Hayes, Goss, Gekeler, Anderson, Brownawell, Thakur, Papanikolas and Cornelissen has established many striking similarities and yet a few astounding differences between elliptic curves and Drinfeld modules (particularly in rank-two cases);  the reader may consult \cite{Goss96, P23, PB21, Tha04} for a more elaborate discussion.  The comparison of similarities and differences in the statements (and proofs) between corresponding results in the classical and Drinfeld theories provides valuable insights into some common themes in number theory. 

Explicit formulas for rank-one Drinfeld modules can be used to compute cyclotomic extensions of function fields. In essence, such explicit expressions are not merely auxiliary tools but lie at the foundation of the study of cyclotomic field extensions in function fields. In the same spirit as the classical Kronecker–Weber theorem, by translating abstract arithmetic properties into computable polynomial forms, they provide powerful tools for the study of Abelian extensions. For further details on their significant role in class field theory over global function fields, which is also highly relevant to the present work, we refer the reader to \cite{Wade46, CL35, Car38, Hay74}.

\subsection{Anderson Generating Function}
 Originally, Anderson generating functions were applied to prove fundamental results on the uniformization of $t$-motives in Anderson's paper \cite{AGW86}*{Chapter 3.2}.
They have also been used to establish connections between $v$-adic and de Rham realizations of $t$-modules and $t$-motives in \cite{AGW93}. 
 
    Let $\theta$ and $\ft$ be independent variables over the finite field $\mathbb{F}_q$ with $q$ elements. Based on an interpretation of the rigid analytic trivialization of the Carlitz motive, namely the Anderson motive derived from the Carlitz module $ \C $, Anderson and Thakur \cite{AT90}*{Chapter 2}
introduced the function  
\[
\omega_\C(\ft) = (-\theta)^{1/(q - 1)} \prod_{i = 0}^{\infty} (1 - \frac{\ft}{\theta ^{q^i}})^{-1}.
\]
The Anderson-Thakur function $ \omega_ \C (\ft)$ converges in the Tate algebra $\mathbb{T} \subseteq \mathbb{C}_\infty[\![\ft]\!]  $ of rigid analytic functions in $\ft$ over the closed unit disk of $\mathbb C_{\infty}$.  Furthermore, by Anderson’s general theory of scattering matrices (see \cite{AGW86}*{Proposition 3.3.2} and \cite{AT90}*{Section 2.5}, the residue of $\omega_\C(\ft) $ at $\ft = \theta$ yields a fundamental period of the Carlitz module, from which we recover the period $ \tilde{\pi}_{\C} $ via 
\begin{equation}\label{eq:residue}
\text{Res}_{\ft=\theta} (\omega_\C(\ft)) = -\tilde\pi_{\C} = (-\theta)^{q/(q - 1)} \prod_{i = 1}^{\infty} \left(1 - \theta^{1 - q^i}\right)^{-1}  .
\end{equation}

For a general Drinfeld module $ \phi $, the Anderson generating function associated with  $U \in \mathbb{C}_{\infty} $ is naturally defined as the series 
\[
g_{\phi}(U; \ft) = \sum_{i = 0}^{\infty}  \exp_{\phi} (\frac{U}{\theta^{{i+1}}}) \ft^i .
\]
In particular, the following identity holds, as established in  \cite{AGB04}*{Proposition. 5.1.3}:
\[
g_{\C}(\tilde{\pi}_{\C}; \ft) = \omega_\C(\ft). \]
Pellarin \cite{PF14} showed that  
\begin{equation}
\label{eq:Pellarin}
g_{\phi}(U; \ft) = \sum_{i = 0}^{\infty} \frac{ U^{q^i}}{D_i (\theta^{q^i} - \ft )},
\end{equation}
where the $ D_i $ are coefficients appearing in the expansion of $ \exp_\phi$.  Papanikolas \cite{P08} introduced a deformation of the Carlitz logarithm 
\[
    \Log(\xi;\ft) = \xi + \sum_{i = 1}^{\infty} \frac{\xi^{q^i}}{ \prod_{j=1}^i ( \ft - \theta^{q^j} ) } 
\]
in order to resolve the so-called ``folklore conjecture''.
El-Gundy and Papanikolas \cite{ElP14} noticed that $\Log(\xi;\ft)$ can be expressed in terms of the Anderson generating function: 
\begin{equation}\label{eq:Logxi}
    \Log(\xi;\ft) = (\theta - \ft ) g_{\C}(U; \ft), 
\end{equation}
where $U = \exp_{\C}(\xi)$.

 Anderson generating functions also play a central role in transcendence theory, since they satisfy the Frobenius differential equation:
\begin{equation}
\tau \diamond \omega_\C  - (\ft - \theta)\omega_\C = 0. \label{eq:diff_eq}
\end{equation}
Here the twisted action $\tau \diamond \omega_\C$ is defined by raising each coefficient of $\omega_\C\in\mathbb{C}_\infty[\![\ft]\!]$ to the $q$-th power.
Consequently, applying the twisted polynomial  $\C_a$ to $\omega_\C$, we obtain 
\begin{equation}\label{eq:Cat}
\C_{a(t)} \diamond \omega_\C  = a(\ft)\cdot \omega_\C 
\end{equation}
for any $a(t) \in \mathbb{F}_{q}[t] $.
This implies that $\omega_\C $ serves as a representation for the Carlitz module so that many arithmetic properties can be derived through $\omega_\C$.   

Sinha \cite{Sin97}*{Chapter 4} demonstrated that the theory
of Anderson functions can be extended to certain $t$-modules with complex multiplication and that, more
generally, logarithms on these $t$-modules can be constructed systematically via the Anderson
generating functions (see \cite{GP18}*{Chapter 5}, as well as \cite{ElP14} and \cite{PF08}*{Chapter 4.2}).
 These functions serve as powerful tools for describing periods, quasi-periods, and module logarithms in positive characteristic.

Pellarin and Perkins \cite{PP16} developed new techniques to study a family of generating functions that share formal similarities with the classical Hurwitz zeta function.
They proved a collection of functional identities that establish explicit connection with certain deformations of the Carlitz logarithm introduced by Papanikolas. These identities naturally involve the Anderson–Thakur function and the Carlitz exponential function.

The Anderson generating functions have played a foundational role in recent work on special values of positive characteristic $L$-series, as well as in transcendence and algebraic independence problems.

\subsection{Pellarin $L$-series}
In 2012, Pellarin \cite{P12} introduced a class of deformations of the values of the Goss zeta
function for the polynomial ring $\mathbb{F}_q[\theta]$. Let $ \chi : \mathbb{F}_q[\theta] \to \mathbb{F}_q[\ft]$ by setting $ \chi(a) = a(\ft)$, the Pellarin $L$-series is given by 
\[
    L(\mathbb{F}_q[\theta]; s ) = \sum_{a \in \mathbb{F}_q[\theta]_+} \frac{\chi (a)} { a(\theta)^s } =  \sum_{a \in \mathbb{F}_q[\theta]_+} \frac{a (\ft)} { a (\theta)^s}. 
\]
He proved, using the theory of deformations of vectorial modular forms and other techniques, a formula for Pellarin $L$-series values at 
$1$, as well as some arithmetic properties of Pellarin $L$-series values at other positive integers.
Precisely, there exists an equality 
\[
    L(\mathbb{F}_q[\theta]; 1)=-\frac{ \tilde{\pi}_{\C}}{(\ft - \theta) \omega_\C(\ft)}. 
\]
This formula involves Anderson-Thakur function $\omega_{\C}$,  shtuka function $ f_{\C} (\ft)= \ft - \theta$ and the Carlitz period $\tilde{\pi}_{\C}$. He discussed that this formula serves to interpolate values of Goss $L$-series of Dirichlet type as well as Carlitz zeta values at some negative integers.

 There has been substantial research on Pellarin $L$-series in recent years, covering special value formulas as well as multivariable generalizations \cite{ANT16,AP14,AP15,APT18,APTR16,G13,PP16,P14MZ,P14P}. 
In \cite{HP22}, Huang and Papanikolas recently generalized the special value formula that takes values in Tate algebras, for the Pellarin $L$-function attached to
the Carlitz module and the Anderson-Thakur function to arbitrary-rank Drinfeld modules and their rigid analytic trivializations.

\subsection{Our Main Results}
 Green and Papanikolas \cite{GP18} gave a detailed account of sign-normalized rank-one Drinfeld $\A$-modules, where $\A$ stands for the
coordinate ring of an elliptic curve over a finite field. This work provided a parallel theory to the Carlitz module $\C$, i.e., $\A = \mathbb{F}_q[t]$. Using the precise formulas for the shtuka function for elliptic curves, they obtained a product formula for the Carlitz period and found a special value of Pellarin $L$-series in terms of the Anderson-Thakur function. Motivated by their idea, this paper undertakes a comprehensive study of Drinfeld  $ \A $-modules, where $\A $ is derived from the rational function field $ \mathbb{F}_q(t)$ associated with the infinite place $ P_{\rho}$ of degree $N \geqslant 2$. Such kinds of Drinfeld modules can be regarded as a natural extension of the Carlitz module, so that they share significant similarities. We develop techniques to express Anderson generating functions in terms of the defining polynomials of the Drinfeld module and its associated Anderson motive. Our main results are listed as follows.

\subsubsection{Anderson Motive from Shtuka Function}
Let $\mathbb{P}^1$ denote the projective line over the finite field $\mathbb{F}_q$, and $\mathcal{O}_{\mathbb{P}^1}$ the structure sheaf of $\mathbb{P}^1$.
We consider the Dedekind domain $ \A = H^0(\mathbb{P}^1 - P_{\rho}, \mathcal{O}_{\mathbb{P}^1})$, where $\rho$ is an irreducible monic polynomial over $\mathbb{F}_q$. Let $\eta$ be a root of $\rho $. The infinite place $ P_{\rho}$ splits into $N$ distinct places, written as $ [\eta^{q^i}] $ for $i = 0, \cdots, N-1$ in the field extension over $ \mathbb{C}_{\infty}$.  We replace the shtuka function $ f_{\C}(t) = t - \theta $ in the Carlitz case with the canonical choice of the shtuka function:
\[
    f(t) = \frac{1}{t - \eta} - \frac{1}{\theta - \eta } .
\] 
Using this shtuka function $f$, we equip the global section 
\[
\A_{\mathbb{C_\infty} } :=  H^{0}\left(\mathbb{P}^{1}_{\mathbb{C}_{\infty}}-\{[\eta], \cdots , [\eta^{q^{N-1}}]   \}, \mathcal{O}_{\mathbb{P}^{1}_{\mathbb{C}_{\infty}}}\right)\]
with the Anderson $\A$-motive structure $ M_{\Psi}$. In this way, we construct the Drinfeld $\A$-module $ \Psi $ associated with $ M_{\Psi}$.

\subsubsection{Exponential and Logarithm Functions}

As established by Anderson and Thakur \cite{Tha93}, the exponential $\exp_{(0)}$ and logarithm functions $\log_{(0)}$  associated with a Drinfeld $\A$-module $ \Psi $ are uniquely determined by the underlying shtuka function $f$, and their construction relies on residue theory.   El-Guindy and Papanikolas \cite{ElP13} generalized this classical framework by deriving explicit series expansions for the exponential and logarithm functions attached to higher-rank Drinfeld modules. These higher-rank formulations naturally specialize to the original Anderson–Thakur formulas in the rank-one case.

Building directly on the residue-theoretic approach of Anderson and Thakur, we obtain the following explicit expressions for the exponential and logarithm of a rank-one Drinfeld $\A$-module $\Psi$:
\begin{thmletter}
    \begin{enumerate}
        \item The exponential function $\exp_{(0)}$ of Drinfeld $\A$-module $ \Psi $ is an $\mathbb{F}_q$-linear entire power series given by  
    \[
        \exp_{(0)}(\xi) = \sum_{i =0 }^{\infty} \frac{\xi^{q^i}}{D_i},
    \]
    where the coefficients $ D_i $ satisfy the recursive formula:
    \begin{equation*} 
    D_0= 1, \qquad D_i = D_{i-1}^q \frac{\theta-\theta^{q^i}}{(\theta^{q^i}-\eta)(\theta-\eta)}.
	\end{equation*}
    \item 
    The logarithm function $\log_{(0)}$ of Drinfeld $\A$-module $ \Psi $  admits a power series representation:
\begin{equation*} 
\log_{(0)}(\xi) =\sum_{j= 0}^{\infty} \frac{(-1)^j\xi^{q^j}}{L_j},
\end{equation*}
where $ L_0 =1 $ and the coefficients $ L_j $ satisfy the recursion:
\[
		L_j = 	L_{j-1} \frac{\theta-\theta^{q^j}}{ (\theta - \eta^{q^{j-2}})(\theta-\eta)^{q^{j-1}} (\theta - \eta^{q^{-1}})^{{q^{j}}-q^{j-1}}}.  
\]
    \end{enumerate}
  
\end{thmletter}

Applying the explicit series expansions of  $\exp_{(0)}$ and $\log_{(0)}$ above, we further establish a residue formula that characterizes the rank-one Drinfeld 
 $\A$-module $\Psi$ as an $\eta^{\frac{1}{q}}$-type module as follows.
 
\begin{thmletter}\label{InTheThe}
 Let $ \omega^{(i)}$ (resp. $f^{(i)}$) denote the $i$-th twist of $\omega $ (resp. $f $). 
    The rank-one Drinfeld $\A$-module $\Psi$ associated with the shtuka function $ f $ is of $\eta^{\frac{1}{q}}$-type, represented as 
    \[
        \Psi_{a } = - \sum_{k=0}^{\deg (a )}\Res_{P_{\rho}} \left( \frac{a  \omega^{(k+1)}}{f^{(0)} \cdots f^{(k)}} \right)\tau^{k},     \]
    for each $ a \in \A$, where  $\omega$ is a specific differential of $\A$; see Equations \eqref{Eq:omegaj} and \eqref{Eq:h^j} for precise definitions. 
\end{thmletter}

\subsubsection{Carlitz Period}
 For $0 \leqslant k \leqslant N-1$, let $\Psi^{(k)}$ denote the $\sigma^k$-twist of Drinfeld module $\Psi$, and write $\exp_{(k)}$ for the corresponding twisted exponential function.  Let $ |\cdot |_{\eta} $ be the norm on $\mathbb{C}_{\infty} $ normalized such that 
    $ |\theta - \eta|_{\eta} = \frac{1}{q}$. Denote by $\sigma $ the Galois action such that $\theta \mapsto \theta, \eta \mapsto \eta^q $.
    Define the period constant $\tilde{\pi}$ via the infinite product
\[ 
\tilde{\pi}= (-\Theta)^{\frac{\sigma^2}{q-1}}\Theta^{\sigma} \prod_{k=1}^{\infty} (\frac{\Theta^{q^k}}{\Theta^{q^k} -\Theta^{\sigma^k}}),
\] 
where  $ (-\Theta)^{\frac{\sigma^2}{q-1}}$ denotes a $(q-1)$-th root of $  -\Theta^{\sigma^2} $, and  $\Theta = \frac{1}{\theta - \eta } $.
 The constant $ \tilde{\pi} $ is called the Carlitz period of $ \Psi^{(2)} $ due to the following result.
 \begin{thmletter} 
    \begin{enumerate}
        \item With respect to the norm $ |\cdot |_{\eta}$, the kernel of $\exp_{(2)}$ is identical to the free $\A$-module generated by $ \tilde{\pi} $.
        \item Let $ I_{\infty} $ denote the ideal of $\A$ corresponding to the zero place of the infinite place $ P_{\infty}$. The kernel of the twisted exponential satisfies
 \[
    \ker \exp_{(j+2)} = \tilde \pi \frac{\Theta^{1+ \sigma + \cdots +\sigma^{j+1} }}{  { \sqrt[q-1]{\Theta^{\sigma^2- \sigma^{j+2}}  } } \cdot \Theta^{1+ \sigma}} I_{\infty}^{-j}.
  \]
    \end{enumerate} 
 \end{thmletter}

\subsubsection{Explicit Connection to Class Field Theory}
The set of rank-one Drinfeld $\A$-modules equipped with isogenies forms a well-behaved category  that is related to the homogeneous space of the class group of $\A$. 
In our setting, where $\A = H^{0}(\mathbb{P}^1- P_{\rho}, \mathcal{O}_{\mathbb{P}^1}) $, the class group $ \Cl(\A) $ is isomorphic to the cyclic group $ \mathbb{Z}_N $, generated by the ideal 
\[
    I_{\infty} = (T_0, \cdots, T_{N-1}),
\]
with $ T_i = \frac{t^i}{\rho(t)} $. The narrow class group $\Cl^{+}(\A) $ is isomorphic to  $ \mathbb{Z}_N \times  \mathbb{Z}_{W_N} $, where $W_N = \frac{q^N-1}{q-1} $. 
In what follows, we explicitly describe the rank-one Drinfeld $\A$-modules corresponding to $\Cl(\A)$ and $\Cl^+(\A)$, respectively.
\begin{thmletter}\label{Thm:expressions}
(1) Let $ \FT{\Theta}_{l}^i$ be the quantities contained in the Hilbert class field $ H $ of $\A$ defined by  
\[
  \FT{\Theta}_{l}^{i}=  
\begin{cases}
 \Theta^{ (1-q)\gamma_i(\sigma)} \Theta^{q^i},  \quad& \text{for } 0\leqslant i<l;\\
  \theta\eta^{ q^{i-1}(1-q)(i-l) - q^i} \Theta^{ (1-q)\gamma_i(\sigma)}\Theta^{q^i}, \quad & \text{for } l\leqslant i,
    \end{cases}
\] 
where
\[
    \gamma_k(\sigma):=\sum_{i=0}^{k-1} q^i \sigma^{k-1-i} \in \mathbb{Z}[\sigma].
\]
Then the standard Drinfeld $\A$-modules are precisely $\Psi^{(k)}$ for $k = 0, \cdots, N-1$, where $ \Psi = \Psi^{(0)}$ is defined by
\[
 \Psi_{T_n}  =\eta^{n q^{-1} } \Theta^{\gamma_N(\sigma) - W_N}\left(\tau+\FT{\Theta}^{N-1}_{N-n}\right)\circ\cdots\circ\left(\tau+\FT{\Theta}^{1}_{N-n} \right)\circ \left(\tau+\FT{\Theta}^0_{N-n}\right).
 \]

(2) Let $ U_{\rho}= \{ u_{k,\mu} \}_{k \in \mathbb{Z}_N, \ \mu^{W_N} =1 } $ be the set of roots of the equation 
 \begin{equation}\label{eq:xwN}
    x^{W_N} = \left(\frac{1}{\eta^{(k)}- \theta} \right)^{\gamma_N(\sigma) }.
\end{equation}
  There are exactly $N \cdot W_N $ distinct Hayes modules, parameterized by  $ u_{k,\mu} \in U_{\rho}  $. Specifically, the Hayes module $\psi^{u_{k,\mu}} $  corresponding to the root $ u_{k,\mu}$ is given by 
 \begin{equation*} 
      \psi^{u_{k,\mu}}_{T_n} = \left(\eta^{(k-1)} \right)^n \psi ^{\prime}_{n} \psi ^{\prime\prime}_{N-n},
\end{equation*}
where
\[
        \psi ^{\prime}_{n} = \left(\tau - \frac{\theta u_{k+N - 1 ,\mu \eta_*^{n}}}{\eta^{(k+N - 1)}}  \right)\circ   \left(\tau - \frac{\theta u_{k+N - 2 ,\mu \eta_*^{n-1}} }{\eta^{(k+N - 2)}} \right)\circ\cdots \circ  \left(\tau - \frac{\theta u_{k+N - n ,\mu \eta_*}}{\eta^{(k+N -n)}}  \right);
\]
and
\[
        \psi ^{\prime \prime}_{N-n} = (\tau -   u_{k+N - n -1,\mu} )\circ(\tau -   u_{k+N - n -2, \mu})\circ\cdots \circ (\tau -   u_{k+1,\mu} )\circ(\tau -   u_{k,\mu} );
\]
with $ \eta_* = \eta^{\frac{1-q}{q}}$ and $\eta^{(i)}=\eta^{q^i}$.
\end{thmletter}

The ideals of $\A$ act on Drinfeld modules via isogenies. For an ideal $I$ of $\A$, let $\psi_I $ denote the annihilator of $I$. Then $ \psi_I $ defines an isogeny from $\psi$ to another Drinfeld module $\psi^I$, i.e., 
\[
    \psi_I \psi_a = \psi^I_a  \psi_I ,
\]
for all $a\in \A$. In this notation,  $ \Psi^{I_\infty^k} $ is isomorphic (though not equal) to the $\sigma^k$-twist 
$ \Psi^{(k)}$  of $ \Psi $.  

As a fundamental result of Hayes' theory \cite{Hay74,Ha92,Goss96, V-S06}, the minimal definition field of rank-one normalized Drinfeld modules is identical to the narrow class field.
Theorem \ref{Thm:expressions} shows that the narrow class field of $K$ is $ H^+ = K(u_{k,\mu}, \eta)  $, for any $ u_{k,\mu} $ defined in Equation \eqref{eq:xwN}. For $\sigma \in \Gal(H^+/K)$, let $\sigma\psi$ denote the $\sigma$-twist of $\psi$, obtained by letting $\sigma$ act on the coefficients of $\psi$.  For any ideal $I$ of $\A$, let $(H^+/K; I) $ denote its Artin symbol. Then Hayes' theory implies that 
\begin{equation}\label{eq:psiI}
    \psi^{I} = (H^+/K; I) \psi  . 
\end{equation}
We compute directly in this paper that 
$ (H^+/K; I_\infty^l) = \sigma_{\infty}^l $, where 
  \[ \sigma_{\infty} : u \mapsto \Theta^{q-1} u^q, \qquad \eta \mapsto \eta^{q }. 
    \]
The compatible equality (specializing Equation \eqref{eq:psiI})
\[
    (\psi^{u_{k,\mu}})^{I_\infty^l} = \sigma_{\infty}^l \psi^{  u_{k,\mu}} 
\]
constitutes one of the key technical ingredients in this paper.
\subsubsection{Generalization of the Anderson-Thakur Function}
We first slightly modify the definition of the Tate algebra to obtain a new algebra $\mathbb{T}_{\eta}$, defined as the set of all analytic functions $f \in \mathbb{C}_{\infty}[\![T_0, \dots, T_{N-1}]\!]$ that converge on the domain
 \[  \mathcal{D} = \{ z \in \mathbb{C}_{\infty} |~  |T_i(z)|_{\eta} \leqslant 1,~\text{for 
 all } i=0,1,\cdots,N-1  \} . 
 \]
With this extension, we consider the series
\[\omega_{f} (\ft)= (-\Theta)^{\frac{1}{q-1}} \prod_{k=0}^{\infty} \frac{\Theta^{q^k}}{\Theta^{q^k} - \frac{1}{\ft - \eta^{(k)}}}.
\]
This series is well-defined and belongs to $ \mathbb{T}_{\eta} $. We refer to it as the Anderson-Thakur function for $\A$ due to the following theorem.
\begin{thmletter}
The solution space of the functional equation
\[
    \tau \diamond \omega =f(\ft)\cdot\omega
\]
is a one-dimensional $ \K $-linear space, and it admits $ \omega_f (\ft)$ as a basis.
\end{thmletter}  

Analogously to \eqref{eq:Cat} in the Carlitz module case, the function  $\omega_{f}(\ft)$ serves to trivialize the Drinfeld modules $\Psi$. More precisely, it satisfies the following relation 
\[
    \Psi_{a} \diamond \omega_f(\ft) = a(\ft) \cdot \omega_f(\ft),
\]
for all $a\in \A$.

Unlike the case of the Carlitz module, the residue of $\omega_f(\ft)$ no longer recovers the Carlitz period of $\Psi$. In fact, the lattice of $\exp_{(0)}$ fails to be cyclic in our setting, so the Carlitz period cannot be defined in the same way as for the Carlitz module. 
To remedy this, we define the dual of $\Psi$ by  
\[ \Phi = C_{0,2}^{-1} \circ \Psi^{(2)} \circ C_{0,2} ,\]
where $C_{0,2}$ denotes a $(q-1)$-th root of $\Theta^{\sigma^2-1} (\Theta^{\sigma-1})^{q-1}$. It can be checked that the lattice of $\exp_{\Phi}$ is generated by a constant $\tilde{\pi}_{\Phi}$, which we call the Carlitz period of $\Phi$.
The following result is analogous to Equation \eqref{eq:residue}.
\begin{thmletter}
    Let $ \tilde{\pi}_{\Phi}$ denote the Carlitz period of $ \Phi $. Then $\tilde{\pi}_{\Phi}$ can be recovered via the residue
\[
    \Res_{\ft=\theta} \omega_{f}(\ft) d\frac{1}{\ft - \eta} = - \tilde\pi_{\Phi}. 
\]

\end{thmletter}

 \subsubsection{Generalization of Generating Functions over $\A$}
We mainly introduce the following three types of generating functions associated with the shtuka function $ f $.
\begin{defn}  The Pellarin-type generating function associated with $f$ is defined as 
\begin{equation*} 
  G(U; \ft )= \sum_{n =0 }^{\infty } \frac{U^{q^n}} { D_n Q_n(\ft) },  
\end{equation*}
where $Q_n(\ft)$ (a rational function in $\ft$) is defined by:
\begin{equation*}
    Q_n(\ft) =\left(\frac{\theta^{q^n} - \eta}{\theta^{q^n} - \eta^{(n)}} \right)^2  \left(\frac{1}{\theta^{q^n} - \eta} - \frac{1}{\ft - \eta } \right).
\end{equation*}
\end{defn}

\begin{defn}
    For any $ U \in \mathbb{C}_{\infty}$,  the Anderson-type generating function is defined as 
\[
H(U;\ft ) := \left( \frac{t- \eta^{(1)} }{t - \eta  }(\frac{1}{t - \eta} - \frac{1}{\ft - \eta })^{-1} \triangleright_{\Phi} U \right) \frac{\theta - \eta  }{\theta - \eta^{(1)}  }, 
\]
where $\triangleright_{\Phi}$ denotes an exponential action (see Definitions \ref{Def:expgU} and \ref{Def:expgtU}) such that 
\[
    a(t) \triangleright_{\Phi} U = \exp_{\Phi} (a(\theta) U), 
\]
for all $a(t) \in \K$.
\end{defn}
\begin{defn} 
  For $ \xi \in \mathbb{C}_{\infty} $, define the logarithmic generating function associated with $f$ as 
\begin{equation*}
      \Log (\xi; \ft ) = \xi + \sum_{n=1}^{\infty} \frac{\xi^{q^{n}}}{f^{(1)}(\ft) \cdots f^{(n)}(\ft)}.
\end{equation*}
\end{defn}

The following theorem establishes fundamental connections among the three families of generating functions defined above. These results extend and parallel classical relations in the Carlitz setting, including Equations \eqref{eq:Pellarin} and \eqref{eq:Logxi}, as well as the characteristic functional equation for the canonical generating function $g_{\C}(U;\ft)$:
\[
\tau \diamond g_{\C}(U; \ft)   = (\ft - \theta)  g_{\C}(U; \ft) + \exp_{\C}(U) .
\]
\begin{thmletter}
With the notions introduced above, the following key equalities hold:
    \begin{enumerate}
    \item  The Anderson–Thakur function $\omega_{f}(\ft)$ coincides with a special evaluation of the Pellarin-type generating function:
\[
    \omega_{f}(\ft)= G (\tilde{\pi}_{\Phi} ; \ft ).
\]
\item  The generating function $G(U;\ft)$ satisfies the Frobenius differential equation:
\[
        \tau \diamond G(U;\ft) = f(\ft) \cdot G(U;\ft) + \exp_{\Phi} (U). 
\]
\item Let $\xi = \exp_{\Phi}(U)$. The three
types of generating functions are related by
\[
    G(U;\ft) = H(U;\ft)= -\frac{1}{f(\ft)} \cdot \Log(\xi; \ft).
\]
    \end{enumerate}
\end{thmletter}

\subsection{Further Research}
 One of our primary goals is to develop a fully explicit theory for the Drinfeld modules $\Psi$ and $\psi^u$ over the coordinate ring $\A$, analogous to the classical Carlitz module framework, and to investigate their analogous arithmetic properties. It is natural to ask whether the current methods can be extended to general domains with higher genus and associated with infinite places of degree $\geqslant 2$.  Our main results provide important insights for researchers who would like to derive general generating functions for general Drinfeld modules and Anderson motives. In particular, for generalizing the Anderson-type generating function for a Drinfeld module $ \Psi $,  one shall determine its dual $ \Psi^{\vee}$, take the sum of all twisted exponential functions and study the isogeny relations. 
Roughly speaking, the Anderson-type generating function for an arbitrary Drinfeld module is  naturally defined by
\[
    H(U;\ft) = (\lambda(t;\ft) \triangleright_{\Psi^{\vee}} U) \cdot \text{Constant},
\]
where the kernel function $\lambda(t;\ft)$ admits a simple pole at $\ft = t$ and is regular at all other finite points. Furthermore, this type of generating function  has potential applications to evaluating special values of $L$-series over $\A$, in the spirit of Pellarin’s foundational work. We will address this problem in our forthcoming paper.

\subsection{Outline}
The outline of the paper is as follows.
First, in Section \ref{Sec:NandP}, we introduce the necessary notions related to Drinfeld modules.
Next, in Section \ref{Sec:PL}, we compute the exponential and logarithmic functions for the standard Drinfeld module $\Psi$, and use them to determine the period lattice of $\Psi^{(2)}$.
The main goal of Section \ref{Sec:Factor} is to establish factorization formulas for both the Hayes module and the standard Drinfeld module. We also study the ideal action and Galois action on these modules. In Section \ref{Sec:IPL}, we 
discuss the isogeny relations between $\Psi^{(i)}$ and $\Psi^{(j)}$ and then compute the period lattices of $\Psi^{(j)}$. We also develop the exponential action $ \triangleright_{\phi} $ to describe the isogeny relation. 
Finally, in Section \ref{Sec:AGF}, we show that the solution space of the corresponding differential operator $ \nabla^f $ forms a one-dimensional vector space spanned by the Anderson-Thakur function.
In our main result, we generalize this classical function into three types: Pellarin type, Anderson type and a deformed logarithm function.

\section{Notations and Preliminaries}\label{Sec:NandP}

\paragraph{\textbf{Fields and rings}}
\begin{itemize}
\item $ \mathbb{F}_q $ : finite field with $q$ elements
\item $ \K = \mathbb{F}_q(t) $ : rational function field
\item $ K = \mathbb{F}_q(\theta) $ : copy of $\K$ via $t \mapsto \theta$
\item $\rho(t) $, $N$ : monic irreducible polynomial over $\mathbb{F}_q$ of degree $N$ 
\item $ v_\eta $ : valuation on $\mathbb{C}_\infty$ normalized so that $v_\eta(\theta-\eta)=1$
\item $ |\cdot|_\eta $ : associated absolute value: $|x|_\eta = q^{-v_\eta(x)}$
\item $ P_\rho $ : place of $\K$ corresponding to irreducible monic polynomial $\rho$ of degree $N \geqslant 2$
\item $\mathbb{P}^1 $ : the projective line over finite field $ \mathbb{F}_q$
\item $ \A $ : Dedekind domain $H^0(\mathbb{P}^1 - P_\rho, \mathcal{O}_{\mathbb{P}^1})$
\item $ A  $ : copy of $\A$ via $t \mapsto \theta$
\item $ L $ : an $\A$-field with injective homomorphism $\iota: \A \to L$
\item $ \A_L = \A \otimes_{\mathbb{F}_q} L $ : base change to $L$, usually $L = \mathbb{C}_\infty$.

\item $ \sigma $ : the Frobenius twist in $\Gal(\mathbb{F}_{q^N}/\mathbb{F}_q)$ (resp. $\Gal(H/K)$), i.e., $\sigma(\eta)=\eta^q $ and $ \sigma(\theta)= \theta $
\item $ (-)^{(k)} $ :  the $q^k$-th power map on $ L $, i.e., $x \mapsto x^{q^k}$
\item $ \mathbb{Z}\{\sigma\} $ : group algebra generated by  $\sigma$

\item $\mathbb{P}^1_{L} $ : the projective line over $ L $, usually $L = \mathbb{C}_\infty$

\item $ \mathbb{C}_\infty $ : completion of an algebraic closure of $K_\rho$

\item $ K_\rho $ : $P_\rho$-adic completion of $K$
\item  $\mathbb{F}_{\rho}$ :  the residue field of $\A$ at $ P_{\rho } $
\item $ \tau $ : $q$-th power Frobenius endomorphism, acting on twisted polynomials
\item $ L \{\tau\} $ : twisted polynomial ring

\end{itemize}

\paragraph{\textbf{Drinfeld modules and related objects}}
\begin{itemize}
\item $ \phi, \psi $ : Drinfeld $\A$-modules
\item $ \Psi $ : standard rank-one Drinfeld $\A$-module (of type $\eta^{(-1)}$)
\item $ \Psi^{(i)} $ : the $\sigma^i$-twist of $\Psi$ (associated with $f^{\sigma^i}$)
\item $ \Phi $ : dual Drinfeld module: $\Phi_a = C_{0,2}^{-1} \Psi^{(2)}_a C_{0,2}$
\item $ \psi^x $ : Hayes $\A$-module with $\psi^x_{I_\infty} = \tau - x$ for $x \in U_\rho$
\item $ \psi^u $ : Hayes module for the fixed root $u$ (with $u^{W_N}=(-\Theta)^{\gamma_N(\sigma)}$)
\item $ \operatorname{LT}_\phi(a) $ : leading term of $\phi_a$ (as a twisted polynomial)
\item $ \phi_I $ : annihilator of an ideal $I$ (monic maximal right-divisor of $\{\phi_a : a\in I\}$)
\item $ \exp_\phi,~ \log_\phi $ : exponential and logarithm functions of $\phi$
\item $ \exp_{(j)}, ~\log_{(j)} $ : exponential and logarithm of $\Psi^{(j)}$
\item $ \tilde\pi $ : Carlitz period of $\Psi^{(2)}$ 
\item $ \tilde\pi_\C $ : Carlitz period of $ \C $
\item $ \tilde\pi_\Phi $ : Carlitz period of $\Phi$, $\tilde\pi_\Phi = \tilde\pi / C_{0,2}$

\item $ D_n $ : coefficients of $\exp_{(0)}$: $\exp_{(0)}(\xi)=\sum_{n\geqslant 0} \xi^{q^n}/D_n$
\item $ L_n $ : coefficients of $\log_{(0)}$: $\log_{(0)}(\xi)=\sum_{n\geqslant 0} (-1)^n \xi^{q^n}/L_n$
\item $ D_n^{\sigma^j} $ : the $\sigma^j$-twist of $D_n$ (coefficients of $\exp_{(j)}$)
\item  $D_n^{\Phi}$, $D_n^{\phi}$ : coefficients of the exponential function of $ \Phi $ and $\phi$, respectively

\item $ C_{i,j} $ : isogeny constant: $\sqrt[q-1]{\Theta^{\sigma^j-\sigma^i}}\;\Theta^{\sigma^i(\gamma_{j-i}(\sigma)-W_{j-i})}$
\item $ \lambda $ : isogeny between $\Psi^{(i)}$ and $\Psi^{(j)}$
\end{itemize}

\paragraph{\textbf{Shtuka function and motives}}
\begin{itemize}
\item $ \eta $ : a root of $\rho$; $\mathbb{F}_\rho = \mathbb{F}_q(\eta)$
\item $ \Theta $ : $1/(\theta-\eta)$
\item $ f(t) $ : shtuka function: $f(t)=\frac1{t-\eta}-\Theta =\frac{1}{t-\eta}-\frac1{\theta-\eta}$

\item $ f^{(i)}(t) $ : the twist of $f $ by raising its coefficients to the $q^k$-th power, i.e.,  $f^{(i)}(t)=\frac{ 1}{t-\eta^{(i)}}- \Theta^{q^i}$
\item $ f^{\sigma^i}(t) $ : the $\sigma^i$-twist of $f $, i.e.,  $f^{\sigma^i}(t)= \frac{1}{t-\eta^{(i)}}-\frac{1}{\theta-\eta^{(i)}}$
\item $ \omega^{(j)} $ : differential form: $h^{(j)}(t)dt$ with $h^{(j)}(t)=-\frac{\theta^{(j-1)}-\eta^{(j-2)}}{\theta^{(j-1)}-\eta^{(j-1)}}\frac1{(t-\eta^{(j-2)})(t-\eta^{(j-1)})}$

\item $  [\alpha] = P_{t - \alpha} $ : zero of $t-\alpha$ in $\K\otimes_{\mathbb{F}_q}L$
\item  $ M_{\Psi} $ :  the Anderson motive of $ \Psi $, associated with the shtuka function $f $
\item $*$ : action of $\A_{L }\{\tau\}$ on  $ M_{\Psi} $
\item $ \mathbf{s}_0 $ : $ \mathbf{s}_0 =1 $, serves as the formal generator of the Anderson motive $M_\Psi$
 \item $\mathbf{s}_i $  : $\mathbf{s}_i = f^{(0)} \cdots f^{(i-1)} $ 
\item $ \mathbf{s}(x) $ : formal generator of $M_{\psi^x}$
\item   $\mathbf{s}(u_{k,\mu})$ :  generator of the associated Anderson motive of $ \psi^{u_{k,\mu}} $

\end{itemize}

\paragraph{\textbf{Galois theory}}
\begin{itemize}
\item $ T_i $ : $t^i/\rho(t)$ for $i=0,\dots,N-1$ (generators of $\A$)
\item $ I_\infty $ : ideal of $\A$ generated by $T_0,\dots,T_{N-1}$
\item $ I_0 $ : ideal generated by $T_0-1/\rho(0), T_1,\dots,T_{N-1}$
\item $ \Cl(\A) $ : ideal class group of $\A$ (isomorphic to $\mathbb{Z}_N$)
\item $ \Cl^+(\A) $ : narrow class group (isomorphic to $\mathbb{Z}_{W_N}\times\mathbb{Z}_N$)
\item $ W_i $ : $(q^i-1)/(q-1)$
\item $ \gamma_k(\sigma) $ : $\sum_{i=0}^{k-1} q^i \sigma^{k-1-i}\in\mathbb{Z}[\sigma]$
\item $ H = \mathbb{F}_q(\theta,\eta) $ : Hilbert class field of $ K$
\item $ H^+ = K(u,\eta) $ : narrow class field of $ K$ with respect to $\operatorname{Sgn}_\eta$
\item $ \operatorname{Sgn}_\eta $ : sign function on $K_\rho$ with $\operatorname{Sgn}_\eta \left( T_i(\theta)\right)=\eta^i$
\item $ u $ : a root of $x^{W_N}=(-\Theta)^{\gamma_N(\sigma)}$ (parameter for Hayes modules)
\item $ u_{k,\mu} $ : $\sigma_\infty^k M_\mu(u)$ with $\mu^{W_N}=1$
\item $ U_\rho $ : set of all conjugates of $u$ over $K$, i.e., $U_\rho = \{ u_{k,\mu} 
 \}$
\item $ \sigma_\infty $ : automorphism of $H^+$: $\sigma_\infty(u)=(\theta-\eta)^{q-1}u^q$, $\sigma_\infty|_H=\sigma$
\item $ M_\mu $ : automorphism: $M_\mu(u)=\mu u$, $M_\mu|_H=\operatorname{id}$
\item $ \eta_* $ : $\eta^{(1-q)/q}$ (satisfies $\eta_*^{W_N}=1$)
\item $ \sigma_0 $ : $\sigma_\infty M_{\eta_*}$ with $\eta_*=\eta^{(1-q)/q}$
\item $\delta$ : map from $ U_{\rho}$ to $ H^+ $,  
 $\delta(x)=\frac{\theta}{\eta_x}x$
\item $ \left(\frac{H^+/K}{-}\right) $ : Artin symbol; $\left(\frac{H^+/K}{I_\infty}\right)=\sigma_\infty$, $\left(\frac{H^+/K}{I_0}\right)=\sigma_0$
\end{itemize}

\paragraph{\textbf{Anderson Generating functions}}
\begin{itemize}

\item $ \mathbb{C}_{\infty}[\![I_{\infty} ]\!] $ :  the formal completion algebra of $  \A_{\mathbb{C}_\infty } $ at the ideal $I_{\infty}$

\item $\mathbb{C}_{\infty}(\!(I_{\infty})\!)$ :  the quotient field of $\mathbb{C}[\![I_\infty]\!]$
\item $\mathbb{T}$ :  Tate Algebra contained in $\mathbb{C}_{\infty} [\![\ft]\!]$
\item $\mathbb{T}_{\eta}$ : Tate Algebra contained in $\mathbb{C}_{\infty} [\![I_\infty]\!]$
\item $ \diamond $ : action of $  \A_{\mathbb{C}_\infty }\{\tau\}  $ on $\mathbb{C}_{\infty}[\![I_{\infty} ]\!] $
\item $\nabla_a $ : the differential operator, $\nabla_a(\omega) = a\diamond \omega - a(\ft) \omega $, for $a\in \A$
\item $\nabla^f $ : the differential operator, $\nabla^f(\omega) = \omega^{(1)} - f \omega $
\item $ \omega_f(\ft) $ : Anderson-Thakur function 
\item $ \triangleright_\phi $ : exponential action: $g(t)\triangleright_\phi U = \exp_\phi(g(\theta)U)$ (extended to $\K \otimes_{\mathbb{F}_q} \mathbb{C}_{\infty} $)
\item $ G(U;\ft) $ : Pellarin-type generating function 
\item $ H(U;\ft) $ : Anderson-type generating function 
\item $ \Log(\xi;\ft) $ : deformed logarithm

\end{itemize}

\paragraph{\textbf{Auxiliary notations}}
\begin{itemize}
\item $ \bket{k} $ : $\Theta^{q^k}-\Theta^{\sigma^k} = -f^{(k)}(\theta)$ 
\item $ \FS{x}_{j}^{i} $ : $\begin{cases} \sigma_\infty^i(x), &0\leqslant i<j \\ \delta\sigma_0^{i-j}\sigma_\infty^j(x), & i\geqslant j \end{cases}$ 
\item $ \FT{\Theta}_{l}^{i} $ : quantities in $H$: $\begin{cases} \Theta^{(1-q)\gamma_i(\sigma)}\Theta^{q^i}, & 0\leqslant i<l \\ \theta\eta^{q^{i-1}(1-q)(i-l)-q^i}\Theta^{(1-q)\gamma_i(\sigma)}\Theta^{q^i}, & i\geqslant l \end{cases}$
\item $ \Upsilon_k(t) $ : rational function:  
$\begin{cases}
    (t-\eta^{(1)})\prod_{i=k}^{-1} \frac{1}{t - \eta^{(-i)}} & \text{for $k\leqslant-1$; } \\
    \prod_{i=-1}^{k-1} (t - \eta^{(-i)}) & \text{for $k\geqslant 0$}.
\end{cases}$
\end{itemize}



%
	

\subsection{Drinfeld Module}

We introduce the necessary notions concerning Drinfeld modules in this section. Let $ X $  be a smooth algebraic curve over the finite field $ \mathbb{F}_q$ associated with an infinity $ \infty $. Denote by $ \mathcal{O}_X$ the structure sheaf of $X$. 
Let $ \A = H^0( X - \infty, \mathcal{O}_{X}) $ be a Dedekind domain associated with the open curve 
 $X-\{\infty\} $. The field $L$ equipped with an $\mathbb{F}_q$-algebra homomorphism $\iota : \mathcal{A} \to L$ is called an $\mathcal{A}$-field. Throughout this paper, we additionally require $\iota$ to be injective; this convention corresponds to working in generic characteristic in the standard literature. Let $L\{\tau \}$ denote the twisted polynomial ring, which is a non-commutative $ L$-algebra generated by
the $q$-th Frobenius endomorphism $\tau $ such that
$\tau a = a^{q} \tau $. An element contained in $L\{\tau \}$ is called a twisted polynomial. The ring $L\{\tau\}$ acts naturally on the algebraic closure $\bar{L}$ via
\begin{equation*}
\sum _{i=0}^{n} a_{i} \tau ^{i} : \ell \to \sum _{i=0}^{n} a_{i}
\ell ^{q^{i}}.
\end{equation*}
For a twisted polynomial $\lambda  $, we denote by $\partial \lambda $ the constant term of $ \lambda $.
\begin{defn}[Drinfeld Module]
    A Drinfeld $ \A$-module is a non-trivial algebra homomorphism 
    \begin{equation*}
\phi : \A \to L\{\tau \}, \quad a \mapsto \phi _{a},
\end{equation*}
    such that $ \partial \phi _{a} = \iota (a)$, where 
the non-trivial condition requires that $\phi _{a} \not \equiv \iota (a)$. 
\end{defn}

\begin{defn}[Isogeny and Isomorphism]
\label{Defn:isogeny}
 Suppose that $ \phi $ and $ \phi' $ are Drinfeld $\A$-modules over $L$. Let $L'$ be a field extension of $L$ and $ \lambda $  a twisted polynomial over $L'$.  We say that $ \lambda $ is an isogeny from
$ \phi $ to $ \phi ' $  over $L'$, if
\begin{equation*}
\lambda \phi _{a} = \phi '_{a} \lambda
\end{equation*}
holds for all $ a \in \A $. If $ \lambda $ is a nonzero constant in $ L' $,
then $ \lambda $ is called an isomorphism over $ L' $. In other words, the Drinfeld module
$ \phi $ is isomorphic to its conjugate
$\lambda  \phi \lambda^{-1} $.
\end{defn}

In the rest of this paper, we always choose the curve $X$ to be the projective space $ \mathbb{P}^1$ over $ \mathbb{F}_q $ associated with
 the function field of $ \K = \mathbb{F}_{q}(t) $.  
 Using the function field language, it is well known that the places of $ \K  $ are represented as 
\[
    \{ P_{\rho(t)} | \text{where $ \rho(t) $ is an irreducible monic polynomial} \}  \cup \{ P_{\infty} \}. 
\]
 Let $ \mathcal{O}_{\mathbb{P}^1}$ be the structure sheaf of $ \mathbb{P}^1 $. For a place $P$, we denote by $ v_P $ the valuation at $ P $. Then the global section  $ H^{0}(\mathbb{P}^1 -  \infty , \mathcal{O}_{\mathbb{P}^1})$ coincides with 
 \[
    \{ f \in \K | v_P (f) \geqslant 0 \text{ for $P \not = \infty$} \}.
 \] 
 For instance, the polynomial ring $ \mathbb{F}_q [t] $ is identical to $ H^{0}(\mathbb{P}^1 - P_{\infty}, \mathcal{O}_{\mathbb{P}^1})$ 
by choosing $ \infty = P_{\infty} $.

\begin{example}
    For the case of the polynomial ring, the Drinfeld $\mathbb{F}_q[t]$-module over the $\mathbb{F}_q[t]$-field 
    \[ \iota: \mathbb{F}_q[t] \to L , t \mapsto \theta ,\] 
    can be written explicitly as 
    \[
        \phi_t = \theta +  g_1 \tau + g_2 \tau^2 + \cdots + g_r \tau^r  
    \]
    for coefficients $ g_i \in L $. The integer $ r \geqslant 1 $ is called the rank of $ \phi $.  The canonical rank-one example is the Carlitz module $\C$, given by $\C_t = \theta + \tau$. 
\end{example}

\subsection{Dedekind Domain}
The natural generalization of the Carlitz module is derived from setting $ \infty $ to be some place $ P_\rho $ represented by the monic irreducible polynomial $ \rho $ of degree $N \geqslant 2$.
Our main purpose is to investigate the rank-one Drinfeld module in this setting. Now we denote the corresponding Dedekind domain by  
 \[ \A:=H^{0}( \mathbb{P}^1 -P_{\rho}, \mathcal{O}_{\mathbb{P}^1}).
\] 
 Fix a root $ \eta $ of $ \rho $. Then the residue field $\mathbb{F}_{\rho}$ of $\K$ at $ P_{\rho}$ is isomorphic to $  \mathbb{F}_{q}(\eta)$. 
 For $i = 0,\cdots, N-1$, set $T_i = \frac{t^i}{\rho(t)}$.
The ring $\A$ is generated as an $\mathbb{F}_q$-algebra by $T_0,\cdots, T_{N-1}$.
The divisor of $ T_i $ in the function field $ \K $ is given by 
\[
    (T_i) = i P_{t} +  (N - i ) P_{\infty} - P_{\rho} .
\]
In terms of ideals, the principal ideal generated by $ T_i $ admits the decomposition 
\begin{equation}\label{Eq:Ideals}
    T_i\cdot\A = I_0^{i} I_\infty^{N-i},
\end{equation}
where $ I_\infty $ is generated by $ T_0, \cdots, T_{N-1} $, and $ I_0 $ is generated by $ T_0-\frac{1}{\rho(0)},  T_1, \cdots, T_{N-1} $. 
\subsection{Rank One Drinfeld $\A$-Module} 
One of the most striking results in the theory of rank-one Drinfeld modules is their deep connection with class field theory for function fields. For a general Dedekind domain $\A$ over a finite field $
\mathbb{F}_q$, the following results are well-established;   we refer the
reader to \cite{Goss96}*{Proposition 7.2.20 and Chapter 7.4},
\cite{P23}*{Theorem 7.5.21} or \cite{PB21} for more details.
\begin{thm}\label{thm:homogeneous}
 The set of isomorphism classes of Drinfeld $\A$-modules forms a principal homogeneous space for the class group  $ \Cl(\A) $ of $\A$. 
\end{thm}
Recall the definition of the ideal class group of $ \A $:
\begin{equation*}
\Cl(\A) :=
\frac{\text{Ideal group of $\A$}}{\text{principal ideal group of $\A$}}.
\end{equation*}
In our setting, the ideal class group $ \Cl(\A) $ is generated by $ I_\infty$ (or $ I_0 $) with the relation 
\[ I_\infty^N = T_0 \A  .\]
 It follows that $ \Cl(\A) $ is isomorphic to $ \mathbb{Z}_N $.

\begin{defn}[Type]
    For a Drinfeld $\A$-module $ \phi $, we denote by
$ \operatorname{LT}_{\phi} (a) $ the leading term of $ \phi _{a} $.
    It is obvious that the quotient $ \operatorname{LT}_{\phi}(T_1)/\operatorname{LT}_{\phi}(T_0) $ is a root of $\rho$, say $ \eta^{(k)}$. We refer to the type of such a Drinfeld $\A$-module $ \phi $ as $\eta^{(k)}$.
\end{defn}
Due to Theorem \ref{thm:homogeneous}, we know that the map 
\[
    \phi \mapsto  \text{the type of $\phi$}
\]
gives an isomorphism from the class of rank-one Drinfeld $\A$-modules to the collection of roots of $ \rho $.  Moreover, the Galois group $ \mathbb{Z}_{N} $ acts transitively and equivalently on the both sides of this map.

\subsection{Hayes $\A$-Modules}
Analogous to Theorem \ref{thm:homogeneous},  the narrow class group of $\A$ is naturally related to the notion of Hayes $\A$-modules.
\begin{thm}\label{thm:homogeneous2}
 The set of Hayes $\A$-modules
 forms a principal homogeneous space for the narrow class group $ \Cl^+(\A) $ of $\A$.
\end{thm}

 We introduce the necessary notations concerning this topic below. 
We adopt the notions $ A = \iota(\A)$ (resp. $ K  = \iota(\K)$), which denote the copy of $ \A $ (resp. $ \K $) obtained by replacing $t $ with $ \theta$.  Suppose that $ K_\rho $ is a  $P_{\rho}$-adic completion of $ K $  containing $ \mathbb{F}_{\rho} $ as a subfield. 
Notice that $K_{\rho}$ is unique up to isomorphism.  We may assume that $\theta-\eta $ is a uniformizer of $ K_{\rho}$, so the valuation $v_{\eta} $ on $ K_{\rho} $ satisfies
\[
    v_{\eta}(\theta- \eta) =1 . 
\]
Define the $\eta$-norm $|\cdot|_{\eta} : = q^{-v_{\eta} (\cdot)}$ to be the absolute value associated with $v_\eta$. 
In this notation, we have 
 \[
     |\theta - \eta |_{\eta} = q^{-1}. 
 \]   
\begin{defn}\label{defn:sign}
       A sign function on $ K_\rho^* $ (resp. $K_{\rho}$) is a homomorphism 
	$ K_\rho^* \to \mathbb{F}_{\rho}^*$ that restricts to the identity map on $ \mathbb{F}_{\rho}^* $.  
\end{defn} 
For $ i \geqslant 1$, denote by $W_i$ the constant 
\begin{equation}\label{Eq:Wn}
W_i = \frac{q^i-1}{q-1}  .
\end{equation}
It follows from \cite{Goss96} that there are exactly $W_N $ non-equivalent sign functions on $ K_{\rho} $. 
\begin{example}
   In this example, we construct the typical sign functions on $ K_{\rho}$, denoted by $ \Sgn_{\eta} $. 
 Observe that $\rho\in K$ serves as a standard uniformizer of the place $P_{\rho}$ corresponding to the $\eta$-norm $|\cdot|_{\eta}$ on $K$. For an element $ a \in K_{\rho}^* $, we have the unique decomposition
\[
	a = \Sgn_{\eta}(a) \rho^{v_{P_\rho}(a)} \ang{a}_1 ,
\]
where $ \ang{a}_1 \in K_{\rho} $ is a $1$-unit (i.e., $ |1-\ang{a}_1|_{\eta} < 1$), and $\Sgn_{\eta}(a) \in \mathbb{F}_{\rho} ^* $. The function 
\begin{equation}\label{Eq:Sgn_eta}
    \Sgn_{\eta}:  K^*_{\rho} \to \mathbb{F}_{\rho}^* ,  \qquad a  \mapsto  \Sgn_{\eta}(a)  
\end{equation}
is a sign function as defined in Definition \ref{defn:sign}. 
In particular, we have 
\[
    \Sgn_{\eta} (  \iota (T_i)  )  =  \eta^{i}. 
\]
\end{example}

\begin{defn}[Hayes Module]
 Given a sign function such as $\operatorname{Sgn}_{\eta}$ on $K_{\rho}$, a Hayes module $\psi$ (with
respect to $ \operatorname{Sgn}_\eta$) over an $ \A$-field is a rank-one
Drinfeld module,  such that $ \operatorname{LT}_{\psi} $ is
a twisted sign function. In other words, there exists some
$\sigma_{\psi} \in \mathrm{Gal}(\mathbb{F}_{\rho}/ \mathbb{F}_{q})$ such that
\begin{equation*}
\operatorname{LT}_{\psi}(a) = \sigma_{\psi} \operatorname{Sgn}_{\eta}( \iota (a) )
\end{equation*}
for each $a \in \A$.
\end{defn}
When $\sigma_{\psi} $ is the $q^k$-th power map $  \sigma^k $ of $ \mathbb{F}_{\rho} $, the corresponding Hayes $\A$-module is of $\eta^{(k)}$-type due to the simple computation:
\[
   \operatorname{LT}_{\psi}(T_1) /   \operatorname{LT}_{\psi}(T_0) = \sigma^{k} \left(\Sgn_{\eta}\left( \iota T_1 \right) /\Sgn_{\eta}(\iota T_0 ) \right) = \eta^{q^k}. 
\]
The narrow class group of $ \A$  is defined as 
\begin{equation*}
\Cl^{+}(\A) :=
\frac{\text{Ideal group of $\A$}}{\text{Subgroup of principal $\Sgn$-positive ideals of $\A$}}.
\end{equation*}
It is straightforward to prove that 
\begin{equation}\label{Eq:cl+}
\Cl^+(\A)\cong \mathbb{Z}_{W_N}\times \mathbb{Z}_{N}.
\end{equation}
So there exist exactly $W_N \times N$ Hayes $\A$-modules according to Theorem \ref{thm:homogeneous2}.
\subsection{The Smallest Field of Definition}
The coefficients of Drinfeld modules are closely linked to the class field and narrow class field of $ \A $ due to Hayes' theory.
\begin{thm} 
\label{thm:Hilbert}
\begin{enumerate}
    \item The smallest field of definition for a rank-one $\A$-Drinfeld module is the Hilbert
    class field $ H$ of $ K $, which is the maximal unramified Abelian extension of $ K$ in which $\infty $ splits completely.
    \item For a fixed sign
    function $\operatorname{Sgn}$, the field of definition for a Hayes $\A$-module equals the narrow class field $H^{+}$ of $ K $ with respect to
    $\operatorname{Sgn}$.  The field $ H^{+} $ is totally and tamely ramified
    over all infinities of $ H $, and the Galois group
    $\mathrm{Gal}(H^{+} / H) $ is isomorphic to
    $\mathbb F_{\infty}^{*}/ \mathbb F_{q}^{*} $, where
    $ \mathbb F_{\infty}^{*} $ denotes the multiplicative group of the residue
    field at $\infty $.
    \item Under the Artin maps from the ideal group to the Galois group, the ideal action on Drinfeld modules  corresponds naturally to the Galois action on their coefficients.
\end{enumerate}

\end{thm}
The corresponding Hilbert field of $K$ is $ H= K(\eta)$, with Galois group
\[ \Gal( H / K) = \{ \sigma^i |  i \in \mathbb{Z}_{N}  \},
\] 
where $\sigma$ denotes the $q$-th power Frobenius map, i.e., $\eta \mapsto \eta^{(1)}$ and $ \theta \mapsto \theta $. As in Part (1) of Theorem \ref{thm:Hilbert}, we know that the smallest field of definition for rank-one Drinfeld $ \A $-modules equals $ H $.
Similarly, Part (2) of Theorem \ref{thm:Hilbert} implies that all Hayes $\A$-modules are defined over the narrow class field $H^+$ of $ K $. An explicit construction of $H^+$ is not immediate, so we defer its detailed description to Section \ref{Sec:Factor}. 
\section{Construction of Standard Drinfeld Modules}\label{Sec:PL}
In this section, we construct the standard model $ \Psi $ of rank-one Drinfeld $\A$-modules associated with the shtuka functions $f $ and study the corresponding Anderson $\A$-motive $M_\Psi$. We explicitly reformulate the residue expression of $\Psi$ using the exponential function $ \exp_{(0)}$ and the logarithm function $ \log_{(0)}$. After that, we compute the period lattice of $ \Psi^{(2)} $, i.e., the $\sigma^2$-twist of $ \Psi $. 
It is interesting to see that the period lattice of $ \Psi^{(2)} $ is generated by the constant $ \tilde{\pi} $, which is analogous to $\tilde{\pi}_\C$ arising from the Carlitz module $ \C $.
\subsection{Anderson $\A$-Motive}
We start with the definitions of $\tau$-module and Anderson motive.
\begin{defn}[$ \tau $-module]\label{def:taumodule}
     Let $R $ be a ring over $ \mathbb{F}_q $. Let $ \tau_R  : R \to R $ be an $ \mathbb{F}_q $-linear endomorphism of $ R $. 
 A (left) module $ M $ over $ R$ is called a $ \tau $-module over $ R $ if $M$ is equipped with a left $\tau_R$-linear endomorphism $ \tau : M \to M $; that is  
 \[
     \tau (r \cdot x ) =  \tau_R(r) \cdot \tau (x)
 \]
 for all $r\in R$, $x\in M$.
 \end{defn}
 
To give the precise definition of Anderson motive, we let $\iota:  \A \to L $ be the $ \A $-field with generic characteristic. 
 For a left $ \A \otimes_{\mathbb{F}_q} L\{\tau \}$-module $ M $, we denote by
\[
    \bar{M} =  M \otimes_{L} \bar{L}
\]
the left $ \A \otimes_{\mathbb{F}_q} \bar{L}\{\tau \}$-module associated with $ M $.
\begin{defn}\label{Def:motive}
    A (one-dimensional) Anderson $\A$-motive $M $ of rank $ r $ over $L $ is an $ \A \otimes_{\mathbb{F}_q} L\{\tau \}$-module such that 
    \begin{enumerate}
        \item $M$ is a free module of rank-one as an $ L\{\tau\}$-module.
        \item $M$ is a projective module of rank $r $ as an $ \A \otimes_{\mathbb{F}_q} L $-module.
        \item For each $a \in \A$, there exists some integer $ n > 0 $ such that 
        \[
            ( a - \iota(a) )^n \bar{M} \subseteq \tau \bar{M}. 
        \]
    \end{enumerate}
\end{defn}

View $M$ in Definition \ref{Def:motive} as a finitely generated module over $ \A \otimes_{\mathbb{F}_q} L $. Define the map $ \tau_R =(-)^{(1)} $ on $ \A \otimes_{\mathbb{F}_q} L $ as 
\[
      \tau_R (a \otimes l)= (a \otimes l)   ^{(1)} = a \otimes l^q . 
\]
The action of $\tau $ on $M$ is clearly $\tau_R$-linear by definition. Hence the Anderson motive $M$ is indeed a special kind of $\tau$-module over $ \A \otimes_{\mathbb F_q} L $ in the sense of Definition \ref{def:taumodule}.

It is well-known \cite{P23, vdHGJ04} that there exists an anti-equivalence $\mathfrak{F} $ between the category of Drinfeld $\A$-modules over $L$ and the category of Anderson $\A$-motives over $L$.
For a Drinfeld module $ \phi $, the associated Anderson motive $ M_{\phi} = \mathfrak{F}(\phi)$ is defined as follows. 
\begin{defn}[Anderson $\A$-motive associated with $\phi$] \label{Def:t-m}
    For a Drinfeld $\A$-module $ \phi $ over $L$, 
    the Anderson $\A$-motive associated with $\phi$ over $L$ is the twisted polynomial ring $ M_\phi \cong L \{\tau\} $ over $ L $ equipped with the left $ \A \otimes_{\mathbb{F}_q} L\{ \tau\} $-module structure:
    \[
    \big(a \otimes f(\tau)\big)*_\phi g(\tau)= f(\tau) \circ g(\tau) \circ \phi_a, \quad 
    \]
    for $ a\otimes f(\tau) \in \A\otimes_{\mathbb F_q}  L \{ \tau \}$ and $ g(\tau) \in {M}_\phi $. 
\end{defn}
Note that $ M_\phi$ is then an $ \A \otimes_{\mathbb{F}_q} L  $-module of rank $r$ with the generators $1$, $\tau$, $\tau^2$,  $\cdots$, $\tau^{r-1}$, where $ r $ is the rank of $ \phi $. We formally write $M_{\phi} = L \{\tau\} \cdot \mathbf{s}_{\phi}$ and call $ \mathbf{s}_{\phi} $ the formal generator of $ M_{\phi} $.

\begin{prop}\label{Pro:isomoti}
Suppose that $L$ is an $\A$-field with generic characteristic.  Let $\lambda : \Psi \to \Phi $ be an isogeny of Drinfeld $\A$-modules over $L$. Denote by $ M_{\Psi} $ and $M_\Phi$ the associated Anderson $\A$-motives, respectively. Let $ \mathbf{s}_{\Psi} $ be the formal generator of $ M_{\Psi} $ as $ L\{ \tau \}$-module. Then  $ M_{\Phi} $ is isomorphic to $ L \{ \tau \} \cdot \lambda \mathbf{s}_{\Psi} $ as an $ \A \otimes_{\mathbb F_q} L \{\tau\}$ submodule of $M_{\Psi}$. 

In particular, we can formally represent the generator $\mathbf{s}_{\Phi}$ of $M_\Phi$ in terms of  $\lambda \mathbf{s}_{\Psi}$.
  \end{prop} 
\begin{proof}
This statement is another way to restate Proposition 3.4.5 in \cite{P23}.    
Through the anti-equivalence $\mathfrak{F} $, the isogeny $\lambda: \, \Psi \to \Phi$ 
is equivalent to the 
$\A\otimes_{\mathbb F_q} L\{\tau\}$-homomorphism of motives 
\begin{align*}
\tilde \lambda:\quad M_\Phi &\to M_{\Psi},\\
   m \mathbf{s}_{\Phi}&\mapsto  (m \circ \lambda) \mathbf{s}_{\Psi}.
\end{align*}
Since $\tilde \lambda$ is injective, 
 $ M_{\Phi} $ is isomorphic to $ (L \{ \tau \} \circ \lambda) \mathbf{s}_{\Psi} $ as an $L \{ \tau \}$-submodule of $M_{\Psi}$, i.e., generated by $ \lambda \mathbf{s}_{\Psi}$. The  $\A\otimes_{\mathbb F_q} L\{\tau\}$-linear property of $ \tilde{\lambda} $ means that 
\[
  \left(a \otimes g(\tau) \right) *_{\Phi} \mathbf{s}_{\Phi} \mapsto a \otimes g(\tau) *_{\Psi} \lambda\mathbf{s}_{\Psi}.
\]
In conclusion, $ M_{\Phi} $ is isomorphic to $ L \{ \tau \} \cdot \lambda \mathbf{s}_{\Psi} $ as an $ \A \otimes_{\mathbb F_q} L \{\tau\}$ submodule of $M_{\Psi}$.
\end{proof}

\subsection{Shtuka Function}
\begin{notation}
     For $\alpha \in L$, we let $[\alpha]$  denote the zero of $ t - \alpha$ in the function field $ \K \otimes_{\mathbb{F}_q} L $. The notations $ [\theta] $, $[\eta]$, $ [\theta^{q^i}] $, $ [\eta^{q^i}] $ and similar expressions will be used subsequently.
\end{notation}
The shtuka function has been described in many references (for example \cite{Tha93}).
\begin{defn}[Shtuka Function]
   Let $X$ be a smooth curve with the infinite place $\infty $. Denote by $g$ the genus of $X$. A shtuka function $ F $ on $ X \otimes_{\mathbb{F}_q} L  $ is a meromorphic function with principal divisor 
    \[
        (F) = [\theta] -  \infty'  + V^{(1)} - V, 
    \]
    where $ \infty' $ is a place over $ \infty $ and $ V $ is an effective divisor of $X \otimes_{\mathbb{F}_q} L  $ with 
    \[
        \deg V = g ,\qquad \dim H^{0} (V) = 1 . 
    \]
\end{defn}
 
Now we return to the situation $ X = \mathbb{P}^1  $. The divisor $ V $ shall be the zero divisor, since the genus of $\mathbb{P}^1 $ is zero.
We choose $ L $ to be a field extension of $ K = \mathbb{F}_{q}(\theta)$ containing the residue field $ \mathbb{F}_{\rho} = \mathbb{F}_{q}(\eta) $.  As in the introduction,  we set 
\[
\A = H^{0}(\mathbb{P}^1 - P_{\rho}, \mathcal{O}_{\mathbb{P}^1}). 
\]
 By choosing $ \infty' =[\eta]$, the unique shtuka function $ f \in  \K \otimes_{\mathbb{F}_q} L  = L(t ) $ (up to scaling) is given explicitly by  
\begin{align}\label{Eq:ft}
    f(t) =\frac{1}{ t-\eta} -\frac{1}{\theta-\eta }. 
\end{align}
Some authors impose a sign-normalization condition on $f(t)$ and claim that the associated Drinfeld module is a Hayes module. This is a common misconception; we do not impose such a restriction here.
Applying the map $(-)^{(i)}$, we obtain the twist of $f$:  
\begin{equation}\label{Eq:f^i}
    f^{(i)}(t) =  \frac{1}{ t-\eta^{(i)}} -\frac{1}{\theta^{(i)}-\eta^{(i)} }.
\end{equation}

\subsection{From Shtuka Function to Drinfeld Module}
 Since the infinite place $P_{\rho}$ splits completely into the places $[\eta^{(0)}],\dots,[\eta^{(N-1)}]$ over $L$, we define the ring 
 \[ \A_L := H^{0} ( \mathbb{P}_L^1 - \{ [\eta^{(0)}],\cdots,[\eta^{(N-1)}]\} , \mathcal{O}_{X_L}), \]
 by base change. It is straightforward to see that  
\[
    \A_L \cong \A \otimes_{\mathbb{F}_q} L  \cong L \Big[ \frac{1}{ t-\eta^{(i)}}, ~\text{for}~ i=0,\cdots, N-1\Big]. 
\]
 We would like to define a $\tau$-module structure on $ \A_L $ via the shtuka function $f$.  Fix a countable topological basis $\{\mathbf{s}_i(t)\}_{i\ge0}$ of $\A_L$, where  
\[ \mathbf{s}_0(t) = 1 
\]
 and 
\begin{equation}\label{Eq:f_i}
\mathbf{s}_i(t)=f^{(0)}(t)f^{(1)}(t)\cdots f^{(i-1)}(t) .
\end{equation}
We now define the map $ \tau : \A_L \to \A_L $ by 
\[
    \tau (g) = g^{(1)} \cdot f(t)
\]
for $g \in \A_L
$. Therefore, $ \tau\mathbf{s}_i =\mathbf{s}_{i+1}$  and $\mathbf{s}_i = \tau^i\mathbf{s}_0 $. Furthermore, $  \A_L $ is a rank-one $L\{\tau\}$-module with the generator  $\mathbf{s}_0$. In this way, $\A_L$ can be rewritten as  
\begin{equation}\label{Eq:HLiso}
    \A_L = L\{ \tau \}\mathbf{s}_0 .
\end{equation}
Together with the canonical left $\A_L$-module structure on $\A_L$ (given by multiplication), this yields a $\tau$-module structure on $\A_L$. Concretely, we define an action $*$ of the ring $\A_L\{ \tau \} := \A \otimes_{\mathbb{F}_q} L\{\tau\}$ on $\A_L$ by  
\[
(a \otimes \lambda(\tau)) * m = \lambda(\tau)(a \cdot m),
\]  
for $a \in \A$ and $\lambda(\tau) \in L\{\tau\}$, where $\cdot$ denotes the product in $\A_L$.
Moreover,  $ \A_L $ becomes an Anderson $\A$-motive over $ L $. 
For $ m = g(\tau )\mathbf{s}_0 $, we set 
\[
     a * \mathbf{s}_0 = \sum_{i=0 }^{N(a)} \bi{a}{i}\mathbf{s}_i  = \sum_{i=0}^{N(a)} \bi{a}{i}\tau^i\mathbf{s}_0  ,
\]
for some coefficients $\bi{a}{i} \in L $ and integer $ N(a) \geqslant 1 $. 
\begin{defn}\label{Def:PL}
From the Anderson motive structure on $ \A_L $, we obtain the Drinfeld module 
\[
 \Psi: \A \to L\{ \tau \}\quad a \mapsto \Psi_{a},  
\]
where  $ \Psi_a $ denotes the twisted polynomial such that 
\[  
    a * \mathbf{s}_0 = \Psi_a \mathbf{s}_0; 
\]
namely,
\[
\Psi_a= \sum_{i=0}^{N(a)} \bi{a}{i}\tau^i  . 
\]
\end{defn}
We remark that by using the basic property of Drinfeld modules, it follows readily that  
 $N(a) = \deg(a) $. In addition, each coefficient $\bi{a}{i} $ is contained in $ K(\eta) $  by Theorem \ref{thm:Hilbert} or the explicit expression of $ \Psi $ later.

\subsection{The Twist of $ \Psi $}
\begin{notation}
   Let $\mathbb{Z}\{\sigma\}$ denote the group ring over $\mathbb{Z}$ consisting of elements
    \[
        \sum_{i = 0 }^{N-1}  a_i \sigma^i .
    \]
\end{notation}
We extend the $\mathbb{Z}$-power action to  
a $ \mathbb{Z} \{ \sigma \}  $-action on the field $ \mathbb{F}_q(\eta, t, \theta) $ by defining
\[
    \eta^{\sigma} = \eta^{(1)} ,~ t^\sigma =  t , ~\theta^\sigma= \theta. 
\]
\begin{notation}\label{No:sigmaDrinfeld}
Analogous to $f$, the function 
\[ f^{\sigma^i} = \frac{1}{ t-\eta^{(i)}} -\frac{1}{\theta-\eta^{(i)} }
\] 
is another shtuka function of $ \mathbb{P}^1 \otimes_{\mathbb{F}_q} L$, relative to the infinity $ [\eta^{(i)}] $. We denote by $\Psi^{(i)}$ the Drinfeld $\A$-module associated with $f^{\sigma^{i}}$.
\end{notation}
Following the construction of $ \Psi $ above, the resulting Drinfeld $\A$-module $\Psi^{(i)}$ associated with $f ^{\sigma^i}$ coincides with the $\sigma^i$-twist of $ \Psi $.

\subsection{Algebraic Structure of Anderson $\A$-Motive}
\begin{notation}\label{Not:bi}
Observe that $\frac{1}{t - \eta } \in \A \otimes_{\mathbb{F}_{q}} \mathbb{F}_{q^N}$ and $ \A $ is generated by $ T_i $. We define  $ b_i \in \mathbb{F}_{q^N} = \mathbb{F}_{q}(\eta) $ as the coefficients such that 
\begin{equation}\label{Eq:nosumbiTi}
 \sum_{i=0}^{N-1} b_i T_i = \frac{1}{t - \eta }.   
\end{equation}
\end{notation}
\begin{defn}[Annihilator]%
    \label{No:phiItphiI}
    Let $I $ be an ideal of $\A $. For a Drinfeld module $\psi $ over $L$, the set
    \begin{equation*}
    \{\psi _{a}~| ~a\in I\}
    \end{equation*}
    forms a left ideal of $L\{\tau \}$.  Since this ideal is principal, let $\psi_I$ denote its unique monic generator (the maximal right divisor). In this notation, we have 
  \begin{equation}\label{Eq:Annihilator}
\ker \psi_I = \bigcap_{a \in I} \ker \psi_a.
\end{equation} 
    The twisted polynomial $ \psi _{I}$ is called the annihilator of $ I $.
    \end{defn}
    
With these notations, we obtain the following useful result.
\begin{prop}\label{prop:ann}
Let $ I_\infty $ be the ideal of $\A$ generated by $ T_0, \cdots , T_{N-1} $, and $ I_0 $ the ideal generated by $ T_0-\frac{1}{ \rho(0)},~ T_1,~ T_2, \cdots , ~T_{N-1} $. Suppose that $ \mathbf{s}_0 $ is the formal generator of $ \Psi $. 
Then 
\begin{equation}\label{Eq:prs0ex}
    \frac{1}{t - \eta }*\mathbf{s}_0  = (\tau+\Theta)\mathbf{s}_0 \qquad \text{and } \qquad  \frac{t}{\eta(t-\eta)}*\mathbf{s}_0  = \left(\tau+\frac{\theta}{\eta}\Theta\right)\mathbf{s}_0.
\end{equation}
Moreover, the annihilators $\Psi_{I_\infty}$ and $\Psi_{I_0}$ are written as 
\begin{equation}\label{Eq:PsiIinfI0}
 \Psi_{I_\infty} = \tau + \Theta 
\qquad
\text{and } \qquad
\Psi_{I_0} = \tau +  \frac{\theta}{\eta} \Theta .  
\end{equation}
\end{prop}

\begin{proof}
Let $b_i$ be as in Notation \ref{Not:bi}.
By the definition of the Anderson $\A$-motive in Definition \ref{Def:t-m}, we have 
\begin{align*}
 (\sum_{i=0}^{N-1} b_i \Psi_{T_i})\mathbf{s}_0  
 &= ( \sum_{i=0}^{N-1} b_i T_i) *\mathbf{s}_0 \\
 &= \frac{1}{t - \eta }*\mathbf{s}_0 \qquad \text{by Equation \eqref{Eq:nosumbiTi}} \\
 &= f*\mathbf{s}_0 + \Theta*\mathbf{s}_0 \\
 &= \tau(\mathbf{s}_0)  + \Theta\mathbf{s}_0.
\end{align*}
Thus we obtain
\[
\frac{1}{t - \eta }*\mathbf{s}_0 = (\tau+\Theta)\mathbf{s}_0 
\] 
and 
\begin{equation}\label{Eq:sumbiPTi}
\sum_{i=0}^{N-1} b_i \Psi_{T_i} =\tau + \Theta. 
\end{equation}
Since $ I_\infty $ is the ideal generated by $ T_0, \cdots , T_{N-1} $, 
the common right-divisor of $ \{ \Psi_{T_i} \}$ is 
\[ \Psi_{I_\infty} =\tau + \Theta =  \tau+ \frac{1}{\theta - \eta }. \]
Together with $ \frac{-\rho(0)}{b_0} = \eta $, we have  
\[
 b_0 \Psi_{T_0-\frac{1}{\rho(0)}} + b_1 \Psi_{T_1} + \cdots + b_{N-1} \Psi_{T_{N-1}}  = \tau + \Theta - \frac{b_0}{\rho(0)} = \tau + \Theta + \frac{1}{\eta}.
\]
Given that $ I_0 $ is the ideal generated by $ T_0-\frac{1}{\rho(0)}, T_1,T_2, \cdots , T_{N-1} $, this implies 
\[
   \Psi_{I_0} =  \tau+ \frac{1}{\theta - \eta }+\frac{1}{\eta} = \tau + \frac{\theta}{\eta} \Theta. 
\]
Similarly, we obtain from Equation \eqref{Eq:ft} that 
\begin{align*}
  \frac{t}{\eta(t-\eta)} *\mathbf{s}_0  & = \left(  
    f  +  \frac{\theta}{\eta}\Theta  
  \right) *\mathbf{s}_0 \\
 &= f*\mathbf{s}_0 +  \frac{\theta}{\eta}\Theta *\mathbf{s}_0 \\
 &= \tau(\mathbf{s}_0)  + \frac{\theta}{\eta}\Theta\mathbf{s}_0.
\end{align*}
This completes the proof.
\end{proof}
We remark that the proof above presents explicit computations for Anderson motives and annihilators, which constitute the main techniques developed in this paper.

\subsection{Exponential Functions}
 We investigate the exponential functions associated with $ \Psi^{(i)} $ in this section. 
\begin{defn}[Exponential Function] 
\label{Defn:exponential}
Recall that $ K_{\rho} $ is the completion of $ K = \mathbb{F}_{q}(\theta)$ with respect to the $\eta$-norm $ |\cdot|_{\eta} $.  
Let $\mathbb{C}_{\infty} $ denote the completion of an algebraic closure of $ K_{\rho} $.
The exponential function
\[\exp_{(0)} : \mathbb{C}_{\infty} \to \mathbb{C}_{\infty}\]
is an $\mathbb F_{q}$-linear function satisfying
 
\begin{equation}\label{Eq:exponential}
\Psi _{a}(\exp _{(0)}(\xi) )= \exp _{(0)} (\iota(a)\xi) \quad \text{for all $a\in \A$},
\end{equation}
 with derivative $1$ at $ \xi = 0$.
\end{defn}

Anderson and Thakur \cite{Tha93} established the explicit formula for such exponential functions.
\begin{thm}\label{Thm:An-Th-exp}
  Let $\Psi$ be the rank-one Drinfeld module derived from the shtuka function $f$. Then the exponential function of $\Psi$ is given by
\[
\exp(\xi)=\sum_{i=0}^{\infty}\frac{\xi^{q^i}}{\left(f^{(0)}(t)f^{(1)}(t)\cdots f^{(i-1)}(t)\right)|_{t = \theta^{(i)}}}.
\]
\end{thm}

Applying Theorem \ref{Thm:An-Th-exp}, the exponential function $\exp_{(0)}$ takes the form
\begin{align}\label{Eq:Dexp}
      \exp_{(0)}(\xi) = \sum_{i = 0}^{\infty} \frac{\xi^{q^i}}{D_i},
\end{align}
where $ D_0 = 1 $ and $ D_i =\mathbf{s}_i(\theta^{(i)}) $. 
From Equation \eqref{Eq:f_i}, a direct computation yields 
\begin{align}\label{Eq:f_iTh}
  \mathbf{s}_i(t)=\frac{(\theta^{(0)}-t)(\theta^{(1)}-t)\cdots(\theta^{(i-1)}-t)}{\prod_{k=0}^{i-1}(t-\eta^{(k)})}\cdot\Theta^{W_i},
\end{align} 
where $W_i$ is as defined in \eqref{Eq:Wn}.
Let $ \floor{\frac{i}{N}} $ denote the floor of $\frac{i}{N}$.
For $i\geqslant 0$,
we obtain  
\begin{align}\label{Eq:D_i}
\nonumber D_i={}&\mathbf{s}_i(\theta^{(i)})=\frac{(\theta-\theta^{(1)})^{(i-1)}(\theta-\theta^{(2)})^{(i-2)}\cdots(\theta-\theta^{(i)})}{(\theta-\eta)^{\floor{\frac{i}{N}}  q^i+W_i}\prod_{k=1}^{N-1}(\theta-\eta^{(k)})^{\floor{\frac{i+k}{N}} q^i}}\\
={}&\Theta^{W_i + q^i \sum_{k=0}^{N-1} \floor{\frac{i+k}{N}}  \sigma^k }{(\theta^{(i-1)}-\theta^{(i)}) (\theta^{(i-2)}-\theta^{(i)}) \cdots(\theta-\theta^{(i)})}. 
\end{align}


\begin{cor}\label{Cor:Dn}
   Let $D_0=1$. For $ i \geqslant 1 $, the coefficients $ D_i \in K(\eta)$ satisfy the recursive formula:
    \begin{equation}\label{Eq:Dn}
    D_i = D_{i-1}^q \frac{\theta-\theta^{(i)}}{(\theta^{(i)}-\eta)(\theta-\eta)}.
	\end{equation}
\end{cor}

\begin{proof}
 A direct verification shows
\begin{align}
 \floor{\frac{i+k}{N}} -\floor{\frac{i+k-1}{N}}  =
    \begin{cases}
       1, \qquad & N | (k+i),\\
      0, &{\rm otherwise},
    \end{cases}
\end{align}
and 
 \[ W_i-qW_{i-1}=1. \]
  Together with Equation \eqref{Eq:D_i}, we find that 
\begin{align*}
  \frac{D_i}{D_{i-1}^q}=\Theta^{1+q^i\sigma^{-i}} (\theta-\theta^{(i)}).
\end{align*}

\end{proof}

\begin{notation}
We denote by  $\exp_{(j)}$  the exponential function associated with the Drinfeld module $\Psi^{(j)}$ (see Notation \ref{No:sigmaDrinfeld}).  
\end{notation}
As a consequence, the exponential function $ \exp_{(j)} $ for $ \Psi^{(j)} $ is written as
\begin{equation}\label{Equ:expj}
    \exp_{(j)}(\xi) = \sum_{i=0}^{\infty} \frac{\xi^{q^i}}{D_i^{\sigma^j}}, 
\end{equation}
where $ D_i^{\sigma^j} $ is the $\sigma^j$-twist of $D_i$.
According to the expression \eqref{Eq:D_i}, one can derive that
  \begin{align}\label{Eq:D_nsigj}
          D_n^{\sigma^j} = D_n  \Big( (\theta - \eta^{(0)})(\theta - \eta^{(1)}) \cdots (\theta - \eta^{(j-1)}) \Big) ^{q^n\sigma^{-n} - q^n}\left(\Theta^{\sigma^j-1}\right)^{W_n}.
    \end{align}

\subsection{Logarithm Functions}
We recall the logarithm
function associated with $ \Psi $  below.
\begin{defn}[Logarithm Function]
\label{Def:Logarithm}
The logarithm function $\log _{(0)}(\xi)$ of $ \Psi $ is the local inverse of the
exponential function $ \exp_{(0)}(\xi)$. This means that
\[
\exp_{(0)} \log_{(0)} (\xi) = \xi = \log_{(0)} \exp_{(0)} (\xi) .
\]
By Equation \eqref{Eq:exponential}, the logarithm function $\log _{(0)}(\xi)$  satisfies
%
\begin{equation}\label{Eq:delog}
\log _{(0)} (\Psi _{a} (\xi) ) = \iota(a)\log _{(0)}(\xi),
\end{equation}
for all $a \in \A$. 
\end{defn}
We may assume that  
\begin{equation}\label{Eq:asslog}
\log_{(0)}(\xi) =\sum_{j= 0}^{\infty} \frac{(-1)^j\xi^{q^j}}{L_j},
\end{equation}
 where $L_0 = 1$, since the derivative of $\log_{(0)}(\xi)$ at the origin is $1$. 
The relation $  \exp_{(0)} \log_{(0)} (\xi) = \xi $ implies that 
\[
\sum_{i=0}^{\infty}\frac{1}{D_i}\Big(\sum_{j=0}^{\infty}\frac{(-1)^j\xi^{q^j}}{L_j}\Big)^{q^i}=\sum_{n=0}^{\infty}\left(\sum_{j=0}^{n}\frac{(-1)^{j}}{D_{n-j} L_{j}^{q^{n-j}}}\right)\xi^{q^n}=\xi.
\]
This is equivalent to the identity 
\begin{align}\label{Eq:DjLn-j}
    \sum_{j=0}^n  \frac{(-1)^{n-j}}{D_{j} L_{n-j}^{q^{j}}}=0,
\end{align}
for each $ n \geqslant 1 $.
It follows that the coefficient $L_j$
  can indeed be derived directly from $D_j$, though the computation is nontrivial. A more practical expression for $L_j$
  uses the residue formula, as presented in the subsequent theorem.
\begin{thm}\label{Thm:An-Th-log}
 Let $ X / \mathbb{F}_q $ be a smooth curve with an infinity $\infty$ of degree $1$.
 Let $\Psi$ be a Drinfeld  $\A$-module associated with a shtuka function $f(t)$ on $X \otimes_{\mathbb{F}_q} L $.  Denote by $\Omega_K$ the set of differential forms over $K$. Assume that $\omega\in \Omega_K $ is the unique differential
 form such that the valuation of $  \omega $ at $ \infty $ equals $ -2 $ and  $\Res _{t=\theta}\frac{\omega^{(1)}}{f^{(0)}(t)} = 1$. Then
\[
\log_{\Psi}(\xi)=\sum_{i=0}^{\infty}\Res_{t=\theta}\frac{\omega^{(i+1)}}{f^{(0)}(t)f^{(1)}(t)\cdots f^{(i)}(t)}\xi^{q^i}.
\]
\end{thm}
However, the differential $\omega$ from Theorem \ref{Thm:An-Th-log} is not valid when $ \deg(\infty) \geqslant 2 $. We therefore modify the choice of $\omega$. For this purpose, we define the differentials: $ \omega^{(0)} = h^{(0)}(t) dt $, with 
\[
     h^{(0)}(t) = -  \frac{(\theta^{(-1)} - \eta^{(-2)} )} { (\theta^{(-1)} - \eta^{(-1)})} \frac{1}{(t - \eta^{(-2)})(t - \eta^{(-1)})} .
\]
Comparing with the conditions in Theorem \ref{Thm:An-Th-log}, the differential $ \omega $ is now defined over $ \mathbb{F}_{q}(\eta) $ (not $\mathbb{F}_{q} $) associated with two poles lying over $ P_{\rho}$. 
Applying the map $(\cdot)^{(j)}$, we have 
\begin{align}\label{Eq:omegaj}
     \omega^{(j)} =h^{(j)}(t)dt ,
\end{align}
and
\begin{align}\label{Eq:h^j}
h^{(j)}(t)=-  \frac{\theta^{(j-1)} - \eta^{(j-2)} } { \theta^{(j-1)} - \eta^{(j-1)}}  \frac{1}{(t-\eta^{(j-2)})(t - \eta^{(j-1)} )}.     
\end{align}
Observe that the simple poles of $ \omega^{(j)} $ are $[\eta^{(j-1)}]$ and $ [\eta^{(j-2)}] $. 

We now state a result analogous to Theorem \ref{Thm:An-Th-log}.
\begin{thm}\label{Thm:L}
      For $j\geqslant 0$, the coefficients $ L_j $ in the logarithm function $\log_{(0)}(\xi)$ satisfy
    \begin{equation}\label{equ:ResLj}
          \frac{(-1)^j}{L_j} = \Res_{t = \theta } \frac{\omega^{(j+1)}}{\mathbf{s}_{j+1}(t)} = \Res_{t = \theta } \frac{\omega^{(j+1)}}{f^{(0)}(t) \cdots f^{(j)}(t) }.
    \end{equation}
\end{thm}
\begin{proof}
    Notice that 
    \[
    \omega^{(1)} = -  \frac{\theta - \eta^{(-1)} } { \theta - \eta} \frac{1}{(t - \eta^{(-1)})(t - \eta^{(0)})} dt .  
    \]
    This implies
    \[
\Res_{t = \theta }  \frac{\omega^{(1)}}{f^{(0)}(t)} = 1  = \frac{1}{L_0}. 
\]
This is consistent with the condition $ L_0 = 1 $.  For the induction step, assume that Equation \eqref{equ:ResLj} holds for $ L_0, L_1, \cdots, L_{j-1}$. We consider the differential 
\[
\frac{\omega^{(j+1)}}{ \mathbf{s}_{j+1}(t)} =\frac{\omega^{(j+1)}}{f^{(0)}(t) f^{(1)}(t)\cdots f^{(j)}(t) }  .\]
The pole of $ \frac{\omega^{(j+1)}}{ \mathbf{s}_{j+1}(t)}  $ is located at $ \theta^{(k)} $ for $ k = 0 ,\cdots, j $.
From the residue theorem in function fields, it follows that
\begin{align*}
    0& = \sum_{k=0}^{j}\Res_{t = \theta^{(k)}} \frac{\omega^{(j+1)}}{f^{(0)}(t) f^{(1)}(t)\cdots f^{(j)}(t) } \\
    & = \sum_{k=0}^{j} \frac{1}{D_k}\Res_{t = \theta^{(k)} } \frac{\omega^{(j+1)}}{ f^{(k)}(t)\cdots f^{(j)}(t) }\\
     & = \sum_{k=0}^{j} \frac{1}{D_k} \left( \Res_{t = \theta } \frac{\omega^{(j-k+1)}}{ f^{(0)}(t)\cdots f^{(j-k)}(t) } \right)^{q^k} \\
      & =     \Res_{t = \theta } \frac{\omega^{(j +1)}}{ f^{(0)}(t)\cdots f^{(j )}(t) }   +   \sum_{k=1}^{j} \frac{(-1)^{j-k}}{D_k L_{j-k}^{q^k}} .
\end{align*}
From \eqref{Eq:DjLn-j}, we have 
 \[
 \Res_{t = \theta } \frac{\omega^{(j +1)}}{ f^{(0)}(t)\cdots f^{(j )}(t) } = \frac{(-1)^j}{L_j}.
\]
\end{proof}

\begin{cor}\label{Cor:Ljn}
\begin{enumerate}
    \item  Using the notation above, the coefficient $L_j$ can be expressed as  
\begin{align}\label{Eq:L-expn}
	L_j =(\theta-\theta^{(1)}) (\theta-\theta^{(2)})  \cdots (\theta -\theta^{(j)})\Theta^{W_j+q^j\sigma^{N-1}+\sum_{k=0}^{N-1} (\floor{\frac{j+N-k-2}{N}}\sigma^{k})}.
\end{align}	
 \item For $ j \geqslant 1 $, the coefficient $ L_j $ satisfies the recurrence relation:
	\[
		L_j = 	L_{j-1} \frac{\theta-\theta^{(j)}}{ (\theta - \eta^{(j-2)})(\theta-\eta)^{q^{j-1}} (\theta - \eta^{(-1)})^{{q^{j}}-q^{j-1}}}. 
		\]
\end{enumerate} 
\end{cor}
\begin{proof}
(1) It follows from Theorem \ref{Thm:L} that
\begin{equation}\label{Equ:Lj}
\frac{(-1)^{j}}{L_j } =\Res_{t = \theta  } \frac{ \omega^{(j+1)} }{ f^{(0)}(t) f^{(1)}(t) \cdots f^{(j)}(t)}.
\end{equation}	
Note that the pole of the term
 \[ \frac{ \omega^{(j+1)} }{ f^{(0)}(t) f^{(1)}(t) \cdots f^{(j)}(t)}  
 \]
 is located at $ \theta^{(k)} $ for $ k = 0 ,\cdots, j $, and all the poles are simple. 
 So Equation \eqref{Equ:Lj} reduces to 
 \begin{align*}
\frac{(-1)^{j}}{L_j }
 &=\frac{h^{(j+1)}(\theta)\prod_{k=0}^{N-1}(\theta-\eta^{(k)})^{ \floor{\frac{j+N-k}{N}} }}{(\theta-\theta^{(1)}) (\theta-\theta^{(2)})  \cdots (\theta -\theta^{(j)})}\cdot\frac{1}{\Theta^{W_{j+1}}}.
\end{align*}	
 Combining Equation \eqref{Eq:h^j}, we find that
\begin{align*}
L_j =(\theta-\theta^{(1)}) (\theta-\theta^{(2)})  \cdots (\theta -\theta^{(j)})\Theta^{W_j+q^j\sigma^{N-1}+\sum_{k=0}^{N-1} (\floor{\frac{j+N-k-2}{N}}\sigma^{k})}.
\end{align*}	

    (2) Applying the equalities 
\begin{align}
 \floor{\frac{j+N-k-2}{N}}-\floor{\frac{j+N-k-3}{N}}=
    \begin{cases}
       1,  & k=\ell N+j-2, \ell~ \text{is an integer},\\
      0, & \text{otherwise}
    \end{cases}
\end{align}
and 
\[ W_j-W_{j-1}=q^{j-1}, \]
Equation \eqref{Eq:L-expn} yields that
  \begin{align}
      \frac{L_j}{L_{j-1}}=(\theta-\theta^{(j)})\Theta^{q^{j-1}+(q^j-q^{j-1})\sigma^{-1}+\sigma^{j-2}}.
  \end{align}      
Therefore, the lemma follows.
\end{proof}

\begin{notation}\label{No:Ljn}
	For integer $ k \geqslant 0 $, we adopt the symbols
	\[
		\bket{k}	:=   \Theta^{q^{k}}  - \Theta^{\sigma^k}=-f^{(k)}(\theta).
	\] 
\end{notation} 
A direct computation shows that
\begin{align}\label{Eq:ttq2k}
	\theta - \theta^{q^{k}}
 = \frac{ \Theta^{q^{k}} -\Theta^{\sigma^k} }{\Theta^{q^{k}+\sigma^k }}=\frac{\bket{k}} {\Theta^{q^{k}+\sigma^k  }}.
\end{align}
Using this notation, we obtain the following lemma.
\begin{lem}\label{Lem:leLj}
	For $j \geqslant 1$, $ L_j $  can be rewritten as
 \[
	L_j = \bket{1}\bket{2}\cdots \bket{j} \Theta^{q^j(\sigma^{N-1}-1)-\sigma^{j}-\sigma^{j-1}+2}.
\] 
An analogous computation gives
\[ L_j^{\sigma^2} = \bket{1} \bket{2}\cdots \bket{j}\Theta^{(q^j-1)(\sigma-1)+W_j(\sigma^2-1)}.\]

\end{lem}
 \begin{proof}
 Substituting \eqref{Eq:ttq2k}  into the expression \eqref{Eq:L-expn} of $L_{j}$, we have
 \begin{equation*}
 \begin{split}
     L_{j}&=
         \frac{\bket{1}}{\Theta^{q+\sigma}} \frac{\bket{2}}{\Theta^{q^2+\sigma^2}} \cdots \frac{\bket{j}}{\Theta^{q^{j}+\sigma^j}}\Theta^{W_j+q^j\sigma^{N-1}+\sum_{k=0}^{N-1} \floor{\frac{j+N-k-2}{N}}\sigma^{k}}\\
     &= \bket{1} \bket{2} \cdots \bket{j} \Theta^{J},
 \end{split}
 \end{equation*}
 where $J$ is given by  
 \begin{equation*}
 \begin{split}
    J={}&W_j+q^j\sigma^{N-1}-(q+\sigma)-\cdots-(q^j+\sigma^j)+\sum_{k=0}^{N-1} \floor{\frac{j+N-k-2}{N}}\sigma^{k}\\
        ={}& W_j+q^j\sigma^{N-1}-W_{j+1}+2-\sum_{k=0}^{N-1} \floor{\frac{j+N-k}{N}} \sigma^{k}+\sum_{k=0}^{N-1} \floor{\frac{j+N-k-2}{N}} \sigma^{k}\\
        ={}& q^j(\sigma^{N-1}-1)-\sigma^{j}-\sigma^{j-1}+2.
 \end{split}
 \end{equation*}
 This establishes the first identity. The second identity follows similarly.
\end{proof}

\subsection{Residue Representation}
In this section, we present a residue representation for the Drinfeld module $\Psi$.  
For $ z(t) \in \A $, using the equality 
\begin{equation}\label{Eq:expression}
\exp_{(0)} \left(z(\theta) \log_{(0)}(\xi)\right) = \Psi_{z(t)} (\xi) =  \sum_{j=0}^{\deg z(t)} \bi{z(t)}{j} \xi^{q^j},
\end{equation}
we obtain  
\begin{equation}\label{Eq:bizr}
\bi{z}{r} = \sum_{j=0}^r (-1)^{r-j} \frac{z^{q^j}}{D_j L_{r-j}^{q^{j}}}.
\end{equation}
It remains to give the coefficients $ \bi{z }{r } $ via a residue formula. 
\begin{notation}
    For a differential  $\omega_* \in \Omega_K \otimes_{\mathbb{F}_q} L $, we define the residue of $\omega_*$ at the place $ P_{\rho } $ as
\[
\Res_{P_\rho} \omega_*  = \sum_{i = 0 }^{N-1} \Res_{t = \eta^{(i)}} \omega_*  .
\] 
 Note that this formula coincides with the canonical residue notion when $ \omega_* \in \Omega_K $.
\end{notation}
\begin{notation}\label{No:gammak}
    For $k \geqslant 1$, we let 
    \[
    \gamma_k(\sigma):=\sum_{i=0}^{k-1} q^i \sigma^{k-1-i} \in \mathbb{Z}[\sigma].
    \]
For convenience, we set $ \gamma_0(\sigma ) = 0 $. 
\end{notation}
\begin{thm}\label{thm:bi}
With the notations above,  the coefficients of $ \Psi_{z(t)} $ are given by  
    \[
 \bi{z(t)}{k} = - \Res_{P_\rho} \frac{z(t)\cdot \omega^{(k+1)}}{f^{(0)}(t) \cdots f^{(k)}(t) },
 \]
 for $k=0,1, \cdots, \deg (z(t) )$. 
\end{thm}
\begin{proof}
 For $ z(t) \in \A $ and $ k \leqslant \deg(z(t) )$, the residue theorem yields
\[
    \sum_{i=0}^k\Res_{t = \theta ^{(i)} } \frac{z(t) \omega^{(k+1)} }{ f^{(0)}(t) f^{(1)} (t)\cdots f^{(k)}(t)} + \Res_{ P_{\rho}} \frac{z(t) \omega^{(k+1)} }{ f^{(0)}(t) f^{(1)}(t) \cdots f^{(k)}(t)}  = 0 .
\]
The first term simplifies as follows:
\begin{align*}
     & \sum_{i=0}^k\Res_{t = \theta ^{(i)} } \frac{z(t) \omega^{(k+1)} }{ f^{(0)}(t) f^{(1)} (t)\cdots f^{(k)}(t)} \\
     = &\sum_{i=0}^k\frac{1}{f^{(0)}(t) \cdots f^{(i-1)}(t)} |_{t = \theta^{(i)} } \Res_{t = \theta ^{(i)} } \frac{z(t) \omega^{(k+1)} }{ f^{(i)}(t) f^{(i+1)}(t) \cdots f^{(k)}(t)} \\
     ={}&  \sum_{i=0}^k\frac{1}{f^{(0)}(t) \cdots f^{(i-1)}(t)} |_{t = \theta^{(i)} } \left( \Res_{t = \theta } \frac{z(t) \omega^{(k-i+1)} }{ f^{(0)} (t)f^{(1)} (t)\cdots f^{(k-i)}(t)} \right)^{(i)} \\
     ={}& \sum_{i=0}^k \frac{1}{D_i} \frac{(-1)^{k-i}z(\theta)^{q^i}}{L_{k-i}^{q^i}} \qquad \text{by Theorem \ref{Thm:L}} \\
     ={}& \bi{z(\theta)}{k}\qquad \text{by Equation \eqref{Eq:bizr}} .
\end{align*} 
Thus, the equality in the theorem holds.
\end{proof}
For example, when $k = N $, we have 
\begin{equation}
    \bi{T_j}{N}  
    =-\sum_{k=0}^{N-1}\Res_{t = \eta ^{(k)} } \frac{  T_j(t) \omega^{(N+1)} }{ f^{(0)}(t) f^{(1)}(t) \cdots f^{(N)}(t)}. 
\end{equation}
Note that $[\eta], [\eta^{(2)}], \cdots, [\eta^{(N-2)}]$ are not poles of  $\frac{  T_j(t) \omega^{(N+1)} }{ f^{(0)}(t) f^{(1)}(t) \cdots f^{(N)}(t)}$, and $[\eta^{(N-1)}]$ is a simple pole of $\frac{  T_j(t) \omega^{(N+1)} }{ f^{(0)}(t) f^{(1)}(t) \cdots f^{(N)}(t)}$. It follows that 
\begin{align}\label{Eq:TjN}
 \nonumber   \bi{T_j}{N}  
   &=-\Res_{t = \eta ^{(N-1)} } \frac{  T_j(t) \omega^{(N+1)} }{ f^{(0)}(t) f^{(1)}(t) \cdots f^{(N)}(t)} \\ \nonumber
   & = (\eta^{(-1)})^j \frac {(\theta- \eta)^{W_N}} { (\theta - \eta^{(N-1)})  \cdots  (\theta - \eta^{(1)})^{ q^{N-2}} (\theta - \eta)^{q^{N-1} }}\\
     & = (\eta^{(-1)})^j  \frac { (\theta - \eta)^{W_{N-1} }} { (\theta - \eta^{(N-1)})  \cdots  (\theta - \eta^{(1)})^{ q^{N-2}}} \\\nonumber
     & =(\eta^{(-1)})^j  \Theta^{\gamma_N(\sigma) -W_N  }.
\end{align}

 Applying Theorem \ref{thm:bi} to Equation \eqref{Eq:expression}, we obtain the main result of this section: 
\begin{thm}
    The rank-one Drinfeld module $\Psi$ associated with the shtuka function $ f $ is represented as 
    \[
        \Psi_{z(t)} = - \sum_{k=0}^{\deg (z(t) )}\Res_{P_{\rho}} \frac{z(t) \omega^{(k+1)}}{f^{(0)} \cdots f^{(k)}} \tau^{k}     \]
    for $ z(t) \in \A$.
\end{thm}
From Equation \eqref{Eq:TjN}, we have 
\[
    \bi{T_1}{N} / \bi{T_0}{N}  = \eta^{(-1)}. 
\]
Consequently, $ \Psi $ is of $\eta^{(-1)}$-type and the coefficients of $ \Psi_{z(t)}$ are defined over $ H = \mathbb{F}_{q}(\theta, \eta) $.
\subsection{Carlitz Period}
Let $E_{k} (z)$ be the polynomial 
\[
E_k (z) =  z \prod_{ 0 \not = a \in A, \,\deg(a) \leqslant Nk} (1 - a^{-1} z ) 
\]
with derivative $1$ at the origin. Since all the elements in $A$ of degree $\leqslant  Nk$ are precisely the zeros of $\bi{z}{ N k +1} $, we have
\[
E_k (z) = - L_{Nk + 1 } \bi{z}{ N k +1} = \sum_{j =0 }^{ N k +1 } (-1)^ j \frac{L_{Nk +1} }{ D_j L_{Nk +1 - j }^{q^j}} z^{q^j}.
\]
\begin{notation}
    We define the series 
\[
E_\infty (z) = \lim_{k \to \infty} E_k(z) = \lim_{k \to \infty}\sum_{j =0 }^{ k N  +1 } (-1)^ j  \frac{L_{k N +1} }{ D_j L_{k N +1 - j }^{q^j}} z^{q^j}.
\]
\end{notation}
Since $E_{\infty}$ is invariant under $\sigma$, this implies 
\begin{equation}\label{Eq:Einf}
 E_\infty (z) = \lim_{k \to \infty}\sum_{j =0 }^{ k N  +1 } (-1)^ j  \frac{ L_{k N +1}^{\sigma^2} }{  D_j^{\sigma^2}  L_{k N +1 - j }^{q^j\sigma^2}} z^{q^j}.   
\end{equation}
It is straightforward to see that $ E_{\infty} $ converges in $\mathbb{C}_{\infty}$ and the kernel of $ E_{\infty} $ is  $ \A $. 
We next show that $E_{\infty}$ is equivalent to $\exp_{(2)}$, i.e. 
\[
E_{\infty}(z)=\tilde{\pi}^{-1} \exp_{(2)}( \tilde{\pi} z), 
\]
where $\tilde{\pi}$ is called the Carlitz period of $\Psi^{(2)}$.

\begin{notation}
	Using the notation $ \bket{k} $ defined in Notation \ref{No:Ljn}. Define the partial products as 
	\[
	\Gamma_{d,k}:= \prod_{j=1}^{d} \left(1 - \frac{ \bket{ N(j-1)+k} }{\bket{ Nj+k} } \right), 
	\]
  for $k=0,\cdots, N-1$,  and $d\geqslant 1$.
\end{notation}

\begin{lem}\label{Lem:ald}
	For each $k=0,\cdots, N-1$,  $ \Gamma_{d,k} $ can be expressed as 
	\[
	\Gamma_{d,k} = \frac{\bket{N}^{\frac{q^k(q^{ N d} - 1)}{q^N-1}}} {\prod_{j=1}^d \bket{N j+k}}. 
	\] 
 Moreover,  $ \Gamma_{d,k} $ converges when $d$ tends to infinity. 
\end{lem}
\begin{proof}
 Recall that $ v_{\eta} $ denotes the valuation of $\mathbb C_{\infty}$ with respect to the place at $ [\eta]$.
From Notation \ref{No:Ljn}, we directly obtain  
\[
v_{\eta}(\bket{k})=-q^k. 
\]
Thus, we have
\[
v_{\eta}\Big(\frac{ \bket{N(j-1)+k} }{ \bket{Nj+k} }\Big)=q^{Nj+k}-q^{N(j-1)+k}.
\] 
This ensures convergence of $ \Gamma_{d,k} $ when $ d \to \infty $.

Notice that 
	\[ \bket{Nj+k} - \bket{N(j-1)+k} = (\Theta^{q^N} - \Theta)^{q^{N(j-1)+k}} = \bket{N}^{q^{N(j-1)+k}}.
	\]
Thus, $\Gamma_{d,k}$ simplifies to  
	\[
	\Gamma_{d,k}=\frac{\bket{N}^{q^k}}{\bket{N+k}}\frac{\bket{N}^{q^{N+k}}}{\bket{2N+k}} \cdots \frac{\bket{N}^{q^{N(d-1)+k}}}{\bket{Nd+k}} = \frac{\bket{N}^{\frac{q^k(q^{ N d} - 1)}{q^N-1}}} {\prod_{j=1}^d \bket{N j+k}}.
	\] 
\end{proof}
Due to Lemma \ref{Lem:ald}, we denote by $ \Gamma_{\infty,k} $ the limit of $ \Gamma_{d,k} $ when $d $ tends to infinity.
The following lemma estimates the product $ \prod_{k=0}^{N-1} \Gamma_{\infty,k} $.
\begin{lem} \label{lem:forpp}
With the above notations, for  $ d \geqslant 0$,  the following equality holds: 
\begin{align}\label{Eq:leGap}
   \prod_{k=0}^{N-1}\Gamma_{\infty,k}=\bket{1}\cdots \bket{N-1}\Theta^{\frac{q-q^N}{q-1}}(1-\Theta^{1-q^N})^{\frac{-1}{q-1}}\prod_{k=1}^{\infty}\frac{\Theta^{q^k}}{\Theta^{q^k} -\Theta^{\sigma^k}}.
\end{align} 
\end{lem}
\begin{proof}
It suffices to show that 
 \begin{align}\label{Eq:leGad}
   \prod_{k=0}^{N-1}\Gamma_{d,k}=\bket{1}\cdots \bket{N-1}\Theta^{\frac{q-q^N}{q-1}}(1-\Theta^{1-q^N})^{\frac{q^{Nd}-1}{q-1}}\prod_{k=1}^{N(d+1)-1}\frac{\Theta^{q^k}}{\Theta^{q^k} -\Theta^{\sigma^k}}.
\end{align} 
Using the notation $\Gamma_{d,k}$ and Lemma \ref{Lem:ald}, we can rewrite 
     \begin{align}\label{Eq:Gap}
\prod_{k=0}^{N-1}\Gamma_{d,k}= \bket{1}\cdots\bket{N-1}\cdot\frac{\bket{N}^{\frac{q^{Nd}-1}{q-1}}}{\prod_{k=1}^{N(d+1)-1} \bket{k} }.
    \end{align}
Notice that 
$ \bket{k}=\Theta^{q^k}-\Theta^{\sigma^k}$, and in particular, 
\begin{align}\label{Eq:[N]}
    \bket{N}=\Theta^{q^N}(1-\Theta^{1-q^N}).
\end{align}
 Thus, Equation \eqref{Eq:Gap} implies Equation \eqref{Eq:leGad}.
\end{proof}

\begin{lem}\label{Lem:LdLdj}
      For positive integers $ j $ and $d$ satisfying $Nd\geqslant j-1$,  we choose some non-negative integers $i \leqslant N-1$ and $ l \leqslant d $ such that 
        \[ Nd+1- j=N(d-l)+i . \]
Let $\zeta_i(\sigma)=(q^i-1)(\sigma-1)+W_i(\sigma^2-1)$ for $i\geqslant 0$. We then obtain
\begin{align}\label{Eq:period}
  \frac{L_{Nd + 1}^{\sigma^2}}{ L_{Nd+1- j} ^{q^j \sigma^2}}={}&\Big(\bket{1}\cdots \bket{N-1}\Big)^{1-q^j} \bket{N}^{W_j} \cdot \Theta^{\zeta_j(\sigma)}\cdot\frac{\Big(\prod_{k=0}^{i}\Gamma_{d-l,k}\prod_{k=i+1}^{N-1}\Gamma_{d-l-1,k}\Big)^{q^j}}{\Gamma_{d,0}\Gamma_{d,1}\prod_{k=2}^{N-1}\Gamma_{d-1,k}}. 
\end{align}
In particular, the following identity holds: 
\begin{align}\label{Eq:periodlim}
 \lim_{d\to\infty} \frac{L_{Nd + 1}^{\sigma^2 }}{ L_{Nd+1- j} ^{q^j\sigma^2 }}=\bket{N}^{W_j} \cdot \Theta^{\zeta_j(\sigma)}\cdot\left(\frac{\prod_{k=0}^{N-1}\Gamma_{\infty,k}}{\bket{1}\cdots\bket{N-1}}\right)^{q^j-1}.
\end{align}
\end{lem}
\begin{proof}

      From Lemmas \ref{Lem:leLj} and \ref{Lem:ald}, we have 
    \begin{equation}\label{Eq:Lnd1}
      L_{Nd+1}^{\sigma^2 }=\frac{\bket{1} \cdots\bket{N-1} \cdot \bket{N} ^{J_1}\Theta^{\zeta_{Nd+1}(\sigma)}}{\Gamma_{d,0}\Gamma_{d,1}\prod_{k=2}^{N-1}\Gamma_{d-1,k}},    
    \end{equation}
where
    \begin{equation*}
   J_1 =\frac{1}{q^N-1}\Big(q^{Nd+1}+q^{Nd}-q - 1+\sum_{k=2}^{N-1}q^k(q^{(d-1)N}-1)\Big).
   \end{equation*}
Similarly, for an integer $d>l$, Lemmas \ref{Lem:leLj} and \ref{Lem:ald} yield 
       \begin{equation}\label{Eq:Lnd1j}
 L_{Nd+1-j}^{\sigma^2}=L_{N(d-l)+i}^{\sigma^2}= \frac{\bket{1} \cdots\bket{N-1}\cdot \bket{N}^{J_2}\Theta^{\zeta_{N(d-l)+i}(\sigma)}}{\prod_{k=0}^{i}\Gamma_{d-l,k}\prod_{k=i+1}^{N-1}\Gamma_{d-l-1,k}},
    \end{equation}
where 
\[ J_2=\frac{1}{q^N-1}\Big(\sum_{k=0}^{i}q^k(q^{(d-l)N}-1)+\sum_{k=i+1}^{N-1}q^k(q^{(d-l-1)N}-1)\Big).\]
    Combining Equations \eqref{Eq:Lnd1} and \eqref{Eq:Lnd1j}, we find
    \begin{align}\label{Eq:periodpf}
 \nonumber \frac{L_{Nd + 1}^{\sigma^2}}{ L_{Nd+1- j} ^{q^j \sigma^2}}={}&\Big(\bket{1}\cdots \bket{N-1} \Big)^{1-q^j} \bket{N}^{ J_1-q^{j}J_2} \cdot \Theta^{ \zeta_{Nd+1}(\sigma)-q^{j} \zeta_{N(d-l)+i}(\sigma)}\\
  &\cdot\frac{\Big(\prod_{k=0}^{i}\Gamma_{d-l,k}\prod_{k=i+1}^{N-1}\Gamma_{d-l-1,k}\Big)^{q^j}}{\Gamma_{d,0}\Gamma_{d,1}\prod_{k=2}^{N-1}\Gamma_{d-1,k}}. 
\end{align}
By the assumption $j=Nl-i+1$, we have 
    \begin{align}\label{Eq:Nc}
      J_1-q^{j}J_2=\frac{1}{q^N-1}\Big(\sum_{k=-N+2}^{1}q^{Nd+k}-\sum_{k=0}^{N-1}q^{k}-\sum_{k=-N+2}^{1}q^{Nd+k}+q^j\sum_{k=0}^{N-1}q^{k}\Big)&=W_j
    \end{align}
    and
    \begin{equation}\label{Eq:Thec}
     \zeta_{Nd+1}(\sigma)-q^{j} \zeta_{N(d-l)+i}(\sigma) = \zeta_j(\sigma).
    \end{equation}
    Substituting the equalities \eqref{Eq:Nc} and \eqref{Eq:Thec} into Equation \eqref{Eq:periodpf} yields
    Equation \eqref{Eq:period}.
   
   Finally, we obtain Equation \eqref{Eq:periodlim} by taking the limit $d\to \infty $ in Equation \eqref{Eq:period}.
\end{proof}
 
\begin{thm}\label{Th:kenelpitil}
   Given the valuation $ v_{\eta} $ on $ \mathbb{C}_{\infty}$, the Drinfeld module associated with the exponential function $ E_\infty (z) $ is isomorphic to 
    $ \Psi^{(2)} $. Moreover, the kernel of $ \exp_{(2)} $ is exactly the lattice $ \tilde{\pi }  A $; i.e., 
\begin{equation}\label{Eq:Einfz}
    E_{\infty}(z)=\tilde{\pi}^{-1} \exp_{(2)}(\tilde{\pi}z), 
\end{equation} 
where
\begin{equation}\label{eq:pi}
\tilde{\pi}= (-\Theta)^{\frac{\sigma^2}{q-1}}\Theta^{\sigma} \prod_{k=1}^{\infty} (\frac{\Theta^{q^k}}{\Theta^{q^k} -\Theta^{\sigma^k}}).
\end{equation}
\end{thm}
\begin{proof}
The limit 
\begin{equation*}
   \tilde \pi = \left(  \lim_{d\to \infty} \frac{L_{Nd + 1}^{\sigma^2}}{(-1)^jL_{Nd+1- j} ^{q^j\sigma^2}}\right)^{\frac{1}{q^j-1}}
\end{equation*}
is independent of the choice of the positive integer $ j $.
 To verify this, we apply Equation \eqref{Eq:periodlim} in Lemma \ref{Lem:LdLdj} to obtain
 \begin{equation*}
   \tilde \pi = \left((-1)^{j}\bket{N}^{W_j} \cdot \Theta^{\zeta_j(\sigma)}\right)^{\frac{1}{q^j-1}}\cdot\frac{\prod_{k=0}^{N-1}\Gamma_{\infty,k}}{{\bket{1}\cdots\bket{N-1}}}.
\end{equation*}
Combining this with Equations \eqref{Eq:leGap} and  \eqref{Eq:[N]} implies Equation \eqref{eq:pi}.

From Equation \eqref{Eq:Einf}, it follows that 
    \begin{equation}\label{Eq:Einf2}
 E_\infty (z) = \lim_{k \to \infty}\sum_{j =0 }^{ k N  +1 } (-1)^ j  \frac{ L_{k N +1}^{\sigma^2} }{  D_j^{\sigma^2}  L_{k N +1 - j }^{q^j\sigma^2}} z^{q^j}= \frac{1}{\tilde{\pi}}\sum_{j =0 }^{ \infty }  \frac{ (\tilde \pi z)^{q^j} }{  D_j^{\sigma^2} } .  
\end{equation}
Combining the expression for $\exp_{(i)}$ in \eqref{Equ:expj} that 
\begin{equation}\label{Eq:piEinf}
  \exp_{(2)}(\tilde \pi z) = \sum_{j =0 }^{ \infty }  \frac{ (\tilde \pi z)^{q^j} }{  D_j^{\sigma^2} } .
\end{equation}
Equation \eqref{Eq:Einfz} now follows from Equations \eqref{Eq:Einf2} and \eqref{Eq:piEinf}.
\end{proof}
Therefore, we call $ \tilde{\pi} $ the Carlitz period of $ \Psi^{(2)} $.

\section{Factor Decomposition Formula for Drinfeld Modules}\label{Sec:Factor}
The main purpose of this section is to establish factorization formulas for rank-one Drinfeld $\A$-modules, including both the standard modules $ \Psi^{(i)}$ and the sign-normalized Hayes modules $ \psi^{u_{k,\mu}}$.  We further study the ideal action and Galois action on these modules.

\subsection{Narrow Class Field of $ K $}
Recall that the class field $ H$ of
$ K $ is a constant field extension of $ K $, namely
$ H = \mathbb{F}_{q^{N}}(\theta)= \mathbb{F}_q(\theta, \eta)$. 
We now use Hayes $\A$-module to determine the narrow class field of $ K $ associated with the sign function $ \Sgn_{\eta}$ (see Equation \eqref{Eq:Sgn_eta}).
Remember that $I_0$ and $I_{\infty}$ are the two ideals defined in Equation \eqref{Eq:Ideals} of Section \ref{Sec:NandP}. We denote by $ \psi _{I_{0}}$ and $ \psi _{I_{\infty}}$ the respective annihilators of $I_{0}$ and $I_{\infty}$. 
\begin{prop}\label{Pro:Hu}
    Let $ \psi^u $ be a Hayes $\A$-module such that the annihilator of the ideal $I_{\infty}$ is given by
     $ \psi_{I_\infty}^u = \tau - u      
     $  for some indeterminate $ u $.
    Suppose that $ \ell \in \bar L^*$ is an isogeny from $ \Psi $ to $ \psi^u $, i.e.,
    \begin{align}\label{Eq:proPpsi}
\Psi_a=\ell^{-1}\psi_a^u \ell
    \end{align}
    for all $a\in \A$. 
    \begin{enumerate}
        \item Then $\ell$ and $u$ satisfy the relation
    \begin{equation}\label{Eq:proell}
        \ell^{q-1}=-\frac{u}{\Theta}.
    \end{equation}
    \item Furthermore, the relation between $ \Theta $ and $  u $ satisfies 
    \begin{align}\label{Eq:PuTh}
          u^{W_N}=(- \Theta)^{\gamma_N(\sigma)} ,
    \end{align}
    where $ \gamma_N (\sigma)$ is defined in Notation \ref{No:gammak}. 
    \item Additionally,  the annihilator of the ideal 
 $I_{0}$ is 
  \begin{equation}\label{Eq:propsiI0}
        \psi_{I_0}^u = \tau - \frac{\theta }{\eta} u.  
    \end{equation}
    \end{enumerate} 
\end{prop}

\begin{proof}
\begin{enumerate}
   
    \item 
Suppose that $ \ell \in \bar L^*$ is an isogeny from $ \Psi $ to $ \psi^u $, i.e., 
\[ \ell^{-1} \psi_{T_i}^u  \ell =\Psi_{T_i} 
\]
for $i = 0,1,\cdots, N-1$.
Then $ \ell $ induces an isomorphism from $ \ker \Psi_{I_{\infty}}$ to $ \ker \psi_{I_{\infty}}^u$.
That is, 
\[
\ker \psi_{I_{\infty}}^u = \{ \ell \cdot x |  \Psi_{I_\infty}(x) = 0 \}. 
\]
Therefore, we have 
\begin{equation}\label{Eq:kerpP}
    \psi_{I_{\infty}}^u (z) = \prod_{y \in \ker \psi_{I_{\infty}}^u } (z- y) = \prod_{x \in \ker \Psi_{I_{\infty}} } (z- \ell x ) .
\end{equation}
Recall $\Psi_{I_{\infty}} = \tau + \Theta$  in Proposition \ref{prop:ann}. It follows that
\begin{equation}\label{Eq:kerP}
\prod_{x \in \ker \Psi_{I_{\infty}} } (\ell^{-1}z- x )=\Psi_{I_\infty}(\ell^{-1}z)=(\tau+\Theta)(\ell^{-1}z)=\ell^{-q}z^q+\Theta \ell^{-1}z.
\end{equation}
Combining Equations \eqref{Eq:kerpP} and \eqref{Eq:kerP}, this yields 
\[
  \psi_{I_{\infty}}^u (z) = \ell^{q}  \Psi_{I_{\infty}} (\ell^{-1} z )=z^q+\Theta \ell^{q-1}z.
\]
As twisted polynomial,  $\psi_{I_{\infty}}^u$ can be written as 
\begin{equation}\label{Eq:isopPinf}
\psi_{I_{\infty}}^u = \ell^q \Psi_{I_\infty} \ell^{-1}=\tau+\Theta\ell^{q-1} . 
\end{equation}
On the other hand, we write $\psi_{I_\infty}^u= \tau - u$ by assumption. Thus, Equation \eqref{Eq:isopPinf} implies  
\begin{align*}
 u   = -\ell^{q-1} \Theta,
\end{align*}
which yields \eqref{Eq:proell}. 

\item 
Since $ \psi^u $ is sign-normalized by definition, the leading term of $ \psi_{T_0}^u $ must be $1$.
Comparing the leading terms of the equality $\ell^{-1} \psi_{T_0}^u  \ell =\Psi_{T_0}$, we have
\begin{equation}\label{Eq:biTN}
    \ell^{q^N-1} = \bi{T_0}{N}= (-\Theta)^{\gamma_N(\sigma) - W_N},
\end{equation}
where the leading term of $  \Psi_{T_0} $ is obtained in Equation \eqref{Eq:TjN}. 
Equation \eqref{Eq:proell} yields 
\begin{align}\label{Eq:ulT}
    \ell^{q^N-1} = u^{W_N}(-\Theta)^{-W_N}. 
\end{align}
Combining Equations \eqref{Eq:biTN} and \eqref{Eq:ulT}, we have 
\[
    u^{W_N} = (-\Theta)^{ \gamma_N(\sigma) }.
\]

\item By analogy with Equation \eqref{Eq:isopPinf}, we have 
 \[
 \psi_{I_{0}}^u = \ell^q \Psi_{I_0} \ell^{-1}. 
 \]
Applying Proposition \ref{prop:ann} to $\Psi_{I_0}$ yields
\begin{align}\label{Eq:isopP0}
\psi_{I_{0}}^u = \tau + \frac{\theta}{\eta}\Theta \ell^{q-1}. 
\end{align}
Substituting $ \Theta \ell^{q-1} = - u $ into Equation \eqref{Eq:isopP0}, we have 
$
\psi_{I_{0}}^u = \tau - \frac{\theta}{\eta} u .$ 
\end{enumerate}
\end{proof}
Throughout this paper, we fix once and for all $ u $ to be a root of the polynomial
 \[
    x^{W_N} - (-\Theta)^{\gamma_N(\sigma) } \in H[x].
\] 
Combining Theorem \ref{thm:Hilbert} and Proposition \ref{Pro:Hu}, we conclude that the narrow class field of $K$ is given by  
\[
H^+=K(u, \eta) = \mathbb{F}_q(\theta, \eta,u).
\]
Notice that the narrow class number of $\A$ is given by 
\[ [H^+: K ]= [H^+: H ][H: K ] =    W_N \cdot N,  \] 
which matches the formula \eqref{Eq:cl+}. 
\begin{notation}\label{No:etax} 
 We retain the notation for $u$ introduced above.
 \begin{enumerate}
     \item Denote by $ U_\rho $ the set of conjugate elements of $ u $ over $ K $.
\item For $ x\in U_{\rho}$, the minimal polynomial of $x$ over $ H $ is given by 
\[
    x^{W_N} = \left(\frac{1}{\eta_x- \theta} \right)^{\gamma_N(\sigma) }, 
\]
where $ \eta_x $ is a root of $ \rho $. 
\item For $ x\in U_{\rho}$, we denote by $\psi^x$ the Hayes $\A$-module over $H^+$ with 
\(\psi^{x}_{I_\infty}=\tau-x\) or equivalently \(
\psi^{x}_{I_0}=\tau-\frac{\theta}{\eta_x}x\) . 
Notice that for any $\sigma\in \Gal(H^+/K)$ and $a\in \A$, we have $\sigma \psi^{x}_{a}=\psi^{\sigma x}_a$.

 \end{enumerate}   
\end{notation}

With these notations, the following corollary can be derived directly from Theorems \ref{thm:homogeneous} and \ref{thm:Hilbert}. 

\begin{cor}
  All Hayes $\A$-modules over $H^+$ are given by
       $ \psi^{x} $, where $x $ ranges over $U_{\rho}$.   Moreover, $\psi^x$ is of $ \eta_x^{(-1)} $-type.
\end{cor}

\subsection{Galois Action}
To study the Galois group of $H^+/ K$, we introduce the following three automorphisms $ \sigma_{\infty}, ~\sigma_0, ~M_{\mu} $ of $H^+$.
\begin{defn} \label{Def:sigM}
\begin{enumerate}
\item   
Define $\sigma_{\infty}$ to be the automorphism of $H^+$ satisfying
 \[
    \sigma_{\infty} (u) = (\theta - \eta)^{q-1} u^{q} \quad \text{and } \sigma_{\infty}|_{H} = \sigma.
    \]

\item Suppose that $ \mu ^{W_N} = 1 $. We define $ M_{\mu} $ by 
\[
    M_\mu (u) =  \mu \cdot u \quad \text{and }  M_\mu |_{H} = \id .
\] 
\item Set $ \eta_* = \eta^{\frac{1-q}{q}} $. It is clear that $ \eta_*^{W_N} = 1 $.
We define $\sigma_{0}  =   \sigma_{\infty}  M_{\eta_*} $.

\end{enumerate}
\end{defn}
 The following proposition is readily verified.
\begin{prop}\label{Pro:sigmaU}
For $x\in U_{\rho}$, let $\eta_x $ be the same notation as in Notation \ref{No:etax}. 
 The following three statements hold. 
\begin{enumerate}
\item The automorphism $\sigma_{\infty}$ satisfies the equality  
    \[
    \sigma_{\infty} (x) = (\theta - \eta_x)^{q-1} x^{q}. 
\]
\item Let $\mu$ and $\nu$ be two $W_N$-th roots of unity. Then,  
\[
    M_\mu (x) = \eta_x \cdot x.
\] 
Moreover, we have
\[
 \quad M_{\mu}M_{\nu}=M_{\mu\nu}. 
\]
\item The automorphism $\sigma_{0}$ satisfies 
\[
   \sigma_{0} (x) = {\eta_x^{1-q}} {\sigma_{\infty}(x)} = \left(\frac{\theta - \eta_x}{\eta_x}\right)^{q-1} x^{q}  
\]
and 
\begin{equation*}
     \sigma_{0}|_{H} = \sigma.
\end{equation*}
\end{enumerate} 
\end{prop}

\begin{prop}\label{Pro:Gag}
    The Galois group $\Gal(H^+/K)$ is generated by $ \sigma_{\infty} $ and $ M_\mu $ with $ \mu^{W_N} =1 $, i.e.,
\[
\Gal(H^+/K)=\{\sigma_{\infty}^k M_{\mu} ~|~ k=0,\cdots,N-1 ~\text{and}~ \mu^{W_{N}}=1\}.
\]
Moreover,  $\Gal(H^+/K)$  is isomorphic to the Abelian group 
\[
    \mathbb{Z}_{N} \times \mathbb{Z}_{W_N} .
\]
\end{prop}
\begin{proof}
The commutativity relation of $ M_\mu $ and  $ \sigma_{\infty} $ is immediate, i.e.,  
\begin{equation}\label{Eq:commutativity}
    M_\mu \sigma_{\infty }=\sigma_{\infty}M_{\mu}. 
\end{equation}
The restriction of the equality \eqref{Eq:commutativity} to $H$ is evident.  
From Definition \ref{Def:sigM}, we have 
\[
    M_\mu (\sigma_{\infty } (u) ) = M_{\mu}(\theta - \eta)^{q-1} M_\mu (u)^q =(\theta - \eta)^{q-1} \mu^q u^q = \sigma_{\infty}(\mu u) = \sigma_{\infty}(M_{\mu} u ).   
\]
It suffices to prove the equalities 
\begin{equation}\label{Eq:sigid}
   \sigma_{\infty}^{N}=\id,  
\end{equation}
and 
\begin{equation}\label{Eq:Mid}
\quad M_{\mu}^{W_N}=\id.
\end{equation}
Equation \eqref{Eq:Mid} follows immediately from (2) of Proposition \ref{Pro:sigmaU}.   
We claim that
\begin{equation}\label{Eq:sigU}
\sigma_{\infty}^{j}(u) = \Theta^{ (1-q)\gamma_j(\sigma)} u^{q^j } .
\end{equation} 
Notice that $ \sigma_{\infty} (u) = \Theta^{(1-q)} u^q $. Induction then gives
\begin{align*}
      \sigma_{\infty}^{j+1}(u) & = \Theta^{(1-q)(\sum_{i=0}^{j-1} q^i \sigma^{j-i} )} (\sigma_{\infty}u)^{q^j } \\
      & =  \Theta^{(1-q)(\sum_{i=0}^{j-1} q^i \sigma^{j-i} )}\Theta^{(1-q)q^j} u^{q^{j+1}} \\
       & =  \Theta^{(1-q)(\sum_{i=0}^{j} q^i \sigma^{j-i} )}  u^{q^{j+1}}\\
       & =\Theta^{ (1-q)\gamma_{j+1}(\sigma)} u^{q^{j+1} }  . 
\end{align*}
Substituting $j=N$ into Equation \eqref{Eq:sigU} gives
\[
\sigma_{\infty}^{N}(u) = \Theta^{(1-q) \gamma_N(\sigma)} u^{q^N }.  
\]
Combining this with the relation  $u^{W_N}=(-\Theta)^{\gamma_N(\sigma)}$ established in Proposition \ref{Pro:Hu}, we have 
\[
\sigma_{\infty}^{N}(u) =(-u)^{(1-q)W_N}\cdot u^{q^N}=(-1)^{1-q^N}\cdot u=u.  
\]
Combining this equality with $\sigma_{\infty}^N|_{H}=\sigma^N|_{H}=\id$ implies Equation \eqref{Eq:sigid}.
This completes the proof.  
\end{proof}

\begin{notation}\label{No:ukmusig}
 Applying Proposition \ref{Pro:Gag}, we can express all the elements of $ U_\rho $ as 
\begin{align*}
     u_{k,\mu} = \sigma_{\infty}^k M_{\mu}  (u), \text{ where $\mu^{W_N} = 1 $}.
\end{align*}
In particular, we have $u=u_{0,1}$.
\end{notation}
Applying equality \eqref{Eq:sigU} respectively,
we have 
\begin{align}\label{Eq:ukThe}
u_{k,\mu}= \mu^{q^k}\sigma_{\infty}^k(u)=\mu^{q^k}\Theta^{ (1-q)\gamma_k(\sigma)} u^{q^k }. 
\end{align}
  It follows from Proposition \ref{Pro:sigmaU} that $ \eta_{u_{k,\mu}} = \eta^{q^k} $ and 
  \begin{align}\label{Eq:uksig}
    \sigma_{\infty}(u_{k,\mu}) = u_{k+1, \mu};   \qquad
    M_\nu(u_{k,\mu }) = u_{k, \nu\mu};
    \qquad
    \sigma_{0}(u_{k,\mu}) = u_{k+1,\mu \cdot \eta_* }.  
  \end{align}
  



\begin{lem}\label{Lem:siginf0}
 Using the notations from Notation \ref{No:ukmusig}, we have 
\[
\sigma_{0}^{n}\sigma_{\infty}^{j}(u_{k,\mu})=u_{j+n+k,\mu\cdot\eta_*^n}=(\mu\eta_*^n )^{q^{j+n+k}}\Theta^{ (1-q)\gamma_{j+n+k}(\sigma)} u^{q^{j+n+k}}  .
\]
\end{lem}

\begin{proof}
From the identities in Equation \eqref{Eq:uksig}, we obtain 
\begin{align*}
  \sigma_{0}^{n}\sigma_{\infty}^{j}(u_{k,\mu})=\sigma_0^{n}(u_{k+j,\mu})=u_{j+n+k,\mu\cdot\eta_*^{n}},
\end{align*}
which establishes the first equality.
Combining Notation \ref{No:ukmusig} with Equation \eqref{Eq:sigU} yields 
\begin{align*}
 u_{j+n+k,\mu\cdot\eta_*^{n}} =(\mu\eta_*^n )^{q^{j+n+k}}\sigma_{\infty}^{j+n+k}(u)=(\mu\eta_*^n )^{q^{j+n+k}}\Theta^{ (1-q)\gamma_{j+n+k}(\sigma)} u^{q^{j+n+k}}.
\end{align*}
\end{proof}

 We are now in a position to describe the Artin symbol of $ H^+ / K $. 
\begin{prop} \label{Pro:AS}
Denote by  
\[
\left(H^+/K;-\right): \text{ideal group of $\A$ }\to  \Gal (H^+/K)
\]
the Artin symbol of the Abelian extension $H^+/K$.  
Then the following three equalities hold:
\begin{equation}\label{Eq:Assiginf}
    \left(H^+/K;I_{\infty}\right) = \sigma_{\infty};
\end{equation}
\begin{equation}\label{Eq:Assig0}
   \left(H^+/K;I_0\right)=  \sigma_{0};
\end{equation}
\begin{equation}\label{Eq:Mi}
   \left(H^+/K;(T_i)\right) = M_{ \eta_*^i },
\end{equation}
where $\eta_* = \eta^{\frac{1-q}{q}}$.
\end{prop} 

\begin{proof}

Let us recall the definition of
the Artin symbol with respect to the Abelian extension $ H^+/ K $. For a prime ideal $P$ of $K$ (or a place in the function field setting) unramified in $H^+/K$, $( H^+/ K; P) $ is the unique automorphism  $ \sigma \in \Gal (H^+/ K ) $ such that 
\[
    \sigma(r) = r^{q^{\deg(P)}} \mod R,
\]
where $R$ is a place of $ H^+ $ over $P$, and $r \in H^+ $ is any integral element at $ R $.

(1) First, to verify the equality \eqref{Eq:Assiginf}, we set 
\[
 \tilde{u} = (\eta- \theta)  u. 
 \]
 Then it follows from  \eqref{Eq:PuTh} that
 \begin{equation}\label{Eq:tilu}
    (\tilde{u})^{W_N} =  (\theta - \eta) ^{W_N}\Theta^{\gamma_N(\sigma)}=\frac{(\theta - \eta) ^{W_N}}{(\theta-\eta^{(N-1)})(\theta-\eta^{(N-2)})^{q}\cdots(\theta-\eta)^{q^{N-1}}},  
 \end{equation}
 and the extension $H^+/K$ is generated by 
 $\tilde{u}$ and $\eta$.

Let $P_{\infty}$ be the place corresponding to the ideal $I_{\infty}$ in $\A$. We need to determine all places in $H^+$ over $P_{\infty}$. 
In the field extension $H/ K $, there exists a unique place, say $Q_{\infty} $, over $P_{\infty}$, 
with ramification index $e(Q_{\infty}| P_{\infty}) =1 $ and relative degree $f (Q_{\infty}| P_{\infty}) = N $. 
It follows from \eqref{Eq:tilu} that
\begin{equation}
 \tilde{u}^{W_N}-1=\prod_{\mu^{W_N}=1} (\tilde{u}-\mu)=0 \mod Q_\infty . 
\end{equation}
Applying Kummer’s Theorem \cite[Theorem 5.8.2]{V-S06}, we conclude that the place $ Q_\infty $ splits completely in the field extension $ H^+ / H $ into $W_N$ distinct places,  corresponding to the zeros $R_{\infty}^{(\mu)}$ of 
$ \tilde{u}- \mu  $ where $\mu^{W_N} = 1$.
In other words, $ R_{\infty}^{(\mu)}$  is the common locus of $1/\theta$ and $\tilde{u}-\mu$. 
The places $Q_\infty $ and $ R_{\infty}^{(\mu)}$
  over $P_\infty$ are illustrated in Figure~\ref{Fig:w1ga0}. 

Moreover, we fix $\mu = \mu_0$ to be some $W_N$-root of unity.  Then 
we know that $\tilde{u} - \mu_0$ is a uniformizer of $ R_{\infty}^{(\mu_0)} $ and integral over $ P_{\infty}$. Thus, the integral elements over $ R_{\infty}^{(\mu_0)} $ are generated by $ \tilde{u} - \mu_0 $ over $\mathbb{F}_{q}(\eta) $. 
This follows from the definition of $\sigma_{\infty}$ that
\[
   \sigma_{\infty}(\tilde{u}) = \frac{\theta-\eta^{(1)} }{ \theta - \eta } \cdot \tilde{u}^q =  \tilde{u}^q \mod R_{\infty}^{(\mu_0)} 
\]
and
\[
    \sigma_{\infty}(v)= v^q
\]
for $ v \in \mathbb{F}_{q} (\eta)\subset H $.  This implies that the equality \eqref{Eq:Assiginf} holds. 
\begin{figure}
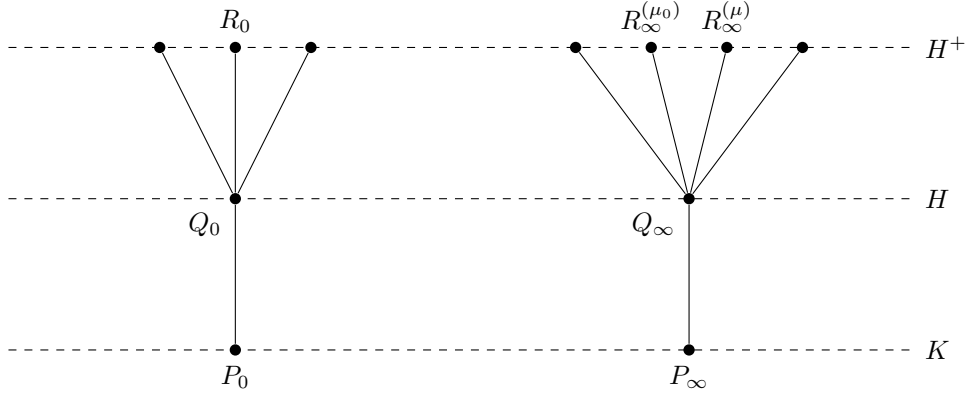

\centering
\include{graph.tex}   
 \caption{Diagram of places over $H^+/K$}   
 \label{Fig:w1ga0}  
\end{figure}

(2) Next, we adopt the same approach to check the equality \eqref{Eq:Assig0}.

Let $P_{0}$ be the place of $K$ associated with the ideal $I_{0}$ in $\A$.  For the field extension $H/K$, there exists a unique place $ Q_0 $ over $ P_0 $ with $ e(Q_0| P_0) = 1 $ and $ f(Q_0| P_0) = N $. Let $ R_{0}  $ be a place of $H^+$ over $ Q_{0} $ as in Figure \ref{Fig:w1ga0}.   Then we know that $R_{0} $ is unramified over $P_0$
 and $ u $ is integral at $ P_{0}$ (resp. $ Q_0 $).
It follows from (3) in Proposition \ref{Pro:sigmaU} that 
\[
    \sigma_0 (u) = \frac{\theta - \eta}{-\eta} u^q = u^q \mod R_{0}  
\]
and 
\[
    \sigma_0(v) = v^q
\]
for any $ v \in \mathbb{F}_{q}(\eta)$. 
Since $ u $ and $ v $ are generators of  $H^+/K$, the equality \eqref{Eq:Assig0} therefore holds.

(3) The principal ideal generated by $T_i$ admits a decomposition $I_0^{i} I_\infty^{N-i}$. 
Therefore, we have
\[
 \left(H^+/\mathbb{F}_q(\theta);(T_i)\right) = \left(H^+/\mathbb{F}_q(\theta);I_0\right)^i  \left(H^+/\mathbb{F}_q(\theta);I_\infty\right)^{N-i} = \sigma_{0}^i\sigma_{ \infty}^{N-i}.
\]
Since $\sigma_\infty^N=\id$, it follows that
\[
\sigma_{0}^i\sigma_{ \infty}^{N-i}(u)= M^{i}_{\eta_*}\sigma_{\infty}^N(u)= M_{\eta_*^i}\, (u)=\eta^i_{*}u.
\]
Thus, Equation \eqref{Eq:Mi} holds.

\end{proof}  

\subsection{Annihilator} 
We next give an explicit description of the annihilator $ \psi_{I_{0}^l I_{\infty}^{j-l} }^x  $ of the ideal $ I_0^l I_{\infty}^{j-l} $. 

\begin{notation}\label{No:su}
     For $x \in U_{\rho} $, we denote by $ \mathbf{s} (x) $  the formal generator of the Anderson motive $ M_{\psi^{x} } ~(\cong H^{+}\{\tau\} )$ associated with $ \psi^{x} $ (see Definition \ref{Def:t-m} and Notation \ref{No:etax} (3)). In this notation, we have
        \[
        a \otimes g(\tau) * ( m \mathbf{s}(x) ) = \left( g(\tau) \circ m \circ \psi_a^{x}  \right) \mathbf{s}(x),
        \]
        for all $ a \otimes g(\tau) \in \A \otimes_{\mathbb{F}_q} H^{+}\{ \tau \} $ and $ m \mathbf{s}(x) \in M_{\psi^x} $. 
\end{notation}
Applying Proposition \ref{prop:ann}, we have the following results concerning $ M_{\psi^x}$.
\begin{cor}\label{cor:gensu}
Let $\mathbf{s}(u)$ be the formal generator of $M_{\psi^u}$. 
The following equalities hold: 
     \[
     \frac{1}{ t - \eta  } * \mathbf{s}(u)
     = -\frac{\Theta}{u}(\tau - u) \mathbf{s}(u)
     \]
     and
      \[
     \frac{t}{ \eta(t - \eta)  } * \mathbf{s}(u)= -\frac{\Theta}{u}\left(\tau -\frac{\theta}{\eta} u \right) \mathbf{s}(u).
     \]
\end{cor}
\begin{proof}
     Recall that $\mathbf{s}_0 (=1)$ is the generator of $M_\Psi$ by the construction of the Drinfeld module $\Psi$.
 Let $\ell$ be an isogeny from $\Psi$ to $\psi^u$. By Proposition \ref{Pro:Hu}, we know that $\ell^{1-q}=-\frac{\Theta}{u}$. Furthermore, Proposition \ref{Pro:isomoti} implies that $\mathbf{s}(u)=\ell \mathbf{s}_0$.
 It follows from Proposition \ref{prop:ann} that
     \[
     \frac{1}{ t - \eta  } * \mathbf{s}(u)=\ell \frac{1}{ t - \eta  }* \mathbf{s}_0=\ell (\tau+\Theta)\ell^{-1}\ell\mathbf{s}_0 =(\ell^{1-q} \tau+\Theta)\mathbf{s}(u)
     = -\frac{\Theta}{u}(\tau - u) \mathbf{s}(u).
     \]
 The second equality can be derived in the same fashion.
\end{proof}
\begin{notation} \label{No:delta_x}
   Define the map $\delta\colon U_\rho\to H^+$ by
\[
    \delta (x) = \frac{\theta}{ \eta_{x} } x
\]
for $x \in U_{\rho} $. 
\end{notation}

\begin{lem}
    The map $\delta$ in Notation \ref{No:delta_x} is Galois invariant. 
\end{lem}
\begin{proof}
Applying Proposition \ref{Pro:sigmaU}, we get
    \[
\sigma_\infty\delta(x)=\sigma_\infty(\frac{\theta}{\eta_x}x)=\frac{\theta}{\eta_x^q}\sigma_\infty(x)=\delta\sigma_{\infty}(x)
    \]
and 
   \[ M_{\mu}\delta(x)=M_{\mu}(\frac{\theta}{\eta_x}x)=\frac{\theta}{\eta_x}M_{\mu}(x)=\delta M_{\mu}(x).
   \]
   Since the Galois group $ \Gal(H^+/ K) $ is generated by $ \sigma_\infty $ and $ M_{\mu}$, this lemma is valid.  
\end{proof}

\begin{lem}\label{Lem:sig0desiginf}
 For $x \in U_\rho $, we have the following equality
\begin{equation}\label{Eq:tausigma}
    (\tau - \sigma_{0}(x) )(\tau - \delta (x) ) = (\tau  - \delta \sigma_{\infty} (x) ) (\tau - x ).
\end{equation}
\end{lem}

\begin{proof}
   The equality \eqref{Eq:tausigma} follows by direct computation. The left-hand side is given by 
    \begin{align}\label{Eq:tausig0del}
        (\tau - \sigma_{0}(x) )(\tau - \delta(x) ) & = \tau^2  - \left( \left(\frac{\theta - \eta_x}{\eta_x}\right)^{q-1}  + \frac{\theta^q}{\eta_x^q}  \right) x^q \tau +  \left(\frac{\theta - \eta_x}{\eta_x}\right)^{q-1}\frac{\theta}{ \eta_{x} }  x^{q+1}\nonumber \\
        & = \tau^2  - \frac{\theta^{q+1} - \eta_x^{q+1}}{(\theta - \eta_x ) \eta_x^q }  x^q \tau +  \frac{(\theta - \eta_x)^{q-1} \theta}{\eta_x^q }  x^{q+1}. 
    \end{align}
It follows from Equation \eqref{Eq:uksig} that $\eta_{\sigma_{\infty} (x)}=\eta_x^q$.    Therefore, 
    \[
        \delta \sigma_{\infty}(x) = \frac{\theta}{\eta_{\sigma_{\infty} (x)}} \sigma_{\infty}(x) = \frac{\theta}{\eta_x^q } (\theta - \eta_x)^{q-1} x^q .
    \]
    So the right-hand side equals 
    \begin{equation}\label{Eq:taudelsiginf}
       (\tau  - \delta \sigma_{\infty} (x) ) (\tau - x ) = \tau^2 - \left(\frac{\theta}{\eta_x^q } (\theta - \eta_x)^{q-1}+1\right) x^q \tau + \frac{\theta}{\eta_x^q } (\theta - \eta_x)^{q-1} x^{q+1}.  
    \end{equation}
    Comparing \eqref{Eq:tausig0del} with \eqref{Eq:taudelsiginf}, we complete the proof.

\end{proof}

\begin{lem}\label{Lem:Isophis}
  For $ x \in U_{\rho} $, the following statements hold.
\begin{enumerate}
   \item  $\psi_{I_{0}^{l} I_{\infty}^j}^x $ is the isogeny from $ \psi^{x} $ to $\psi^{ \sigma_{0}^l \sigma_{\infty}^{j} x}$.
   \item Moreover, we have
      \[ 
    \psi_{I_0^{ l}I_\infty^{ j-l} }^{u_{k,\mu}} \mathbf{s} (u_{k, \mu }) =\mathbf{s}( \sigma_{0}^l \sigma_{\infty}^{j-l} u_{k,\mu})= \mathbf{s}( u_{j+k,\mu\cdot \eta_*^{l }} ) .
    \]
\end{enumerate} 
\end{lem}
\begin{proof}
(1) By definition, $ \psi_{I_0^l I_{\infty}^{j}}^x $ is the isogeny from  $\psi^x$ to $(\psi^{x})^{I_{0}^l I_{\infty}^{j} }$.  
Applying Theorem \ref{thm:Hilbert} and Proposition \ref{Pro:AS},  we have the isogeny 
\[
     (\psi^{x})^{I_{0}^l I_{\infty}^{j} } = \sigma_{0}^l \sigma_{\infty}^{j}\psi^{ x} .
\]
From (3) of Notation \ref{No:etax}, we have 
\[\sigma_{0}^l \sigma_{\infty}^{j}\psi^{ x} = \psi^{\sigma_{0}^l \sigma_{\infty}^{j}x } 
\]
This proves the first assertion.

(2) In particular,  by taking $x=u_{k,\mu}$,  we find that  $\psi_{I_0^{ l} I_\infty^{ j-l} }^{u_{k,\mu}}$ is an isogeny from $\psi^{u_{k,\mu}}$ to $\psi^{ \sigma_{0}^l \sigma_{\infty}^{j-l} u_{k,\mu}}$.  Applying Proposition \ref{Pro:isomoti} yields that 
\[
 \psi_{I_0^{ l}I_\infty^{ j-l}  }^{u_{k,\mu}}  \mathbf{s} (u_{k, \mu }) =\mathbf{s}( \sigma_{0}^l \sigma_{\infty}^{j-l} u_{k,\mu}).
\]
So the second assertion follows from
     \begin{equation}\label{Eq:sig0inf}
     \sigma_{0}^l \sigma_{\infty}^{j-l} u_{k,\mu} = u_{ j+k , \mu\cdot \eta_*^{l }},
     \end{equation}
    where we make use of Lemma \ref{Lem:siginf0}. 
\end{proof}

\begin{notation}
Let $ S $ be a subset of $\{ 0, \cdots, j-1 \}$ with $ j \geqslant 1 $. For $i<j$, we introduce the notations:
\[
    [S_i^-] := \# \{ s \in S | s < i  \},   
\]
and
\[
    [S_i] := \# \{ s \in S | s = i  \}  = 0 \text{ or } 1. 
\]
In particular, we always have $[S_0^-]=0$.  For $x \in U_{\rho}$, we define the twisted polynomial 
\begin{equation}\label{Eq:psijS}
  \psi_{j,S}^x =  ( \tau - \delta^{[S_{j-1}]} \sigma_{0}^{[S_{j-1}^-]} \sigma_{\infty}^{j-1-[S_{j-1}^-]} x)\cdots  ( \tau - \delta^{[S_1]} \sigma_{0}^{[S_1^-]} \sigma_{\infty}^{1-[S_1^-]} x) ( \tau - \delta^{[S_0]}  x).  
\end{equation}
\end{notation}

\begin{lem}\label{Lem:psiSS'}
The twisted polynomial $ \psi_{j, S }^x $  depends only on the cardinality of $ S $; that is 
\begin{equation}\label{Eq:psiSS'}
    \psi_{j, S}^ x  =\psi_{j, S'}^x,
\end{equation}
whenever $S' $ has the same cardinality as $S$.
\end{lem}
\begin{proof}
 Let $ l $ be the cardinality of $S$. Assume that 
\[
S=\{s_1,s_2, \cdots, s_l\}
\]
and
\[
S'=\{s_1',s_2', \cdots, s_l'\},
\]
where $s_m<s_{m+1}$, $s_m'<s_{m+1}'$, $s_{l+1}=s_{l+1}'=j$ and $1\leqslant m\leqslant l$.
Without loss of generality, to prove the equality \eqref{Eq:psiSS'} it suffices to consider the case where $S$ and $S'$ differ by exactly one element; any two such sequences can be connected by a series of such elementary swaps. Hence we may assume that there exists an index $n$ ($1 \leqslant n \leqslant l$) such that $s_m = s_m'$ for all $m \neq n$, and $s_n' = s_n + 1$, with the additional condition $s_n + 1 < s_{n+1}$ to preserve the strict ordering (where $s_{l+1}=j$).

Under this assumption, the following relations clearly hold:
\[
[S_{s_{n}+1}]=[S'_{s_{n}'-1}]=0, \qquad [S_{s_{n}}]=[S'_{s_{n}'}]=1,
\]
\[
[S_{s_{n}}^-]=[S^{'-}_{s_{n}'-1}]=[S^{'-}_{s_{n}'}]=n-1,  \qquad[S^-_{s_{n}+1}]=n.
\]
From Equation \eqref{Eq:psijS},  we obtain 
\begin{align*}
  \psi_{j,S}^x =\Lambda' \circ \Lambda_S  \circ \Lambda''
\end{align*}
and 
\begin{align*}
  \psi_{j,S'}^x = \Lambda' \circ\Lambda_{S'}  \circ \Lambda'', 
\end{align*}
where $\Lambda'$, 
$\Lambda''$,
$\Lambda_S$ and $\Lambda_{S'}$ are twisted polynomials defined as follows:
\[
\Lambda'=( \tau - \delta^{[S_{j-1}]} \sigma_{0}^{[S_{j-1}^-]} \sigma_{\infty}^{j-1-[S_{j-1}^-]} x)\cdots( \tau - \delta^{[S_{s_n+2}]} \sigma_{0}^{[S_{s_n+2}^-]} \sigma_{\infty}^{s_n+2-[S_{s_n+2}^-]} x)
\]
\[
\Lambda''=( \tau - \delta^{[S_{s_n-1}]} \sigma_{0}^{[S_{s_n-1}^-]} \sigma_{\infty}^{s_n-1-[S_{s_n-1}^-]} x)\cdots ( \tau - \delta^{[S_1]} \sigma_{0}^{[S_1^-]} \sigma_{\infty}^{1-[S_1^-]} x) ( \tau - \delta^{[S_0]} x)
\]
\[
\Lambda_S = ( \tau - \sigma_{0}^{n} \sigma_{\infty}^{s_{n}+1-n} x)( \tau - \delta \sigma_{0}^{n-1} \sigma_{\infty}^{s_n-n+1} x)
\]
and
\[
\Lambda_{S'} = ( \tau - \delta\sigma_{0}^{n-1} \sigma_{\infty}^{s_{n}-n+2} x)( \tau -  \sigma_{0}^{n-1} \sigma_{\infty}^{s_n-n+1} x).
\]
%
  Using Lemma \ref{Lem:sig0desiginf} by replacing variable $x$ with the expression  $\sigma_0^{n-1}\sigma_{\infty}^{s_n+1-n}x$, we get
  \(
\Lambda_{S}= \Lambda_{S'}.
  \)
Therefore, the equality \eqref{Eq:psiSS'} holds.
    
\end{proof}


\begin{notation}\label{No:Gammai}
 For $x\in U_{\rho}$ and $i,j\in \mathbb{Z}_{\geqslant 0}$,  we define 
\begin{equation}\label{Eq:noGa}
  \FS{x}_{j}^{i}=  
\begin{cases}
 \sigma_{\infty}^i(x),  \qquad &  \text{ for } 0\leqslant i<j;\\
   \delta\sigma_0^{i-j}\sigma_{\infty}^j(x), \qquad & \text{ for } i\geqslant j.
    \end{cases}
\end{equation}

\end{notation}

\begin{lem}\label{Lem:exFSx}
 The symbol $\{ x \}_{j}^i$ admits the explicit formula
    \begin{equation}\label{Eq:FSukmu}
  \FS{x}_{j}^{i}=  
\begin{cases}
  \Theta^{ \sigma^k (1-q)\gamma_i(\sigma)} x^{q^i},  \qquad & { \rm for } ~0\leqslant i<j;\\
     \theta\eta_x^{ q^{i-1}(1-q)(i-j) - q^i} \Theta^{\sigma^{k} (1-q)\gamma_i(\sigma)}x^{q^i}, \qquad & {\rm for } ~i\geqslant j,
    \end{cases}
\end{equation}
where $k$ is the integer such that $ \eta^{q^k} = \eta_x $. 
\end{lem}
\begin{proof}
    We verify only the second case.
    Set $x = u_{k,\mu}$ and $ \eta_* = \eta^{\frac{1-q}{q}}$. 
    It is evident that 
    \[
        \gamma_{j+n+k}(\sigma) =  \gamma_{j+n} (\sigma)   \sigma^k + q^{j+n}\gamma_k(\sigma) .
    \]
    Then 
    \begin{align*}
        \sigma_{0}^{n}\sigma_{\infty}^{j}(u_{k,\mu}) & =\eta_*^{nq^{k+n+j}}\Theta^{\sigma^k(1-q)\gamma_{j+n}(\sigma)} (\mu^{q^k}\Theta^{ (1-q)\gamma_k(\sigma)} u^{q^k } )^{q^{j+n}}  \quad \text{by Lemma \ref{Lem:siginf0}} \\ 
        & = \eta_*^{nq^{k+n+j}}\Theta^{\sigma^k(1-q)\gamma_{j+n}(\sigma)} u_{k,\mu}^{q^{j+n}}  \quad\text{by Equation \eqref{Eq:ukThe}}.
    \end{align*}
  Substituting $n = i-j$ yields
    \[
          \delta \sigma_{0}^{i-j}\sigma_{\infty}^{j}(u_{k,\mu}) = \theta \eta_*^{(i-j)q^{k+i }}\eta^{ - q^{k+i}}  \Theta^{\sigma^k(1-q)\gamma_{ i }(\sigma)} u_{k,\mu}^{q^{ i }} .
    \]
   This confirms the second equality in Equation \eqref{Eq:FSukmu}.
\end{proof}
Using these notations,  
Equation \eqref{Eq:psijS} can be rewritten as  
\begin{equation}\label{Eq:psiSj}
  \psi_{j,S}^x  =(\tau-\FS{x}^{j-1}_{j-n})\circ\cdots\circ(\tau-\FS{x}^{1}_{j-n} )\circ (\tau-\FS{x}^0_{j-n} ),
\end{equation}
by fixing the subset $S = \{ j-n ,\, j-n+1,\,  \cdots ,\,j-1 \}$. 
If we choose $x = u_{k,\mu}$, then 
\begin{equation}\label{Eq:psiSj_u}
      \psi_{j,S}^x = \psi_{j,S}^{\prime} \psi_{j,S}^{\prime\prime},
\end{equation}
where
\[
        \psi_{j,S}^{\prime} = (\tau - \delta u_{k+j - 1 ,\mu \eta_*^{n}} )\circ (\tau - \delta u_{k+j - 2 ,\mu \eta_*^{n-1}} )\circ\cdots \circ(\tau - \delta u_{k+j - n ,\mu \eta_*} )
\]
and 
\[
        \psi_{j,S}^{\prime \prime} = (\tau -   u_{k+j - n -1,\mu} )\circ(\tau -   u_{k+j - n -2, \mu})\circ\cdots \circ (\tau -   u_{k+1,\mu} )\circ(\tau -   u_{k,\mu} ).
\]
The following lemma describes the annihilators for ideals  $I_0^nI_\infty^{j-n}$.
 
 \begin{lem}\label{Lem:psiIS}  
     For $x\in U_\rho$, $ 0 \leqslant n \leqslant j $ and any subset $S \subseteq \{ 0 , \cdots , j-1 \}$ of cardinality $ n $, we have 
    \begin{equation}\label{Eq:psiIS}
        \psi_{I_0^{ n} I_\infty^{ j-n}  }^x = \psi_{j, S}^x  .
    \end{equation}
   \end{lem}

   \begin{proof}
For the case $ j = 1 $, the subset $S$ must be $ \emptyset$ or $\{0\}$. In this case, the equality \eqref{Eq:psiIS} in fact 
 reduces to the definitions in Notation \ref{No:etax}, i.e., 
 \begin{equation}\label{Eq:lepsiIinfx}
      \psi^{x}_{I_{\infty}} = \tau - x,  
 \end{equation}
    and
\begin{align}\label{Eq:lepsiI0x}
      \psi^{x}_{I_{0}} = \tau - \delta x.  
\end{align}
To prove the case $ j\geqslant 2$, we take $x=u_{k,\mu}$.    According to Lemma \ref{Lem:Isophis}, it suffices to show that 
       \begin{equation}\label{Eq:psijS=II}
            \psi_{j,S}^{u_{k,\mu}} \mathbf{s}(u_{k,\mu}) = \mathbf{s} ( \sigma_0^{n } \sigma_\infty^{j - n}u_{k,\mu})=\mathbf{s} ( u_{k+j,\mu\cdot \eta_*^n}).
       \end{equation}   
From Equation \eqref{Eq:lepsiIinfx} and Lemma \ref{Lem:Isophis}, we have 
     \begin{equation}\label{Eq:tau-xs}
  (\tau - u_{k,\mu} ) \mathbf{s} (u_{k, \mu})   =  \mathbf{s} ( \sigma_{\infty}(u_{k, \mu}) )  =  \mathbf{s} (u_{k+1, \mu})
     \end{equation}
    and
    \begin{equation}\label{Eq:tau-dexs}
  (\tau - \delta u_{k,\mu}) \mathbf{s} (u_{k, \mu})  =\mathbf{s} (\sigma_0 (u_{k, \mu}) )=\mathbf{s} (u_{k+1, \mu \cdot \eta_* }) .
    \end{equation} 
It follows from Lemma \ref{Lem:psiSS'} that $\psi^x_{j,S}$ does not depend on the choice of  $S$.  We may assume that $S= \{ j-n ,\, j-n+1,\,  \cdots ,\,j-1 \}$ of cardinality $ n $.   
Substituting the equalities \eqref{Eq:tau-xs} and \eqref{Eq:tau-dexs} into \eqref{Eq:psiSj_u} gives
\[
\psi_{j,S}^{u_{k,\mu}} \mathbf{s}(u_{k,\mu})=     \mathbf{s}( u_{k+j,\mu\eta_*^n}),
\]
which completes the proof. 
   \end{proof}
 
\subsection{Hayes Modules over {$H^+$}}
In this section, we derive the explicit expression of Hayes $\A$-modules over $H^+$ with the help of the explicit expression $ \psi_{ I_0^{ n} I_\infty^{ j-n}  }^x$. Firstly, we note that the formulas in Corollary \ref{cor:gensu}  can be generalized in the following way.

   \begin{lem}\label{Lem:psis}
   With the notations in Notation \ref{No:su}, we obtain 
    \begin{equation}\label{Eq:lesiginfs}
     \frac{1}{t - \eta^{(k)} } * \mathbf{s} (u_{k, \mu}) 
 =    -\frac{\Theta^{\sigma^k}}{u_{k,\mu}} \mathbf{s}(
 u_{k+1,\mu})
    \end{equation}
 and  
 \begin{equation}\label{Eq:lesig0s}
  \frac{t}{\eta^{(k)} }  \frac{1}{ t - \eta^{(k)}  }    * \mathbf{s} (u_{k, \mu })
  =  -\frac{ \Theta^{\sigma^k}}{u_{k,\mu }} \mathbf{s} ( u_{k+1,\mu\eta_* }).
    \end{equation}
\end{lem}
\begin{proof}
Let $b_i$ be the same notation as in Notation \ref{Not:bi}. Then
     \begin{align*}
\left(\sum_{i=0}^{N-1} b_i \psi_{T_i}^{u}\right) \mathbf{s}(u)  &=  \left( \sum_{i=0}^{N-1} b_i T_i\right) *\mathbf{s}(u) \\
&=\frac{1}{ t - \eta  } * \mathbf{s}(u) \quad  \\
&= -\frac{\Theta}{u}(\tau - u) \mathbf{s}(u) \quad \text{ by Corollary \ref{cor:gensu}},
\end{align*}
i.e.,
 \begin{equation}\label{Eq:psiThe}
\sum_{i=0}^{N-1} b_i \psi_{T_i}^{u}   = -\frac{\Theta}{u}(\tau - u) .
\end{equation}
    Applying $ \sigma_{{\infty}}^k M_{\mu} $ to both sides of Equation \eqref{Eq:psiThe} we have 
    \begin{equation}\label{Eq:siginfMpsiThe}  
  \sum_{i=0}^{N-1}   b_i^{\sigma^k} \sigma_{{\infty}}^k M_{\mu}(\psi_{T_i}^{u}) =  -\sigma_{{\infty}}^k M_{\mu}\left(\frac{\Theta}{u}(\tau - u)\right)= -\frac{\Theta^{\sigma^k}}{u_{k,\mu}} (\tau - u_{k,\mu}) .       
    \end{equation}
   Note that by definition $ \sigma \psi^{u} =\psi^{\sigma u } $ (see (3) in Notation \ref{No:etax}). The first term in \eqref{Eq:siginfMpsiThe} equals 
 \begin{equation*}
       \sum_{i=0}^{N-1}   b_i^{\sigma^k}  \psi_{T_i}^{\sigma_{{\infty}}^k M_{\mu}(u)}   =\sum_{i=0}^{N-1} b_i^{\sigma^k}    \psi_{T_i}^{u_{k,\mu}} .
\end{equation*}
This implies that \eqref{Eq:siginfMpsiThe} is equivalent to 
\begin{equation}\label{Eq:siginfpsi}
   \sum_{i=0}^{N-1} b_i^{\sigma^k}    \psi_{T_i}^{u_{k,\mu}}= -\frac{\Theta^{\sigma^k}}{u_{k,\mu}} (\tau - u_{k,\mu}) .
\end{equation}
Moreover, we have
\begin{align*}
 \nonumber  \frac{1}{ t - \eta^{(k)} } * \mathbf{s}(u_{k,\mu})&= \sigma^k\left(\frac{1}{ t - \eta }\right) * \mathbf{s}(u_{k,\mu})\\
 &=\left(\sum_{i=0}^{N-1} b_i^{\sigma^k} T_i\right) *\mathbf{s}(u_{k,\mu}) \quad \text{by Notation \ref{Not:bi}}\\
 &= \sum_{i=0}^{N-1} b_i^{\sigma^k}   \psi_{T_i}^{u_{k,\mu}}\mathbf{s}(u_{k,\mu})\\
&=
-\frac{\Theta^{\sigma^k}}{u_{k,\mu}} (\tau - u_{k,\mu})\mathbf{s}(u_{k,\mu}) \quad \text{by Equation \eqref{Eq:siginfpsi} }\\
&=-\frac{\Theta^{\sigma^k}}{u_{k,\mu}} \mathbf{s}(u_{k+1,\mu})  \quad \text{by Equation \eqref{Eq:tau-xs}.}
\end{align*}
   Therefore,  Equation \eqref{Eq:lesiginfs} holds.

In the following, 
we proceed analogously to establish Equation \eqref{Eq:lesig0s}. 
Analogous to Equation \eqref{Eq:psiThe}, we have 
 \begin{equation}\label{Eq:psiThe2}
\frac{t}{\eta}\sum_{i=0}^{N-1} b_i \psi_{T_i}^{u}   = -\frac{\Theta}{u}(\tau -\frac{\theta}{\eta} u) .
\end{equation}
  Applying $ \sigma_{0}^kM_{\frac{\mu}{\eta_*^k}} $ to both sides of Equation \eqref{Eq:psiThe2} we have 
\begin{equation}\label{Eq:sig0MpsiThe}  
 \frac{t}{\eta^{(k)}} \sum_{i=0}^{N-1} b_i^{\sigma^k} \sigma_{0}^k M_{\frac{\mu}{\eta_*^k}}(\psi_{T_i}^{u}) =  -\sigma_{0}^k M_{\frac{\mu}{\eta_*^k}}\left(\frac{\Theta}{u}(\tau - \frac{\theta}{\eta}u)\right).       
    \end{equation}
  Expand both sides of  \eqref{Eq:sig0MpsiThe} to obtain 
 \begin{align}\label{Eq:sig0psi}
  \nonumber  \frac{t}{\eta^{(k)}}\sum_{i=0}^{N-1} b_i^{\sigma^k}   \psi_{T_i}^{u_{k,\mu}} & =  \frac{t}{\eta^{(k)}}
  \sum_{i=0}^{N-1} b_i^{\sigma^k}  \psi_{T_i}^{\sigma_{0}^k M_{\frac{\mu}{\eta_*^k}}(u)} \quad \text{by Notation \ref{No:ukmusig}}\\\nonumber
  &=  \frac{t}{\eta^{(k)}}\sum_{i=0}^{N-1} b_i^{\sigma^k} \sigma_{0}^k M_{\frac{\mu}{\eta_*^k}}(\psi_{T_i}^{u})  \quad \text{by Notations \ref{No:etax}} \\\nonumber
  &=-\sigma_{0}^k M_{\frac{\mu}{\eta_*^k}} \left(\frac{\Theta}{u}(\tau - \frac{\theta}{\eta}u) \right) \quad \text{by Equation \eqref{Eq:sig0MpsiThe}} \\
 &= -\frac{\Theta^{\sigma^k}}{u_{k,\mu}} (\tau - \delta (u_{k,\mu}) ) \quad \text{by Notations \ref{No:delta_x}}.
\end{align}
Moreover, we have
\begin{align*}
 \nonumber  \frac{t}{\eta^{(k)} (t - \eta^{(k)}) } * \mathbf{s}(u_{k,\mu})&= \sigma^k\left(\frac{t}{\eta(t - \eta) }\right) * \mathbf{s}(u_{k,\mu})\\
 &=\sigma^k\left(\frac{t}{\eta}\sum_{i=0}^{N-1} b_i T_i\right) *\mathbf{s}(u_{k,\mu})\\
 &= \frac{t}{\eta^{(k)}}\sum_{i=0}^{N-1} b_i^{\sigma^k}  \psi_{T_i}^{u_{k,\mu}}\cdot\mathbf{s}(u_{k,\mu})\\
&=
-\frac{\Theta^{\sigma^k}}{u_{k,\mu}} (\tau -\delta( u_{k,\mu}) )\mathbf{s}(u_{k,\mu}) \quad \text{by Equation \eqref{Eq:sig0psi}  }\\
&=-\frac{\Theta^{\sigma^k}}{u_{k,\mu}} \mathbf{s}(u_{k+1,\mu\eta_*})  \quad \text{by Equation \eqref{Eq:tau-dexs} }.
\end{align*}
      Therefore,   Equation \eqref{Eq:lesig0s} holds. 
\end{proof}

Combining the equalities in the previous lemma, we obtain a general formula for the Anderson $\A$-motive $M_{\psi^x}$. 

\begin{lem}\label{lem:psi0infsu}
     Let $\mathbf{s}(u_{k,\mu})$ be as in Notation \ref{No:su}. Then
  \begin{equation}\label{Eq:psi0infsu}
    \frac{t^n}{(t-\eta^{(j-1)})(t-\eta^{(j-2)})\cdots(t-\eta^{(0)})} * \mathbf{s}(u_{0,\mu})
    = \eta^{n q^{j-1}} \cdot \frac{(-\Theta)^{\gamma_{j}(\sigma)}}{(\mu u)^{W_j}} \mathbf{s}(u_{j,\mu\eta_*^n}).
  \end{equation}
\end{lem}

\begin{proof}
Applying Equations \eqref{Eq:lesiginfs} and \eqref{Eq:lesig0s} recursively, we have   
      \begin{align}
        &\frac{t^n}{(t- \eta^{(j-1)})(t- \eta^{(j-2)})\cdots (t- \eta^{(0)})}  * \mathbf{s}({u_{0,\mu}})  \nonumber \\  
       ={}& \frac{1}{(t-\eta^{(j-1)})\cdots(t-\eta^{(n)})}*\frac{t^n}{(t- \eta^{(n-1)})(t- \eta^{(n-2)})\cdots (t- \eta^{(0)})}  * \mathbf{s}({u_{0,\mu}})   \nonumber \\
       ={}&\frac{1}{(t-\eta^{(j-1)})\cdots(t-\eta^{(n)})}*\Big(\eta^{W_n}\frac{(-\Theta)^{1+\sigma+\cdots+\sigma^{n-1}}}{u_{0,\mu}u_{1,\mu\eta_*}\cdots u_{n-1,\mu\eta_*^{n-1}}}\mathbf{s}({u_{n,\mu\eta_*^n}}) \Big) \nonumber \\
       ={}&\eta^{W_n}\frac{(-\Theta)^{1+\sigma+\cdots+\sigma^{n-1}}}{u_{0,\mu}u_{1,\mu\eta_*}\cdots u_{n-1,\mu\eta_*^{n-1}}}\Big(\frac{1}{(t-\eta^{(j-1)})\cdots(t-\eta^{(n)})}*\mathbf{s}({u_{n,\mu\eta_*^n}}) \Big)   \nonumber \\
       ={}&\eta^{W_n}\frac{(-\Theta)^{1+\sigma+\cdots+\sigma^{n-1}}}{u_{0,\mu}u_{1,\mu\eta_*}\cdots u_{n-1,\mu\eta_*^{n-1}} }\Big(\frac{1}{(t-\eta^{(j-1)})\cdots(t-\eta^{(n+1)})}*\frac{(-\Theta)^{\sigma^n}}{u_{n,\mu\eta_*^n}}\mathbf{s}({u_{n+1,\mu\eta_*^n}})\Big)
         \nonumber \\
       ={}&\eta^{W_n}\frac{(-\Theta)^{1+\sigma+\cdots+\sigma^{n}}}{u_{0,\mu}u_{1,\mu\eta_*}\cdots u_{n-1,\mu\eta_*^{n-1}} u_{n,\mu\eta_*^{n}}}\Big(\frac{1}{(t-\eta^{(j-1)})\cdots(t-\eta^{(n+1)})}*\mathbf{s}({u_{n+1,\mu\eta_*^n}})\Big)\nonumber \\
       ={}&\eta^{W_n}\frac{(-\Theta)^{1+\sigma+\cdots+\sigma^{j-1}}}{U}\mathbf{s}({u_{j,\mu\eta_*^n}}), \label{Eq:aligntn}
     \end{align}
     where $U$
 denotes the product:
     \[
     U = {u_{0,\mu}u_{1,\mu\eta_*}\cdots u_{n-1,\mu\eta_*^{n-1}} u_{n,\mu\eta_*^{n}}\cdots u_{j-1,\mu\eta_*^n}}.
     \]
    By Equation \eqref{Eq:ukThe},
     \[
     u_{j,\mu\eta_*^n}=(\mu\eta_*^n)^{q^j}\sigma_\infty^j(u)=(\mu\eta_*^n)^{q^j}\Theta^{(1-q)\gamma_j(\sigma)}u^{q^j}.
     \]
       It follows that
     \[
     u_{0,\mu}u_{1,\mu\eta_*}\cdots u_{n-1,\mu\eta_*^{n-1}}=  (\mu)^{q^0} (\mu \eta_*)^{q^1} (\mu \eta_*^2 )^{q^2} \cdots (\mu \eta_{*}^{n-1})^{q^{n-1}} \Theta^{(1-q) \sum_{i=0}^{n-1} \gamma_{i}(\sigma) } u^{\sum_{i=0}^{n-1} q^i }
     \]
     and
     \[
      u_{n,\mu\eta_*^{n}}\cdots u_{j-1,\mu\eta_*^n} = (\mu \eta_*^n)^{q^n} \cdots (\mu \eta_*^n)^{q^{j-1}}\Theta^{(1-q) \sum_{i=n}^{j-1} \gamma_{i}(\sigma) } u^{\sum_{i=n}^{j-1} q^i }.
     \]
 Using the equality 
    \[
    \sum_{i=1}^{n-1} i q^i = \frac{ q - n q^n + (n-1) q^{n+1}}{(1-q)^2},
    \]
    we obtain 
    \begin{align*}
      U &= u_{0,\mu}u_{1,\mu\eta_*}\cdots u_{n-1,\mu\eta_*^{n-1}}  \cdot u_{n,\mu\eta_*^{n}}\cdots u_{j-1,\mu\eta_*^n} \\
     & = (\mu u)^{W_j} \eta_*^{\sum_{i=1}^{n-1} i q^i  + n q^n W_{j-n} } \Theta^{(1-q) \sum_{i=0}^{j-1} \gamma_i (\sigma)} \\
     & = (\mu u)^{W_j} \eta_*^{\frac{ q - n q^n + (n-1) q^{n+1}}{(1-q)^2} +  \frac{n q^{j}- n q^n}{q-1} } \Theta^{(1-q) \sum_{i=0}^{j-1} \gamma_i (\sigma)}.
     \end{align*}
     Substituting $U$ into Equation \eqref{Eq:aligntn} yields
     \begin{align*} 
      &\frac{t^n}{(t- \eta^{(j-1)})(t- \eta^{(j-2)})\cdots (t- \eta^{(0)})}  * \mathbf{s}({u_{0,\mu}})\\  
      = &\eta^{W_n }\cdot \eta_*^{ - \frac{ q - n q^n + (n-1) q^{n+1}}{(1-q)^2} -  \frac{n q^{j}- n q^n}{q-1} } \cdot \frac{(-\Theta)^{1+\sigma+\cdots+\sigma^{j-1}}}{ ( \mu u)^{W_j}  \Theta^{(1-q) \sum_{i=0}^{j-1} \gamma_i (\sigma)} }  \mathbf{s}({u_{j,\mu\eta_*^n}}) \\
      = &\eta^{W_n }\cdot \eta ^{ - \frac{ 1 - n q^{n-1} + (n-1) q^{n}}{(1-q)} +  n q^{j-1}- n q^{n-1} } \cdot \frac{ (-1)^{j} (\Theta)^{ \gamma_{j}(\sigma)  }}{ (  \mu u)^{W_j}  }  \mathbf{s}({u_{j,\mu\eta_*^n}})\\  
      = &\eta^{ n q^{j-1}   } \cdot \frac{ (-\Theta)^{ \gamma_{j}(\sigma)  }}{ ( \mu u)^{W_j}  }  \mathbf{s}({u_{j,\mu\eta_*^n}}) .
      \end{align*}
 Here we make use of the equality
    \begin{equation}\label{Eq:sum1-qgam}
      (1-q) \sum_{i=0}^{j-1} \gamma_i (\sigma)=-\sum_{i=0}^{j-1}(q^{j-i-1}-1)\sigma^i=-\sum_{i=0}^{j-1}q^{i}\sigma^{j-1-i}+\sum_{i=0}^{j-1}\sigma^i.   
    \end{equation}
  This completes the proof.
\end{proof}

  \begin{cor}\label{Cor:rhotspsiSs}
 For $j\geqslant n$, the following equality holds:
      \begin{align}\label{Eq:cosuj}
 \frac{t^n}{(t- \eta^{(j-1)})(t- \eta^{(j-2)})\cdots (t- \eta^{(0)})}  * \mathbf{s}(u) 
   =\eta^{ nq^{j-1}   } \cdot \frac{ (-\Theta)^{ \gamma_{j}(\sigma)  }}{  u^{W_j}  }   \psi_{j,S}^u \, \mathbf{s}(u) ,
   \end{align}
    for any subset $ S \subseteq \{ 0, \cdots, j - 1 \} $ of cardinality $ n $. 
  \end{cor}
  \begin{proof}
Substituting $ \mu=1 $ into Lemma \ref{lem:psi0infsu} yields 
      \begin{align*}
   &  \frac{t^n}{(t- \eta^{(j-1)})(t- \eta^{(j-2)})\cdots (t- \eta^{(0)})}  * \mathbf{s}(u) =\eta^{ nq^{j-1}   } \cdot \frac{ (-\Theta)^{ \gamma_{j}(\sigma)  }}{  u^{W_j}  }   \mathbf{s}({u_{j,\eta_*^n}}) \\
    ={}&\eta^{ nq^{j-1}   } \cdot \frac{ (-\Theta)^{ \gamma_{j}(\sigma)  }}{  u^{W_j}  }   \psi_{I_0^n I_\infty^{j-n}}^u \mathbf{s}(u)\qquad \text{by Lemma \ref{Lem:Isophis}  }\\
   ={}&\eta^{ nq^{j-1}   } \cdot \frac{ (-\Theta)^{ \gamma_{j}(\sigma)  }}{  u^{W_j}  }   \psi_{j,S}^u\,\mathbf{s}(u) \qquad \text{by Lemma \ref{Lem:psiIS}},
   \end{align*}
    for any subset $ S \subseteq \{ 0, \cdots, j - 1 \} $ of cardinality $ n $.  
  \end{proof}
  
 We now give an explicit description of $\psi^u$. 

  \begin{thm}\label{Thm:Hmopsi}
The Hayes $\A$-module $ \psi^u $ over $H^+ $ with  $ \ker\psi_{I_\infty}^u = \tau - u$
is given by  
\begin{align}\label{Eq:thpsiTNS}
     \psi_{T_n}^u=  \eta^{n q^{-1} } \psi_{N,S}^{u}
\end{align}
for an arbitrary subset $S$ contained in $\{0,1, \cdots, N-1\}$ of cardinality $n$. 
More precisely, we have 
 \begin{align}\label{Eq:thpsiTn}
    \psi_{T_n}^u =\eta^{n q^{-1} } (\tau-\FS{u}^{N-1}_{N-n} )\circ\cdots\circ(\tau-\FS{u}^{1}_{N-n} )\circ (\tau-\FS{u}^0_{N-n} ),
\end{align}
    where $ \FS{u}^{i}_{l} $ is given by Lemma \ref{Lem:exFSx}. 
\end{thm}
\begin{proof}
Recall $T_n = \frac{t^n}{\rho(t)}$ and $ u = u_{0,1}$. By the definition of Anderson motive, we have 
   \begin{align*}
  \psi_{T_n}^u \mathbf{s}(u) ={}& \frac{t^n}{\rho(t)}  * \mathbf{s}(u) \qquad \text{by Definition \ref{Def:t-m}}  \\
    ={}&\eta^{n q^{-1} } \psi_{N,S}^{u} \, \mathbf{s}(u )\qquad \text{by Corollary \ref{Cor:rhotspsiSs} and Equation \eqref{Eq:PuTh}},
    \end{align*}
    for any subset $ S \subseteq \{ 0, \cdots, N - 1 \} $ of cardinality $ n $. This implies the equality \eqref{Eq:thpsiTNS}.
    
  Now we fix $ S = \{ N-n, \cdots, N-1 \} $. Then Equation \eqref{Eq:thpsiTn} follows immediately by setting $j=N$ and $x=u$ in Equation \eqref{Eq:psiSj}. 
\end{proof}
Note that Theorem \ref{Thm:Hmopsi}, together with Equation \eqref{Eq:psiSj_u}, yields part (2) of Theorem \ref{Thm:expressions}.

\subsection{Standard Model of Drinfeld Modules}
The purpose of this section is to derive the standard model $\Psi$ of Drinfeld $\A$-modules over the Hilbert class field $H$.
As in the case of Hayes $\A$-modules,  we explicitly describe the Anderson $\A$-motive of $\Psi$.
\begin{notation}\label{No:Thetakj}
  Let $\FS{u}^{i}_{l}$ be the bracket as in Notation \ref{No:Gammai}. Let $\ell$ be a $(q-1)$-th root of $-\frac{u}{\Theta}$.  For $i,l\in \mathbb{Z}_{\geqslant 0}$, we define $\FT{\Theta}_{l}^{i}\in H$ by
\begin{equation}\label{Eq:GauThe}
 \ell^{-q^{i+1}}(\tau - \FS{u}^{i}_{l}  ) \ell^{q^i}= \tau+\FT{\Theta}^i_{l}.
\end{equation}
Explicitly, these quantities are given by
\begin{equation}\label{Eq:noGaThe}
  \FT{\Theta}_{l}^{i}=  
\begin{cases}
 \Theta^{ (1-q)\gamma_i(\sigma)} \Theta^{q^i},  \quad& \text{for } 0\leqslant i<l;\\
  \theta\eta^{ q^{i-1}(1-q)(i-l) - q^i} \Theta^{ (1-q)\gamma_i(\sigma)}\Theta^{q^i}, \quad & \text{for } i\geqslant l.
    \end{cases}
\end{equation}
\end{notation}
The following proposition is analogous to Corollary \ref{Cor:rhotspsiSs}.
\begin{prop}\label{Pro:sPsiThe}
 Let $ \mathbf{s}_0 $ be the formal generator of the Anderson motive $ M_{\Psi }$. For $j\geqslant n$,  we have
  \begin{align}\label{Eq:prorhoPsis}
  \nonumber&\frac{t^n}{(t - \eta^{(0)})(t - \eta^{(1)}) \cdots (t - \eta^{(j-1)})} * \mathbf{s}_0 \\
  &=\eta^{ nq^{j-1}}\Theta^{ \gamma_{j}(\sigma)  -W_j  }
  \cdot \left(\tau+\FT{\Theta}^{j-1}_{j-n}\right)\circ\cdots\circ\left(\tau+\FT{\Theta}^{1}_{j-n} \right)\circ \left(\tau+\FT{\Theta}^0_{j-n}\right)\mathbf{s}_{0}.
  \end{align}
 
\end{prop}

\begin{proof}
     As before, we let $\ell$ be an isogeny from $\Psi$ to $\psi^{u}$   with $\ell^{1-q}=-\frac{\Theta}{u}$. Then the formal generators $ \mathbf{s}_0 $ and $ \mathbf{s}(u) $ satisfy 
     \[
     \mathbf{s}(u)=\ell \mathbf{s}_0.
     \]
Therefore,  for any subset $ S \subseteq \{ 0, \cdots, j - 1 \} $ of cardinality $ n $, we have 
\begin{align}
 \mathbf{s}_{n,j}:= &\frac{t^n}{(t - \eta^{(0)})(t - \eta^{(1)}) \cdots (t - \eta^{(j-1)})} *  \mathbf{s}_0  \nonumber \\ 
  &=\frac{\ell^{-1} t^n}{(t - \eta^{(0)})(t - \eta^{(1)}) \cdots (t - \eta^{(j-1)})} *  \ell \mathbf{s}_0  \nonumber\\ 
  &=\frac{\ell^{-1} t^n}{(t- \eta^{(j-1)})(t- \eta^{(j-2)})\cdots (t- \eta^{(0)})}  * \mathbf{s}(u)  \qquad \text{by $\mathbf{s}(u)=\ell \mathbf{s}_0$ } \label{Eq:pro271}  \nonumber \\ 
    &=\eta^{ nq^{j-1}  }  \ell^{-1} \frac{ (-\Theta)^{ \gamma_{j}(\sigma)  }}{  u^{W_j}  } \psi_{j,S}^u \mathbf{s}(u)\qquad\text{by Corollary \ref{Cor:rhotspsiSs}}.    
\end{align}
From Equations \eqref{Eq:psiSj} and \eqref{Eq:GauThe}, it follows that 
  \begin{align}
  \mathbf{s}_{n,j}&=\eta^{ nq^{j-1}  }  \ell^{-1} \frac{ (-\Theta)^{ \gamma_{j}(\sigma)  }}{  u^{W_j}  } (\tau-\FS{u}^{j-1}_{j-n} )\circ\cdots\circ(\tau-\FS{u}^{1}_{j-n} )\circ (\tau-\FS{u}^0_{j-n} )\ell \mathbf{s}_0  \nonumber \\ 
  &=\eta^{ nq^{j-1}  }  \ell^{q^j-1} \frac{ (-\Theta)^{ \gamma_{j}(\sigma)  }}{  u^{W_j}  }\ell^{-q^j}(\tau-\FS{u}^{j-1}_{j-n} )\ell^{q^{j-1}}\cdots \ell^{-q^{2}}(\tau-\FS{u}^{1}_{j-n} )\ell^q\circ \ell^{-q}(\tau-\FS{u}^{0}_{j-n} ) \ell \mathbf{s}_0 \nonumber \\ 
  &= \eta^{ nq^{j-1}   } \frac{ (-\Theta)^{ \gamma_{j}(\sigma)  }}{  u^{W_j}  }\ell^{q^j-1}  \left(\tau+\FT{\Theta}^{j-1}_{j-n}\right)\circ\cdots\circ\left(\tau+\FT{\Theta}^{1}_{j-n} \right)\circ \left(\tau+\FT{\Theta}^0_{j-n}\right)\mathbf{s}_0  .
\end{align}
Applying the equality $(\ell^{1-q})^{-W_j} =\Big(-\frac{\Theta}{u}\Big)^{-W_j}$, we obtain Equation \eqref{Eq:prorhoPsis}.  
\end{proof}

 \begin{thm}\label{Thm:DMPsi}
 With the above notation, the Drinfeld module $\Psi$ is given by
 \begin{equation}\label{Eq:thPsiTn}
    \Psi_{T_n}  =\eta^{n q^{-1} } \Theta^{\gamma_N(\sigma) - W_N}\left(\tau+\FT{\Theta}^{N-1}_{N-n}\right)\circ\cdots\circ\left(\tau+\FT{\Theta}^{1}_{N-n} \right)\circ \left(\tau+\FT{\Theta}^0_{N-n}\right),
\end{equation}
for $0\leqslant n\leqslant N-1$.
\end{thm}

\begin{proof}
    By Definition \ref{Def:PL}, 
    \begin{align*}
    \frac{t^n}{(t - \eta^{(0)})(t - \eta^{(1)}) \cdots (t - \eta^{(N-1)})} *  \mathbf{s}_0
   = T_n *\mathbf{s}_0 =  \Psi_{T_n} \cdot \mathbf{s}_0.   
\end{align*}
Therefore, the theorem follows immediately from Proposition \ref{Pro:sPsiThe}.
\end{proof}
In analogy with Lemma \ref{Lem:psiIS}, we have the following corollary.

\begin{cor}
    The annihilator of $ I_0^{ n}I_\infty^{ j-n} $ is given by 
     \begin{equation}\label{Eq:lePsiI0Iinf}
    \Psi_{ I_0^{ n}I_\infty^{ j-n} }= \left(\tau+\FT{\Theta}^{j-1}_{j-n}\right)\circ\cdots\circ\left(\tau+\FT{\Theta}^{1}_{j-n} \right)\circ \left(\tau+\FT{\Theta}^0_{j-n}\right).
    \end{equation} 
\end{cor}

\begin{proof}
 
  We fix $ S = \{ j-n, \cdots, j-1 \} $.
 The expression of $\psi_{j,S}^u$ in Equation \eqref{Eq:psiSj} together with Lemma \ref{Lem:psiIS} yields 
 \begin{align}\label{Eq:psiI0Iinf}
\psi_{ I_0^{ n}I_\infty^{ j-n} }^u= \psi_{j,S}^u=
(\tau-\FS{u}^{j-1}_{j-n} )\circ\cdots\circ(\tau-\FS{u}^{1}_{j-n} )\circ (\tau-\FS{u}^0_{j-n} ).
\end{align}
 Similar to Equation \eqref{Eq:proPpsi} in Proposition \ref{Pro:Hu}, we have 
 \[
\psi_{ I_0^{ n}I_\infty^{ j-n} }^u=\ell^{q^j} \Psi_{I_0^{ n}I_\infty^{ j-n} }\ell^{-1}.
    \]
   Therefore, Equation \eqref{Eq:lePsiI0Iinf} follows from Equations \eqref{Eq:GauThe} and \eqref{Eq:psiI0Iinf}.
\end{proof}

\subsection{Examples}
In this section, we present explicit expressions for Drinfeld $\A$-modules for the cases $ N =2 $ and $ N =3$. 
\begin{example}\label{Exam:Psi2}
Let $ N =2 $. Applying Theorem \ref{Thm:DMPsi},  the rank-one Drinfeld module $\Psi$ of type $\eta^{(-1)}$ is given by 
\begin{align*}
    \Psi_{T_0} ={}&  \Theta^{\sigma -1} \tau ^{2} + (\Theta^{\sigma}+\Theta ^{q-1+\sigma})
\tau +\Theta^{1 + \sigma}; 
   \\
    \Psi_{T_1} ={}&\eta^{(1)} \Theta^{\sigma -1} \tau ^{2} +(\theta \Theta^{\sigma}+\eta^{(1)} \Theta^{q-1+
\sigma} )\tau +\theta \Theta^{1 + \sigma}. 
\end{align*}
\end{example}

\begin{example}\label{Exam:psi2}
    Let $ N =2 $. 
 By Theorem \ref{Thm:Hmopsi}, the rank-one Hayes $\A$-module parameterized by $u$ is given by
    \begin{align*}\label{Eq:psiExp}
\psi _{T_0}^u &= (\tau - \sigma_{\infty} u )(\tau - u ) = \tau^{2} - ( \Theta^{1+ \sigma}u^{-1}+u^{q} )\tau + \Theta^{1 + \sigma} ;\\
\psi _{T_1}^u &= \eta^{(1)} \left(\tau - \theta \Theta^{ (1-q)\gamma_{1}{(\sigma)}}   \left(\frac{ u }{\eta} \right)^{q }\right) (\tau - u ) = \eta^{(1)}\tau^{2}- (\theta \Theta^{1+\sigma}u^{-1}+\eta^{(1)} u^{q})\tau + \theta \Theta^{1 + \sigma}.
\end{align*}
\end{example}
 Examples \ref{Exam:Psi2} and \ref{Exam:psi2} coincide with our previous results in \cite{HH24}.
\begin{example}
Let $ N =3 $. Applying Theorem \ref{Thm:DMPsi},   the rank-one Drinfeld module $\Psi$ of type $\eta^{(-1)}$ is given by 
\begin{align*}
    \Psi_{T_0} ={}& \Theta^{\sigma^2+q\sigma-q-1} \tau^3 +\big(\Theta^{\sigma^2+q\sigma+q^2-q-1}+\Theta^{-1+\sigma^2+q\sigma}+\Theta^{\sigma^2+\sigma-1}\big)\tau^2\\
    &+\big(\Theta^{\sigma^2+q\sigma+q-1}+\Theta^{\sigma^2+\sigma+q-1}+\Theta^{\sigma^2+\sigma}\big)\tau+ \Theta^{1 + \sigma+\sigma^2};\\
\Psi_{T_1} = & \eta^{q^2}\Theta^{\sigma^2+q\sigma-q-1} \tau^3 +\big(\eta^{q^2}\Theta^{\sigma^2+q\sigma+q^2-q-1}+\eta^{q^2}\Theta^{-1+\sigma^2+q\sigma}+\theta\Theta^{\sigma^2+\sigma-1}\big)\tau^2\\
    &+\big(\eta^{q^2}\Theta^{\sigma^2+q\sigma+q-1}+\theta\Theta^{\sigma^2+\sigma+q-1}+\theta \Theta^{\sigma^2+\sigma}\big)\tau+\theta \Theta^{1 + \sigma+\sigma^2}   
;\\
    \Psi_{T_2} ={}& \eta^{2q^2}\Theta^{\sigma^2+q\sigma-q-1} \tau^3 +\big(\eta^{2q^2}\Theta^{\sigma^2+q\sigma+q^2-q-1}+\eta^{q^2}\theta^{q}\Theta^{-1+\sigma^2+q\sigma}+\eta^{(1)}\theta \Theta^{\sigma^2+\sigma-1}\big)\tau^2\\
&+\big(\eta^{q^2}\theta^{q}\Theta^{\sigma^2+q\sigma+q-1}+\eta^{(1)}\theta \Theta^{\sigma^2+\sigma+q-1}+\theta^2 \Theta^{\sigma^2+\sigma}\big)\tau+\theta^2 \Theta^{1 + \sigma+\sigma^2} .
 \end{align*}

\end{example}

\begin{example}
Let $ N =3 $.  Theorem \ref{Thm:Hmopsi} implies that the rank-one Hayes $\A$-module parameterized by $ u $ is given by 
\begin{align*}
    \psi_{T_0}^u ={}&  \tau^3 +\frac{-u^{q^2+q+1}+\Theta^{\sigma^2+q\sigma+q}+\Theta^{\sigma^2+\sigma+q}}{u^{1+q}}\tau^2\\
&-\frac{\Theta^{\sigma^2+q\sigma+q}+\Theta^{\sigma^2+\sigma+q}+\Theta^{\sigma^2+\sigma+1}}{u}\tau+ \Theta^{1 + \sigma+\sigma^2};\\
          \psi_{T_1}^u = & \eta^{q^2} \tau^3 +\frac{-\eta^{q^2} u^{q^2+q+1}+ \eta^{q^2} \Theta^{\sigma^2+q\sigma+q}+ \theta \Theta^{\sigma^2+\sigma+q}}{u^{1+q}}\tau^2\\
    &-\frac{\eta^{q^2}\Theta^{\sigma^2+q\sigma+q}+\theta \Theta^{\sigma^2+\sigma+q}+\theta \Theta^{\sigma^2+\sigma+1}}{u}\tau+\theta \Theta^{1 + \sigma+\sigma^2};   
\\
    \psi_{T_2}^u ={}& \eta^{2q^2}\tau^3 +\frac{-\eta^{2q^2} u^{q^2+q+1}+ \eta^{q^2} \theta^q \Theta^{\sigma^2+q\sigma+q}+ \eta^{(1)} \theta \Theta^{\sigma^2+\sigma+q}}{u^{1+q}}\tau^2\\
&-\frac{\eta^{q^2} \theta^q \Theta^{\sigma^2+q\sigma+q}+\eta^{(1)} \theta  \Theta^{\sigma^2+\sigma+q}+\theta^2 \Theta^{\sigma^2+\sigma+1}}{u}\tau+\theta^2 \Theta^{1 + \sigma+\sigma^2} .
 \end{align*}
Note that the coefficients lie in the narrow class field $ H^+ = K(\eta,u) $. 
\end{example}

\section{Isogenies and Period Lattices of Drinfeld Modules}\label{Sec:IPL}
\subsection{The Period Lattices of Drinfeld Modules}
This section is devoted to computing the period lattices of $\Psi^{(j)}$.  
Recall that in Theorem \ref{Th:kenelpitil}, we have described the period lattices of $ \Psi^{(2)}$. Once we establish the isogeny relation between $ \Psi^{(j)}$ and $ \Psi^{(2)}$, the corresponding period lattices of $ \Psi^{(j)}$ can be obtained easily.
\begin{prop}\label{Pro:isoPsiij}
 For $ j > i $, we set 
 \begin{equation}\label{Eq:coCj}
     C_{i,j}= \sqrt[q-1]{\Theta^{\sigma^j -\sigma^i }} \cdot\Theta^{\sigma^i (\gamma_{j-i}(\sigma)-  W_{j-i})},
     \end{equation}
     where $\sqrt[q-1]{\Theta^{\sigma^j -\sigma^i }}$ denotes a $ (q-1) $-th root of $ \Theta^{\sigma^j-\sigma^i} $.
 Then the twisted polynomial 
  \begin{equation}\label{Eq:prolambda}
      \lambda:= C_{i,j} \left(\tau + \sigma^i\FT{\Theta}^{j-i-1}_{j-i} \right)\circ \cdots \circ  \left(\tau + \sigma^i\FT{\Theta}^{1}_{j-i}  \right)\circ\left(\tau + \sigma^i\FT{\Theta}^{0}_{j-i} \right)
  \end{equation}
    is an isogeny from $ \Psi^{(i)}$ to $\Psi^{(j)} $.
 \end{prop}

\begin{proof}
Let $\ell_i$ (resp. $\ell_j$) be a  $(q-1)$-th root of $- \frac{\sigma^{i}_\infty u}{ \sigma_{\infty}^i \Theta}$ (resp. $- \frac{\sigma_{\infty}^j u}{ \sigma_{\infty}^j \Theta}$).
 Similar to Equation \eqref{Eq:proell} in Proposition \ref{Pro:Hu}, it follows that $\ell_i$ and $\ell_j$ represent the isogenies
 \[
 \ell_i: ~~\Psi^{(i)} \to \psi^{ u_{i,1}}
 \]
 and 
  \[
 \ell_j: ~~\Psi^{(j)} \to \psi^{u_{j,1}}
 \]
respectively. 
    It follows from Lemma \ref{Lem:Isophis}  that  $\psi^{u_{i,1}}_{I_\infty^{j-i}}$ is an isogeny from  $\psi^{u_{i,1}}$ to $\psi^{ u_{j,1}}$. Therefore, we have the relations of isogenies:
    \[
    \begin{tikzcd}
        \psi^{u_{i,1}}\ar[r, "\psi^{ u_{i,1}}_{I_\infty^{j-i}}"]  & \psi^{ u_{j,1} } \\
        \Psi^{(i)} \ar[u, "\ell_i"]\ar[r , "\lambda"]  & \Psi^{(j)} \ar[u, "\ell_j"] \, ,\\
    \end{tikzcd}
    \]
for some twisted polynomial $ \lambda $. 
By combining Equation \eqref{Eq:psiSj} and Lemma \ref{Lem:psiIS}, we have   
    \begin{align*}
  \psi^{u_{i,1}}_{I_\infty^{j-i}}=\psi_{j-i,S}^{u_{i,1}}=  (\tau -  \FS{u_{i,1}}_{j-i}^{j-i-1}  )\circ \cdots\circ (\tau - \FS{u_{i,1}}_{j-i}^{1}  ) \circ (\tau - \FS{u_{i,1}}_{j-i}^{0}  ),     
    \end{align*}
    where $S=\emptyset$.
   Therefore, the isogeny $ \lambda $ is given by 
    \begin{align*}
    \lambda =  \ell^{-1}_j\psi^{u_{i,1}}_{I_\infty^{j-i}}\ell_i  =  \ell_j^{-1}(\tau -  \FS{u_{i,1}}_{j-i}^{j-i-1}  )\circ \cdots\circ (\tau - \FS{u_{i,1}}_{j-i}^{1}  ) \circ (\tau - \FS{u_{i,1}}_{j-i}^{0}  ) \ell_i .    
    \end{align*}
Equation \eqref{Eq:FSukmu} in Notation \ref{No:Gammai} yields 
    \begin{align}\label{Eq:laisoPiPj}
   \nonumber \lambda ={}&  \ell_j^{-1} \ell_i^{q^{j-i}}
\ell_i^{-q^{j-i}}(\tau - \Theta^{\sigma^i(1-q)\gamma_{j-i-1}(\sigma) } u_{i,1} ^{q^{j-i-1} }  )\ell_i^{q^{j-i-1}}\circ\\
&\cdots\circ \ell_i^{q^2}(\tau - \Theta^{\sigma^i(1-q)\gamma_{1}(\sigma) } u_{i,1} ^{q }  )\circ\ell_i^q\ell_i^{-q} (\tau - u_{i,1}  )\ell_i.  
    \end{align}
Following the notation in Notation \ref{No:Gammai}, a direct computation shows that
\begin{equation*}
 \ell_i^{-q^{l+1}}(\tau - \Theta^{\sigma^i(1-q)\gamma_{l}(\sigma) } u_{i,1} ^{q^{l}} ) \ell_i^{q^l}= \tau+\sigma^i\FT{\Theta}^{l}_{j-i} 
\end{equation*}
for each $0\leqslant l \leqslant j-i$.     
    Therefore, Equation \eqref{Eq:laisoPiPj} can be rewritten as 
    \[ \lambda = \ell_j^{-1} \ell_i^{q^{j-i}}\left(\tau + \sigma^i\FT{\Theta}^{j-i-1}_{j-i} \right)\circ \cdots \circ  \left(\tau + \sigma^i\FT{\Theta}^{1}_{j-i}  \right)\circ\left(\tau + \sigma^i\FT{\Theta}^{0}_{j-i} \right).
    \]
    Let $ C_{i,j}= \ell_j^{-1}\ell_i^{q^{j-i}} $. Then 
    \begin{align*}
   \nonumber C_{i,j} ^{q-1} &= \frac{\sigma^j \Theta }{ \sigma_\infty^j u} \frac{(\sigma_\infty^i u)^{q^{j-i}}}{(\sigma^i \Theta)^{q^{j-i}}} = \frac{\Theta^{\sigma^j -q^{j-i}\sigma^i }}{\Theta^{(1-q)(\gamma_j(\sigma) - q^{j-i}\gamma_i(\sigma) )}} \\
    &= \Theta^{(q-1)\sigma^i (\gamma_{j-i}(\sigma)-  W_{j-i})} \Theta^{\sigma^j -\sigma^i }. 
    \end{align*}
    This yields Equation \eqref{Eq:coCj}.
 \end{proof}

The following theorem (see \cite{Goss96}*{Corollary 4.9.5} or \cite{P23}*{Theorem 5.2.11}) establishes a connection between the lattices of isogenous Drinfeld modules.

\begin{thm}\label{Thm:isgdm}
    Let $\phi$ and $\psi$ be two Drinfeld modules over $\mathbb {C}_{\infty}$ associated with lattices $\Lambda_{\phi}$ and $\Lambda_{\psi}$ respectively. Suppose that the twisted polynomial $\lambda$ is an isogeny from 
$\phi$ to $\psi$ with nonzero constant term $ \partial(\lambda) $ and $I$ is the ideal in $\A$ annihilated by $\lambda$, i.e., $ \lambda = \ell \phi_I$ for some constant $\ell$. Then we have
\[
	\Lambda_{\psi} =  \Lambda_{\phi} \cdot \partial(\lambda) I^{-1}
\]
and 
\[
\lambda(\exp_{\phi}(z) )=\exp_{\psi} (\partial(\lambda)z).
\]

\end{thm}
  
The relationship between  $\exp_{(i)}$ and  $\exp_{(j)}$ is derived in the following corollary.

\begin{cor}\label{Cor:expij} 
 Using the notations for $ \exp_{(j)} $ and $C_{i,j} $ as above. 
For $j\geqslant i\geqslant 0$, we have
     \begin{align}\label{Eq:coexpij}
 & \exp_{(j)}\left( \sqrt[q-1]{\Theta^{\sigma^j -\sigma^i }} \cdot \Theta^{\sum_{l=i}^{j-1}\sigma^{l} }  U \right) \\  \nonumber
     ={}& C_{i,j}\left(\tau + \sigma^i\FT{\Theta}^{j-i-1}_{j-i} \right)\circ \cdots \circ\left(\tau + \sigma^i\FT{\Theta}^{0}_{j-i} \right)\exp_{(i)}(U).
     \end{align}
 \end{cor}
 
\begin{proof}
 By Theorem \ref{Thm:isgdm}, it suffices to compute the constant term $\partial \lambda$ of the isogeny $\lambda\colon\Psi^{(i)}\to\Psi^{(j)}$.
From Equation \eqref{Eq:prolambda} and the expression of $\FT{\Theta}^i_{l}$ in Notation \ref{No:Gammai}, we have
 \[
        \partial \lambda=C_{i,j}\prod_{k=0}^{j-i-1} \left(\sigma^i\FT{\Theta}^{k}_{j-i}\right)=C_{i,j}\Theta^{\sigma^i \left(W_{j-i}+(1-q) \sum_{l=0}^{j-i-1} \gamma_l (\sigma)\right)}.
    \]
It follows from Equation \eqref{Eq:sum1-qgam} that   \begin{equation*} 
      (1-q)\sigma^i \sum_{l=0}^{j-i-1} \gamma_l (\sigma)=-\sigma^i \gamma_{j-i}(\sigma)+\sum_{l=i}^{j-1}\sigma^{l}.   
    \end{equation*}
 Combining with the expression for $C_{i,j}$ in Equation \eqref{Eq:coCj},
   \begin{equation}\label{Eq:parlam}
      \partial \lambda=\sqrt[q-1]{\Theta^{\sigma^j -\sigma^i }} \cdot\Theta^{\sum_{l=i}^{j-1}\sigma^{l}}.
    \end{equation}
Therefore, Equation \eqref{Eq:coexpij} is derived from Theorem \ref{Thm:isgdm}. 

\end{proof}

\begin{cor}
 With the above notations, the period lattice of $\Psi^{(j+2)}$ is
  \[
    \ker \exp_{(j+2)} = \tilde \pi \frac{\Theta^{1+ \sigma + \cdots +\sigma^{j+1} }}{  { \sqrt[q-1]{\Theta^{\sigma^2- \sigma^{j+2}}  } } \cdot \Theta^{1+ \sigma}} I_{\infty}^{-j}.
  \]
   In particular, by setting $ j = N-2 $, we have 
\begin{equation}\label{Eq:Coken0}
    \ker \exp_{(0)} =   \frac{\tilde \pi}{\sqrt[q-1]{\Theta^{  \sigma^2-1} }\Theta^{1+\sigma}}\cdot   I_{\infty}^{2 }.
  \end{equation}
\end{cor}

\begin{proof}

Recall in Theorem \ref{Th:kenelpitil} that the kernel of $ \exp_{(2)} $ is isomorphic to the rank-one $A$-module generated by $ \tilde{\pi }  $.   Theorem \ref{Thm:isgdm} implies that the kernel of $ \exp_{(j+2) } $ is given by 
  \[
    \ker \exp_{(j+2)} = \tilde \pi \Theta^{\sum_{l=2}^{j+1}\sigma^{l}+
\frac{\sigma^{j+2} -\sigma^2}{q-1} } I_{\infty}^{-j}.
  \]

In particular, Equation \eqref{Eq:Coken0} holds.
  \end{proof}
  
\subsection{Exponential Action}
The relation between the exponential functions attached to isogenous Drinfeld modules given in Theorem \ref{Thm:isgdm} is not well adapted for explicit computations. To address this issue, we introduce the following exponential action.
Let $\phi$ be a Drinfeld module over $\mathbb{C}_{\infty}$ with associated exponential function $\exp_{\phi}(U)$. We define a natural map
\[
    \triangleright_{\phi}:  \A \times \mathbb{C}_{\infty} \to \mathbb{C}_{\infty} 
 \]
 by 
 \[
    a \triangleright_{\phi} U  := \exp_{\phi}( a(\theta) \cdot U ), 
 \]
 for $a \in \A$. Moreover, we extend $\triangleright_{\phi}$ as follows.
 
\begin{defn}\label{Def:expgU}
Let $ \phi $ be a Drinfeld $ \A $-module over $\mathbb{C}_{\infty}$. We define 
 \begin{align*}
    \triangleright_{\phi}  :  (\K \otimes_{\mathbb{F}_q} \mathbb{C}_{\infty}) \times \mathbb{C}_{\infty} &{} \to \mathbb{C}_{\infty} \\
     \left( \frac{\sum_{i=1}^{m} b_i \otimes l_i }{a }\right) \triangleright_{\phi}  U &{} := \sum_{i=1}^{m} l_i \phi_{b_i }\exp_{\phi}(\frac{U}{a(\theta)}) 
\end{align*}
 for $ a(t),~ b_i(t)\in \A  $ and $ l_i , ~U \in \mathbb{C}_{\infty}$.
\end{defn}
We remark that Definition  \ref{Def:expgU} is well-defined and independent of the choices of $a,~ b_i, ~l_i $.
The following lemma holds immediately.
\begin{lem}\label{lem:gaU}
 Assume that $g(t) \in \K \otimes_{\mathbb{F}_q} \mathbb{C}_{\infty}$ and $ U \in \mathbb{C}_\infty $.
    \begin{enumerate}
        \item For $a(t) \in \K $, we get
    \[
        ( g(t)\cdot a(t) ) \triangleright_{\phi} U = g(t) \triangleright_{\phi} ( a(\theta) U).
    \]
    \item For $ l \in \mathbb{C}_{\infty} $, we have 
    \[
        (l \cdot  g(t)) \triangleright_{\phi} U = l \cdot( g(t) \triangleright_{\phi}  U ) .
    \]
    \end{enumerate}
 \end{lem}
 The value of $g(t) \triangleright_{\phi} U$ can be derived from the motive of $\phi$ via the following lemma.

\begin{lem}\label{lem:motivetri}

    Assume that $ g(t) \in a^{-1}\A \otimes_{\mathbb{F}_q} \mathbb{C}_{\infty } $ for some $ a \in \A $. Let $ h(\tau)$ be the twisted polynomial determined by the equality
\[
 g(t) a(t)  * \mathbf{s}_{\phi} = h(\tau)\mathbf{s}_{\phi}  
\] 
in the motive of $ \phi $. Then 
\begin{equation}\label{eq:gttri}
     g(t) \triangleright_{\phi} U = h(\tau) \left(\exp_{\phi}(\frac{U}{a(\theta)})\right).
\end{equation}
\end{lem}
\begin{proof}
     Write $ g(t) = \frac{1}{a}\sum_{i=1}^m b_i \otimes l_i $. The assumption of this lemma yields
     \[
         \sum_{i=1}^m  l_i \phi_{b_i} = h(\tau), 
     \]
     which confirms the equality \eqref{eq:gttri} directly. 
\end{proof}
 
\begin{lem}\label{Lem:eguexp}
    Let $ \phi $ be a Drinfeld $\A$-module of rank $r$ over $\mathbb{C}_{\infty} $.
    Assume that the exponential function of $\phi $ is given by 
    \[
        \exp_{\phi} (U) = \sum_{n=0}^\infty \frac{ U^{q^n} }{D_n^{\phi}},
    \]
    where $ D_n^{\phi} \in \mathbb{C}_{\infty} $. 
        For $ g(t) \in \K \otimes_{\mathbb{F}_q} \mathbb{C}_{\infty} $,  we have 
    \begin{align}\label{Eq:gtUgthe}
          g(t)\triangleright_{\phi}  U  =  \sum_{n=0}^{\infty} \frac{U^{q^n}}{D_n^{\phi}} g(\theta^{q^n}).
    \end{align}
 
\end{lem}
\begin{proof}
    Assume that $ g(t) \in \A  \otimes_{\mathbb{F}_q} \mathbb{C}_{\infty} $ is written as 
    $ g(t) = \frac{1}{a} \sum_{i=1}^m b_i \otimes l_i  $.  By definition, we have 
    \begin{align*}
       g(t) \triangleright_{\phi} U  
      & =  \sum_{i=1}^{m}  l_i   \exp_{\phi}(\frac{ b_i(\theta)U }{a(\theta)} ) 
      \\ 
     & =   \sum_{n=0}^{\infty}  \frac{  U^{q^n}}{D_n^{\phi}} \sum_{i=1}^{m}  l_i \frac{ b_i(\theta^{q^n} )}{a(\theta^{q^n}) } 
     \\& =   \sum_{n=0}^{\infty}  \frac{  U^{q^n}}{D_n^{\phi}} g(\theta^{q^n}). 
    \end{align*}
   So the equality in this lemma holds. 
\end{proof}

 Proposition \ref{Pro:sPsiThe} allows us to relate the actions $\triangleright_{\Psi}$ and $\triangleright_{\Psi^{(j)}}$.

\begin{cor}\label{Cor:ePsie} 
For $j>0$, let $ \exp_{(j)} $ denote the exponential function associated with $\Psi^{(j)}$. Using the symbol $\FT{\Theta}^{i}_{j}$  as defined in Notation \ref{No:Thetakj}, we obtain the following identities:
     \begin{align*}
       \nonumber &  \frac{1}{(t - \eta^{(0)})(t - \eta^{(1)}) \cdots (t - \eta^{(j-1)})} \triangleright_{\Psi}  U      \\  
          ={}& \Theta^{ \gamma_{j}(\sigma)-  W_{j}}\left(\tau+\FT{\Theta}^{j-1}_{j}\right)\circ\cdots\circ\left(\tau+\FT{\Theta}^{1}_{j} \right)\circ \left(\tau+\FT{\Theta}^0_{j}\right) \exp_{(0)}(U) \\
     ={}& \frac{1}{ \sqrt[q-1]{\Theta^{\sigma^j-1}}}\exp_{(j)}\left(     \frac{ U \sqrt[q-1]{\Theta^{\sigma^j-1}} }{(\theta - \eta^{(0)})(\theta - \eta^{(1)}) \cdots (\theta - \eta^{(j-1)})} \right )\\
       ={}& \frac{1}{ \sqrt[q-1]{\Theta^{\sigma^j-1}}}\triangleright_{\Psi^{(j)} } \left(     \frac{  U \sqrt[q-1]{\Theta^{\sigma^j-1}}  }{(\theta - \eta^{(0)})(\theta - \eta^{(1)}) \cdots (\theta - \eta^{(j-1)})} \right ).
     \end{align*}
 \end{cor} 
 \begin{proof}
 The corollary follows directly by applying Proposition \ref{Pro:sPsiThe}, Corollary  \ref{Cor:expij} and Lemma \ref{lem:motivetri}.
 \end{proof}
 
As a consequence, we have the following corollary.
\begin{cor} \label{Cor:Psi0jtri}
 For $ g(t) \in \K \otimes_{\mathbb{F}_q} \mathbb{C}_{\infty} $,  we have 
 \begin{align}\label{Eq:Psiphitrige}
       \nonumber &  \frac{g(t)}{(t - \eta^{(0)})(t - \eta^{(1)}) \cdots (t - \eta^{(j-1)})} \triangleright_{\Psi}  U      \\  
       ={}& \frac{g(t)}{ \sqrt[q-1]{\Theta^{\sigma^j-1}}}\triangleright_{\Psi^{(j)} } \left(     \frac{  U \sqrt[q-1]{\Theta^{\sigma^j-1}}  }{(\theta - \eta^{(0)})(\theta - \eta^{(1)}) \cdots (\theta - \eta^{(j-1)})} \right ).
\end{align}
\end{cor}
From Lemma \ref{Lem:eguexp} and Corollary \ref{Cor:Psi0jtri}, one can easily derive the coefficients of $ \exp_{(j)}$, which recovers \eqref{Eq:D_nsigj}.
 
\section{Anderson Generating Functions}\label{Sec:AGF}
In this section, we point out that the set of solutions of the differential operator $ \nabla^{f} := \tau - f $ forms a one-dimensional $ \K$-linear space spanned by the Anderson-Thakur function $ \omega_{f}$. We show that this kind of Anderson-Thakur function connects to two Drinfeld $\A$-modules, i.e., $ \Phi $ and $ \Psi $.   We aim to give the product series expression of $ \omega_{f}$ in Theorem \ref{Thm:omeg0}, and then investigate its deformation versions, namely the Pellarin-type (Definition \ref{Def:GU}), the Anderson-type (Definition \ref{Def:U}) and the deformed logarithm given in Definition \ref{Defn:deformationLog}.
\subsection{Frobenius Differential Equations and Anderson-Thakur Function}
Recall that 
$ \mathbb{C}_{\infty} $ is the completion of an algebraic closure of $ K_{\rho}$ associated with the map $ (-)^{(1)} : = x \mapsto x^q $. The following notation is the appropriate counterpart of the formal series ring $\mathbb{C}_{\infty}(\!(\ft)\!)$ in the context of Carlitz module theory.
\begin{notation}
    Let $ \mathbb{C}_{\infty}[\![I_{\infty} ]\!] $ be the formal completion of the $ \mathbb{C}_{\infty}$-algebra $ \mathbb{C}_{\infty}[\frac{\ft^0}{\rho(\ft)}, \cdots , \frac{\ft^{N-1}}{\rho(\ft)}]  $ at the ideal $(\frac{\ft^0}{\rho(\ft)}, \cdots , \frac{\ft^{N-1}}{\rho(\ft)})$. Denote by $\mathbb{C}_{\infty}(\!(I_{\infty})\!)$ the quotient field of $\mathbb{C}_{\infty}[\![I_\infty]\!]$.
\end{notation}
Obviously, $ \A_{ \mathbb{C}_{\infty} } = \A\otimes_{\mathbb{F}_q}\mathbb{C}_{\infty} $ embeds into $ \mathbb{C}_{\infty}(\!(I_{\infty})\!) $ by sending $ a \otimes l $ to $ a(\ft) \cdot l $. In this way, we view $ \mathbb{C}_{\infty}(\!(I_{\infty})\!) $ as an $\A_{ \mathbb{C}_{\infty} } $-module.
The twist $(-)^{(1)}$ on $ \mathbb{C}_{\infty}$ (resp. $\A\otimes_{\mathbb{F}_q} \mathbb{C}_{\infty}$) extends canonically to   $\mathbb{C}_{\infty}(\!(I_\infty)\!)$. 
Combining the two structures, we obtain a natural action of $ \A_{ \mathbb{C}_{\infty} } \{\tau \} $ on the field $ \mathbb{C}_{\infty}(\!(I_{\infty})\!) $, described as follows: 
\begin{equation}\label{Eq:actaulg}
 \tau \diamond g(\ft) = g(\ft)^{(1)} ,\qquad l(t,\theta) \diamond g(\ft) = l(\ft,\theta) \cdot g(\ft)
\end{equation}
for $g(\ft)\in \mathbb{C}_{\infty}(\!(I_{\infty})\!) $ and  $l(t,\theta) \in \A_ {\mathbb{C}_{\infty} }$. Since the multiplication by $a \in \A$ commutes with the $ \tau $-action, we indeed obtain a $\tau$-module structure (see Definition \ref{Def:t-m}) on  $\mathbb{C}_{\infty}(\!(I_{\infty})\!) $.

\begin{defn}\label{Defn:FDO}
 Given $a\in \A$,   we define the Frobenius differential operator $\nabla_{a}$ associated with $a$ by: 
\[
\nabla_{a}  (\omega)= (\Psi_{a}-a) \diamond \omega = \Psi_{a} \diamond \omega - a(\ft) \cdot \omega,  
\]
for  $ \omega = \omega(\ft) \in  \mathbb{C}_{\infty}(\!(I_{\infty})\!) $. We also define the differential operator $ \nabla^{f} $ associated with the shtuka function $f(t) \in \A $ as follows:
\begin{equation*}
    \nabla^{f}(\omega)=(\tau - f(t) )\diamond \omega=\omega^{(1)}-f(\ft)\cdot \omega.
\end{equation*}
\end{defn}
\begin{thm}\label{Thm:Nabsoluset}
The solutions of $ \nabla^{f} $ and $ \nabla_{a} $ for all $ a \in \A $ coincide. More precisely, the functions $ \omega_*\in  \mathbb{C}_{\infty}(\!(I_{\infty})\!) $ such that 
\begin{equation}\label{Eq:thnaPo}
      f(\ft) \omega_{*} = \omega_{*}^{(1)}.
    \end{equation}
 are the solutions satisfying
\begin{equation}\label{Eq:thnaa}
        \nabla_{a}(\omega_* ) = 0 ,
    \end{equation}
    for all $ a \in \A $. 
\end{thm}
\begin{proof}
    Recall that the coefficients 
  $ b _i\in \mathbb{F}_q(\eta) $ satisfy Equations  \eqref{Eq:nosumbiTi} and \eqref{Eq:sumbiPTi}. 
We obtain 
\begin{equation}\label{Eq:nablaPsi}
      \nabla^{f} =\tau-f(t) = \tau + \Theta - \frac{1}{t -\eta } = \sum_{i=0}^{N-1} b_i (\Psi_{T_i}-T_i) .
\end{equation}
Substituting $a=T_i$ in Equation \eqref{Eq:thnaa} for $i=0,1\cdots, N-1$, it follows that
\[
\nabla_{T_i}(\omega_*) =\Psi_{T_i} \diamond \omega_* - T_i(\ft) \cdot \omega_*=0.
\]
Combining this with Equation \eqref{Eq:nablaPsi}, we have   
\[
 \nabla^{f}(\omega_{*}) =\sum_{i=0}^{N-1} b_i (\Psi_{T_i}-T_i) \diamond \omega_* =  \sum_{i=0}^{N-1}{b_i \nabla_{T_i}(\omega_*)} =0.
\]
This implies Equation \eqref{Eq:thnaPo}.

Conversely, let $ \omega_* $ be a solution of \eqref{Eq:thnaPo} and let $ M_{\omega_*} = \mathbb{C}_{\infty}\{ \tau \} \diamond \omega_* $ be the submodule of $\mathbb{C}_{\infty}(\!(I_\infty)\!)$ generated by $ \omega_*$. It is naturally isomorphic to $ \A_{\mathbb{C}_\infty} $ as a rank-one $\mathbb{C}_{\infty}\{\tau\} $-module (see Equation \eqref{Eq:HLiso}). Indeed, they are isomorphic as $\tau$-modules since Equation \eqref{Eq:thnaPo} implies that
\[
    f(t) \diamond \omega_* = \tau \diamond \omega_*,
\]
which exactly matches the Anderson motive structure of $\A_{ \mathbb{C}_{\infty} } $.
According to Definition \ref{Def:PL} of $ \Psi $, we obtain 
\[
    a(t) \diamond \omega_* = \Psi_{a} \diamond \omega_*
\]
for $a(t) \in \A$. In other words, $ \omega_* $ is a solution of $\nabla_a$ for all $ a(t) \in \A$.  
\end{proof}
\begin{defn}
   Let $\omega_{f}(\ft) $  be a function in   $\mathbb{C}_{\infty}(\!(I_{\infty})\!)$ defined by the following product series:
\begin{equation}\label{Eq:omegaf}
     \omega_{f}(\ft) := \sqrt[q-1]{-\Theta}\,  \prod_{k=0}^{\infty} \left( \frac{\Theta^{q^k}}{\Theta^{q^k} - \frac{1}{\ft - \eta^{(k)}}}\right).
\end{equation}
\end{defn}
\begin{thm}\label{Thm:omeg0}
   The solution space of $ \nabla^{f} $ is a one-dimensional $ \K=\mathbb{F}_q(t)$-linear space spanned by the function $ \omega_{f}$.
\end{thm}
\begin{proof}
It is obvious that the solution space of $ \nabla^{f} $ is a one-dimensional $\K$-linear space. We only need to verify that $ \omega_{f} $ is a nontrivial solution of $ \nabla^{f} $. 
A direct computation yields
\begin{align*}
    \omega_{f}^{(1)}& =-\Theta \sqrt[q-1]{-\Theta}\, \prod_{k=1}^{\infty} \left( \frac{\Theta^{q^k}}{\Theta^{q^k} - \frac{1}{\ft - \eta^{(k)}}}\right)  
     =  \frac{1}{\eta - \theta } \cdot \frac{\ft -\theta}{\ft-\eta}\cdot \omega_{f}\\
    &=  \left(  \frac{1}{\ft - \eta } - \frac{1}{\theta  - \eta  }\right) \cdot \omega_{f}=f(\ft) \diamond \omega_{f}.
\end{align*}
This confirms that $ \omega_{f}$ is a nontrivial solution of $ \nabla^{f}$.

\end{proof}
 Due to Theorem \ref{Thm:omeg0}, we call $ \omega_f $ the Anderson-Thakur function associated with the shtuka function $f$.

\subsection{The Dual Drinfeld Module of $ \Phi$}
We remark that the residue formula of $ \omega_f $ connects to Carlitz period of the Drinfeld module $ \Phi$, defined as follows.   
\begin{defn}\label{Def:Phi}
    Recall that $ C_{0,2} = \sqrt[q-1]{\Theta^{\sigma^2 -1} } \Theta^{\sigma-1}  $ in Equation \eqref{Eq:coCj}.
    Let us define 
    $ \Phi $ to be the $\eta^{(1)}$-type Drinfeld $\A$-module defined by
        \[ 
    \Phi_a(\xi) = C_{0,2}^{-1}\Psi_{a}^{(2)}(C_{0,2} \xi )
    \]
    for $a \in \A $. 
    \end{defn}
    \begin{example}
   For example, when $N=2$, according to Example \ref{Exam:Psi2}, we have 
\begin{align}\label{Eq:Phi}
\begin{cases}
    \Phi_{T_0}  =(\Theta^{\sigma-1} )^{q^2}\tau^2+(\Theta+\Theta^q)(\Theta^{\sigma-1} )^q\tau+\Theta^{1 + \sigma} \\
      \Phi_{T_1} = \eta^{(1)}(\Theta^{\sigma-1} )^{q^2}\tau^2+(\theta\Theta+\eta^{(1)}\Theta^q)(\Theta^{\sigma-1} )^q\tau+\theta\Theta^{1 + \sigma}. 
\end{cases} 
\end{align}
\end{example}
By definition, it is easy to see that 
\begin{equation}\label{eq:expPhi}
\exp_{\Phi}(U) =  C_{0,2}^{-1} \exp_{(2)} ( C_{0,2} U )
\end{equation}
and 
\begin{equation}\label{eq:logPhi}
\log_{\Phi}(U) =  C_{0,2}^{-1} \log_{(2)} ( C_{0,2} U ).
\end{equation}

Thus, the period $\tilde{\pi}_{\Phi}$ of $ \Phi $ is given by 
\begin{equation}\label{Eq:tilpiPhi}
    \tilde{\pi}_{\Phi} = \frac{\tilde\pi}{C_{0,2}}.
\end{equation}

Moreover, from Corollary \ref{Cor:expij}, we have 
\[
    C_{0,2} (\tau + \Theta)^2 \exp_{(0)} (U) = \exp_{(2)} (C_{0,2} \Theta^2 U ).  
\]
Thus, Equation \eqref{eq:expPhi} yields 
\begin{equation}\label{eq:expPhiand0}
    \exp_{\Phi}(U) =    (\tau + \Theta)^2 \exp_{(0)}  ( \frac{U}{\Theta^2} ).
\end{equation}
By Equation \eqref{Eq:asslog} and Lemma \ref{Lem:leLj}, the logarithm function of $\Phi$ can be expressed as 
\begin{equation}\label{Eq:logphi}
   \log_{\Phi}(\xi) = \xi + \sum_{n=1}^{\infty} \frac{1}{f^{(1)}(\theta) \cdots f^{(n)}(\theta)} \xi^{q^n} . 
\end{equation}

\begin{thm}\label{thm:resomw0}
    The Carlitz period $ \tilde \pi_{\Phi}$ of $\Phi$ (given in Equation \eqref{Eq:tilpiPhi}) can be recovered via the residue formula of $ \omega_f $:
     \[
     \Res_{\ft = \theta} \omega_{f} d\frac{1}{\ft - \eta}    = - \frac{\tilde{\pi}}{C_{0,2}}=-\tilde{\pi}_\Phi .
     \]
\end{thm}

\begin{proof}
It is straightforward to verify that the residue of $ \omega_f $ at $ \ft = \theta $ is given by
\[
    \Res_{\ft = \theta} \omega_{f} d\frac{1}{\ft - \eta} =  -\Theta (-\Theta)^{\frac{1}{q-1}}  \prod_{k=1}^{\infty} (\frac{\Theta^{q^k}}{\Theta^{q^k} - \Theta^{\sigma^k} }).
\]
The lemma immediately follows from the expression of $\tilde \pi$ in Theorem \ref{Th:kenelpitil}.
\end{proof}

\begin{lem}\label{lem:expPhi} 
The exponential function of $ \Phi $ is given by 
   \[\exp_\Phi(U) =  \sum_{n=0}^\infty \frac{U^{q^n}}{D_n^{\Phi}},
   \] 
   where the coefficient $ D_n^{\Phi} $  is defined as
   \[
    D_n^{\Phi} := \frac{\theta - \eta }{\theta - \eta^{(1)} } D_n \Theta^{2 q^n} (\theta ^{q^n} - \eta)(\theta ^{q^n} - \eta^{(1)}). 
\]
\end{lem}
\begin{proof}
From Corollary \ref{Cor:ePsie}, we have\begin{equation}\label{Eq:expUThe}
\frac{1}{(t - \eta )(t - \eta^{(1)} )}  \triangleright_{\Psi} U = \Theta^{\sigma-1} (\tau+ \Theta)^2 \exp_{(0)}(U).
\end{equation}
On the other hand, 
it follows from Lemma \ref{Lem:eguexp} that
\begin{equation}\label{Eq:eUexpU}
\frac{1}{(t - \eta )(t - \eta^{(1)} )}  \triangleright_{\Psi} U =  \sum_{n=0}^{\infty} \frac{U^{q^n}}{D_n}\frac{1}{(\theta^{q^n} - \eta )(\theta^{q^n} - \eta^q )}   . 
\end{equation}
    Equating Equations \eqref{Eq:expUThe} and \eqref{Eq:eUexpU}, we derive
\begin{equation}
\label{Eq:lesunproUtauThe} 
\frac{\theta-\eta^{(1)}}{\theta-\eta} \sum_{n=0}^{\infty} \frac{U^{q^n}}{D_n }  \frac{1}{(\theta^{q^n} - \eta )(\theta^{q^n} - \eta^{(1)} )}=(\tau+ \Theta)^2 \exp_{(0)}\big(U\big).
\end{equation}
Combining this with Equation \eqref{eq:expPhiand0},  we obtain 
\begin{align*}
\exp_\Phi(U)
& =\frac{\theta-\eta^{(1)}}{\theta-\eta} \sum_{n=0}^{\infty} \frac{ (\frac{U}{\Theta^2}) ^{q^n}}{D_n } \frac{1}{(\theta^{q^n} - \eta )(\theta^{q^n} - \eta^{(1)} )} \nonumber \\
& = \sum_{n=0}^\infty \frac{U^{q^n}}{D_n^{\Phi}}.
\end{align*}
\end{proof}
Applying Lemma \ref{Lem:eguexp}, the following corollary follows immediately.

\begin{cor}\label{cor:PhiAction}
 Using the above notations, we have
    \begin{align}\label{Eq:gtriPhiU}
          g(t)\triangleright_{\Phi}  U  = \frac{\theta - \eta^{(1)} }{\theta - \eta} \sum_{n=0}^{\infty}\left(\frac{U}{\Theta^2}\right)^{q^n}  \frac{g(\theta^{q^n})}{ D_n(\theta ^{q^n} - \eta)(\theta ^{q^n} - \eta^{(1)})} 
    \end{align}
    for $ g(t) \in \K \otimes_{\mathbb{F}_q} \mathbb{C}_\infty $.
\end{cor}
  
 Similarly, we obtain a relation between the actions $ \triangleright_{\Psi} $ and $ \triangleright_{\Phi} $.
\begin{cor}
For $ g(t) \in \K \otimes_{\mathbb{F}_q} \mathbb{C}_{\infty}  $,  we have
  \begin{equation}\label{Eq:PsiPhitri}
    \frac{g(t)}{(t - \eta )(t - \eta^{(1)} )} \triangleright_\Psi \frac{U}{\Theta^2}  =   \frac{\theta - \eta }{\theta -\eta^{(1)}} \Big( g(t) \triangleright_{\Phi} U  \Big). 
    \end{equation}
\end{cor}
\begin{proof}
Lemma \ref{Lem:eguexp} yields that  
\begin{align*}
 \frac{\theta -\eta^{(1)}}{\theta - \eta }  \frac{g(t)}{(t - \eta )(t - \eta^{(1)} )} \triangleright_\Psi \frac{U}{\Theta^2}  =&\frac{\theta -\eta^{(1)}}{\theta - \eta }\sum_{n=0}^{\infty}\left(\frac{U}{\Theta^2}\right)^{q^n}  \frac{g(\theta^{q^n})}{ D_n(\theta ^{q^n} - \eta)(\theta ^{q^n} - \eta^{(1)})} . 
\end{align*}
 The equality \eqref{Eq:PsiPhitri} follows from Corollary \ref{cor:PhiAction}.   
\end{proof}

\subsection{Tate Algebra}
In this part, we introduce the generalized Tate algebra $\mathbb{T}_{\eta}$ and show that $\omega_f$ lies in $\mathbb{T}_{\eta}$.

\begin{defn} We define the domain 
    \begin{equation}\label{Eq:D}
    \mathcal{D}=\left\{z \in  \mathbb{C}_{\infty}\, \Big| ~\left|  ~\frac{1}{z-\eta^{(j)}}\right|_{\eta}\leqslant 1, j=0,\cdots,N-1\right\}.
\end{equation}
In other words, $z \in  \mathcal{D}$ is equivalent to $ v_{\eta}(z - \eta^{(j)})  \leqslant 0 $ for each $ j =0,\cdots,N-1$.
\end{defn} 
Notice that the open disk
\[
\{ z \in \mathbb{C}_{\infty}\, | \, |z|_{\eta}<1 \}
\]
is contained in $ \mathcal{D} $.
\begin{lem}
     None of the points $   \theta^{q^k} $ for $ k \in \mathbb{Z}_{\geqslant 0}  $ lie in $ \mathcal{D}$.
\end{lem}
\begin{proof}
  To see this, note that for any $ z \in \mathcal{D} $, 
\[
    v_{\eta} (z - \theta) = v_{\eta} \left( (z - \eta )- (\theta - \eta) \right)  = v_{\eta} (z - \eta) \leqslant 0 ,
\]
since $  v_{\eta} ( \theta - \eta ) = 1 $. 
It follows that 
\begin{equation}\label{Equ:normttheta}
| z - \theta^{q^k} |_{\eta} = q^{- v_{\eta} (z - \eta)} \geqslant 1 . 
\end{equation}
Consequently, $ z \not = \theta^{q^k}$ for any $ k \geqslant 0 $.  
\end{proof}
\begin{lem}    For $ j = 0, \cdots, N-1 $, the closed disks 
    \[ \mathcal{D}_j := \{ z \in \mathbb{C}_{\infty} | ~| T_j(z ) |_{\eta} \leqslant 1 \} 
    \]
coincide, each being equal to the domain $\mathcal{D}$. 
\end{lem}
\begin{proof}
    First,  we show the inclusion $ \mathcal{D}_{i} \subseteq \mathcal{D} $. 
    Set $z \in \mathcal{D}_i $, and it follows from the definition that $\left|  ~T_i(z) \right|_{\eta}\leqslant 1$.
    Assume that $ z  \not\in \mathcal{D} $. Then there exists some $ n $, such that $v_\eta(z -\eta^{(n)})> 0$. Thus 
    \[
     v_\eta(z)=v_\eta(z-\eta^{(n)}+\eta^{(n)})= 0 ,
     \]
     and for $i\neq n$,
     \[
     v_\eta(z-\eta^{(i)})=v_\eta(z-\eta^{(n)}+\eta^{(n)}-\eta^{(i)})= 0.
    \] 
     Therefore, we obtain 
    \[
    v_\eta\left(\frac{z^j}{(z-\eta)\cdots (z-\eta^{(N-1)})}\right)=jv_\eta(z)-v_{\eta}(z-\eta)-\cdots -v_{\eta}(z-\eta^{(N-1)}) = -v_{\eta}(z - \eta^{(n)}) <  0 ,
    \]
    i.e.,  $ | T_j |_{\eta} > 1 $ for all $j$.
    This contradicts our assumption.
 
    Conversely, we check $ \mathcal{D}_{i} \supseteq \mathcal{D}  $. 
Let $ z \in \mathcal{D} $. This means $v_\eta(z-\eta^{(j)})\leqslant 0$ for $j=0,\cdots,N-1$. 
If there exists some $n$, such that $v_\eta(z-\eta^{(n)})< 0$, then we
    have 
 \[
     v_\eta(z)=v_\eta(z-\eta^{(n)}+\eta^{(n)})= v_\eta(z-\eta^{(n)}) 
     \]
     and 
     \[ v_\eta(z-\eta^{(j)})=v_\eta(z-\eta^{(n)}+\eta^{(n)}-\eta^{(j)}) = v_\eta(z-\eta^{(n)}),
    \] 
    for $j\neq n$. It is evident that
    \begin{align}\label{Eq:vetaTi2}
       v_\eta\left(T_i(z)\right)
       =&{}iv_\eta(z)-v_{\eta}(z-\eta)-\cdots -v_{\eta}(z-\eta^{(N-1)})\nonumber\\
       =&{}i v_\eta(z-\eta^{(n)}) -N v_\eta(z-\eta^{(n)}) > 0, 
    \end{align}
    which yields $ |T_i(z)|_{\eta} < 1 $ for $i=0,\cdots,N-1$. So in this case we have $z \in \mathcal{D}_{i} $. 

Otherwise, we have $v_\eta(z-\eta^{(j)})= 0$ for all $ j $.  Notice that 
\[ v_{\eta}(z) \geqslant \min \left(v_{\eta}(z - \eta^{(j)}), v_\eta (\eta^{(j)}) \right) \geqslant 0 .
\] 
Then  
    \begin{align*} 
v_\eta\left(T_i(z)\right)=iv_\eta(z)-v_{\eta}(z-\eta)-\cdots -v_{\eta}(z-\eta^{(N-1)})=i v_\eta(z)   \geqslant 0. 
    \end{align*}
    Again, we have $z \in \mathcal{D}_i $. 
    This completes the proof. 
\end{proof}


    

\begin{defn}[Tate algebra]
    The Tate algebra $\mathbb{T}_{\eta}$ is the subalgebra of $ \mathbb{C}_{\infty} [\![I_{\infty}]\!]$ consisting of functions $g(\ft) \in \mathbb{C}_{\infty} [\![I_{\infty}]\!]$ that converge in the domain $  \mathcal{D}$ with respect to the $\eta$-norm $| \cdot |_{\eta}$. The Gauss norm of $ g(\ft )$ is defined as
\[
\|g(\ft)\|_{\eta} = \sup_{z \in \mathcal{D}}{ |g(z)|_{\eta} }.
\]
\end{defn}

\begin{prop}
    The Anderson-Thakur function $ \omega_{f}$  belongs to the Tate algebra $ \mathbb{T}_{\eta} $. 
\end{prop}
\begin{proof}
We verify the convergence of $ \omega_f $. 
According to \eqref{Eq:omegaf}, 
\[
    \omega_f(\ft) =\sqrt[q-1]{-\Theta}  \prod_{k=0}^{\infty} (1 + \omega_k(\ft) ) ,
\]
where 
\[
\omega_k(\ft)=  \frac{\Theta^{q^k}}{\Theta^{q^k} - \frac{1}{\ft - \eta^{(k)}}} - 1 = \frac{\frac{1}{\ft - \eta^{(k)}}}{\Theta^{q^k} - \frac{1}{\ft - \eta^{(k)}}} =  \frac{(\theta  - \eta)^{q^k} }{\ft - \theta^{q^k}}.    
\]
From the inequality \eqref{Equ:normttheta}, it follows that
\[
    |\omega_k(\ft)|_{\eta} = q^{-q^k} \cdot \left|\frac{1}{\ft - \theta^{q^k}} \right|_{\eta} \leqslant q^{-q^k}. 
\]
Therefore, $ |\omega_k(\ft)|_{\eta} \to 0 $ as $ k \to \infty$, which yields that $ \omega_f $ converges for $ \ft \in \mathcal{D} $.
\end{proof}

\subsection{Pellarin-Type Generating Function}


\begin{defn}\label{Def:GU}
    We define the Pellarin-type generating function $G(U;\ft)$ as
\begin{align}\label{Eq:deG} 
  G(U; \ft )= \sum_{n =0 }^{\infty } \frac{U^{q^n}} { D_n Q_n(\ft) },  
\end{align}
where $Q_n(\ft) $ is defined as
\begin{align}\label{Eq:Qn} 
    Q_n(\ft) = (\theta^{q^n} - \eta)^2 \Theta^{2q^n} \left(\frac{1}{\theta^{q^n} - \eta} - \frac{1}{\ft - \eta } \right)=\left(\frac{\theta^{q^n} - \eta}{\theta^{q^n} - \eta^{(n)}} \right)^2  \left(\frac{1}{\theta^{q^n} - \eta} - \frac{1}{\ft - \eta } \right)
\end{align}
and  $D_n$'s are the coefficients of the exponential function $\exp_{(0)}$.
\end{defn}
The following lemma shows that $ G(U;\ft) $ satisfies the deformed version of the Frobenius differential equation in  \eqref{Eq:thnaPo}.  
\begin{prop} \label{prop:G(1)fG}
   We have the following equalities
   \begin{equation}\label{Eq:GfGapr1}
         G^{(1)}(U;\ft) 
         =f(\ft)\cdot G (U;\ft) +   \exp_{\Phi}(U).
    \end{equation}
\end{prop}

\begin{proof}
 
Set 
  \begin{align}\label{Eq:delan} 
  \lambda_n (\ft) :=   \left( \frac{1}{\theta^{q^n}- \eta } - \frac{1}{ \theta - \eta } \right) \cdot  \frac{Q_{n} (\ft)}{ Q_{n-1}^{(1)} (\ft)} - f(\ft). 
\end{align}
It follows from the expression of $Q_n$ in Equation \eqref{Eq:Qn} that
\begin{align}\label{Eq:QnQn-1}
\frac{Q_{n}(\ft)}{ Q_{n-1}^{(1)}(\ft)} 
=\frac{(\theta^{q^n}-\eta)(\ft-\eta^{(1)})}{(\theta^{q^n}-\eta^{(1)})(\ft-\eta)}.
\end{align}
This implies that  
\begin{align}\label{Eq:lanexp}
\nonumber \lambda_n(\ft)&=\frac{(\theta-\theta^{q^n})(\ft-\eta^{(1)})}{(\theta^{q^n}-\eta^{(1)})(\theta-\eta)(\ft-\eta)}-\frac{\theta-\ft}{(\theta-\eta)(\ft-\eta)}
\\\nonumber
      &=\frac{(\ft-\theta^{q^n})(\theta-\eta^{(1)})}{(\theta^{q^n}-\eta^{(1)})(\theta-\eta)(\ft-\eta)}
 \\    \nonumber
    &=\frac{\theta-\eta^{(1)}}{(\theta^{q^n}-\eta^{(1)})(\theta-\eta)}-\frac{(\theta^{q^n}-\eta)(\theta-\eta^{(1)})}{(\theta^{q^n}-\eta^{(1)})(\theta-\eta)(\ft-\eta)}\\
    &=\frac{  (\theta - \eta^{(1)})Q_n(\ft)}{(\theta^{q^n} - \eta^{(1)}) (\theta^{q^n} - \eta)(\theta - \eta) \Theta^{2 q^{n}}}.
\end{align}
By Definition \ref{Def:GU}, we have
    \begin{align*} 
 G^{(1)}(U; \ft )&= \sum_{n =0 }^{\infty } \frac{U^{q^{n+1}}} { D_n^{(1)} Q_n^{(1)}(\ft) }
\\ 
&= \sum_{n =0 }^{\infty } \frac{U^{q^n}} { D_n Q_{n-1}^{(1)}(\ft) } ( \frac{1}{\theta^{q^n}- \eta } - \frac{1}{ \theta - \eta } ) \qquad \text{by Equation \eqref{Eq:Dn}}\\
&= \sum_{n =0 }^{\infty } \frac{U^{q^n}} { D_n Q_n(\ft) }  \left( \frac{1}{\theta^{q^n}- \eta } - \frac{1}{ \theta - \eta } \right) \cdot  \frac{Q_{n}(\ft)}{ Q_{n-1}^{(1)}(\ft)}\\ 
&=\sum_{n =0 }^{\infty } \frac{U^{q^n}} { D_n Q_n(\ft) }   \left(f(\ft)+\lambda_n(\ft)\right) \qquad  \text{by Equation \eqref{Eq:delan}}\\ 
&=f(\ft) G (U;\ft) + \frac{  \theta - \eta^{(1)}}{ \theta - \eta} \sum_{n =0 }^{\infty } \frac{1 } { D_n }\Big(\frac{1}{\theta^{q^n} - \eta^{(1)}}  \frac{1}{\theta^{q^n} - \eta} \Big) \left( \frac{U}{\Theta^2}\right)^{q^n}   \qquad  \text{by Equation \eqref{Eq:lanexp}}\\
&=f(\ft) G (U;\ft) +  (\tau + \Theta)^2 \exp_{(0)} \left(\frac{U}{\Theta^2}\right)  \qquad  \text{by Equation \eqref{Eq:lesunproUtauThe}}.
\end{align*}
Utilizing the relation of $\exp_{(0)}$ and $\exp_{\Phi}$ in \eqref{eq:expPhiand0}, we complete the proof of Equation \eqref{Eq:GfGapr1}.
\end{proof}

The following theorem is the consequence of Theorem  \ref{Thm:omeg0} and Proposition \ref{prop:G(1)fG}.
\begin{thm}
Recall that $\tilde \pi_{\Phi} = \frac{\tilde{\pi}}{C_{0,2}}$ denotes the Carlitz period of $ \Phi $.   We obtain the following equality:
\begin{equation}\label{eq:omegafft}
    \omega_{f}(\ft)= G ( \tilde{\pi}_{\Phi} ; \ft ).
\end{equation}

\end{thm}
\begin{proof}
    By setting  $U = \tilde{\pi}_{\Phi} $ in Proposition \ref{prop:G(1)fG}, we have
    \[
        G(\tilde{\pi}_{\Phi};\ft)^{(1)}  = f G(\tilde{\pi}_{\Phi};\ft) + \exp_{\Phi}(U)= f G(\tilde{\pi}_{\Phi};\ft) .
    \]
    So $G(\tilde{\pi}_{\Phi};t)$ is a solution of $ \nabla^{f} $.
  By Theorem \ref{Thm:omeg0}, we get 
    \begin{equation*}
        G(\tilde{\pi}_{\Phi}; \ft) = \lambda(\ft)\cdot \omega_{f}(\ft)
    \end{equation*}
    for some rational function $\lambda(\ft)$, which is regular at $\ft=\theta$.
Combining this with Theorem \ref{thm:resomw0}, we have 
\[
    \Res_{\ft = \theta} G(\tilde{\pi}_{\Phi}; \ft) d\frac{1}{\ft - \eta} =  \Res_{\ft = \theta} \lambda(\ft) \omega_{f}(\ft)d\frac{1}{\ft - \eta}=-\lambda(\theta)\tilde{\pi}_{\Phi} . 
\]
On the other hand, a direct computation shows that
\[
    \Res_{\ft = \theta} G(\tilde{\pi}_{\Phi}; \ft) d\frac{1}{\ft - \eta} = - \tilde{\pi}_{\Phi}.
\]
It follows that $\lambda(\theta) = 1$ and thus $\lambda(\ft)=1$. So we obtain \eqref{eq:omegafft}.
\end{proof}
\subsection{Anderson-Type Generating Function}

\begin{defn}\label{Def:expgtU}
 We extend the exponential action $\triangleright_{\phi}$ from Definition \ref{Def:expgU} to the Tate algebra $\mathbb{T}_{\eta}$ as follows. Define
\[
\triangleright_{\phi} : (\K \otimes_{\mathbb{F}_q} \mathbb{T}_{\eta}) \times \mathbb{C}_{\infty} \longrightarrow \mathbb{T}_{\eta}
\]
by
\[
\biggl( \sum_{i} a_i(t) \otimes T_i(\ft) \biggr) \triangleright_{\phi} U = \sum_{i} \bigl( a_i(t) \triangleright_{\phi} U \bigr) T_i(\ft),
\]
where $a_i(t) \in \K$, $T_i(\ft) \in \mathbb{T}_{\eta}$, and the index $i$ may range over an infinite set provided the series converge.
\end{defn}
This notation allows us to naturally generalize the Anderson-type generating function. For the case of Drinfeld module $ \phi $ over polynomial ring $\mathbb{F}_q[t]$, the algebra $ \K $ can be realized as the rational function field $ \mathbb{F}_q(t) $.  We have $ t \triangleright_{\phi} U = \phi_{ t} \exp_{\phi}(U) = \exp_{\phi} (\theta U )$. In analogy with Definition \ref{Def:expgU}, we get 
\[ \frac{1}{t} \triangleright_{\phi} U =  \exp_{\phi} \left({\frac{U}{\theta}}  \right). \]
From the Taylor expansion
\begin{equation}\label{eq:Taylor}
    \frac{1}{t -\ft } = \frac{1}{t} +  \frac{1}{t^2} \ft + \frac{1}{t^3} \ft^2 + \cdots ,
\end{equation}
we have 
\begin{align}
   \frac{1}{t -\ft }  \triangleright_{\phi} U :=&\frac{1}{t} \triangleright_{\phi}  U   +  \left( \frac{1}{t^2 }\triangleright_{\phi} U  \right) \ft + \left( \frac{1}{t^3} \triangleright _{\phi} U   \right) \ft^2 + \cdots \nonumber \\
   =&\sum_{k=0}^{\infty}\exp_{\phi}({\frac{U}{\theta^{k+1}})} \ft^{k}. \label{eq:Andersonphi}
\end{align}
This is exactly the desired Anderson-type generating function over polynomial ring $\mathbb{F}_q[t]$.

 Inspired by this construction, we define the Anderson-type generating function in our setting as follows.
\begin{defn}\label{Def:U}
    For any $ U \in \mathbb{C}_{\infty}$,  the Anderson-type generating function is defined as 
\[
H(U;\ft ) := \frac{\theta - \eta  }{\theta - \eta^{(1)}  } \frac{t - \eta^{(1)} }{ t - \eta  } \left(    \frac{1}{ t - \eta} - \frac{1}{ \ft - \eta } \right) ^{-1} \triangleright_{\Phi} U,
\]
where $\triangleright_{\Phi}$ is the symbol for the exponential action associated with  $\Phi$. 
\end{defn}
From the relation \eqref{Eq:PsiPhitri} between $ \triangleright_{\Psi} $ and $ \triangleright_{\Phi} $, it follows that 
\[
H(U;\ft ) := \frac{1 }{ (t-\eta)^2 }   \left(    \frac{1}{ t - \eta} - \frac{1}{ \ft - \eta } \right) ^{-1} \triangleright_{\Psi} \frac{U}{\Theta^{2}}.
\]

To derive the explicit expression of $ H(U;\ft)$, we need the expansion of the inverse of the expression $1 - \frac{t - \eta}{\ft - \eta}$  analogous to Equation \eqref{eq:Taylor}.

\begin{lem}\label{Lem:1-theeaexp}
Let $ \bar{\sigma} $ be a generator of $\Gal(H/ K)$. Then the inverse of  $1 - \frac{t - \eta}{\ft - \eta}$  can be expressed as the series
\begin{equation}\label{Eq:le1-thet}
        \left({1 - \frac{t - \eta}{\ft - \eta}} \right)^{-1} =  1 + \frac{t - \eta }{\ft - \bar{\sigma}\eta }+ \frac{(t - \eta )(t - \bar{\sigma}\eta)}{(\ft - \bar{\sigma}\eta) (\ft - \bar{\sigma}^2\eta) }+ \frac{(t - \eta )(t - \bar{\sigma}\eta)(t - \bar{\sigma}^2 \eta)}{(\ft - \bar{\sigma}\eta) (\ft - \bar{\sigma}^2\eta)(\ft - \bar{\sigma}^3 \eta) } + \cdots.
\end{equation}
\end{lem}

\begin{proof}
The equality in the lemma can be rewritten as
\begin{align}\label{Eq:LeGDzeta}
\left(
  1 - \frac{t - \eta}{ \ft - \eta}\right)^{-1} = \sum_{l=0}^{\infty} \mathcal{B}_l,
\end{align}
where
\begin{equation}\label{Eq:Bl}
  \mathcal{B}_{l}=\frac{(t - \eta)(t - \bar{\sigma}^{1}\eta)\cdots (t - \bar{\sigma}^{l-1}\eta)}{(\ft - \bar{\sigma}^{1}\eta)(\ft - \bar{\sigma}^{2}\eta)\cdots(\ft - \bar{\sigma}^{l}\eta)}\quad\text{and } \qquad \mathcal{B}_{0}=1.  
\end{equation}
We claim that for $  m \geqslant 1 $, the following identity holds
\begin{equation}\label{Eq:mathcalBl}
\sum_{l=1}^{m}\left(\mathcal{B}_l-\frac{t - \eta}{ \ft - \eta}\mathcal{B}_{l-1}\right)=\frac{\bar{\sigma}^{m}\eta-\eta}{ \ft - \eta}\mathcal{B}_m.
\end{equation}
Observe that the sequence $ \{ \mathcal{B}_{l} \} $ verifies the recursive formula 
\begin{equation}\label{Eq:Bll-1} 
\mathcal{B}_{l}=\frac{\ft-\bar{\sigma}^{l+1}\eta}{t-\bar{\sigma}^l\eta}\mathcal{B}_{l+1}  .  
\end{equation}
From Equation \eqref {Eq:Bll-1}, it is clear that Equation \eqref {Eq:mathcalBl} holds for $m = 1$:
\[
\mathcal{B}_1-\frac{t - \eta}{ \ft - \eta}\mathcal{B}_{0}= \left(1-\frac{\ft -\bar{\sigma}\eta}{ \ft - \eta} \right) \mathcal{B}_1=\frac{\bar{\sigma}\eta-\eta}{  \ft - \eta}\mathcal{B}_1.
\]
Assume that the equality \eqref{Eq:mathcalBl} holds for $m \geqslant 1$.
Then, we have 
\begin{align*}
\sum_{l=1}^{m+1}\left(\mathcal{B}_l-\frac{t - \eta}{ \ft - \eta}\mathcal{B}_{l-1}\right)={}&\frac{\bar{\sigma}^{m}\eta-\eta}{ \ft - \eta}\mathcal{B}_m-\frac{t-\eta}{\ft-\eta}\mathcal{B}_{m}+\mathcal{B}_{m+1}\\
={}&\left(-\frac{\ft-\bar{\sigma}^{m+1}\eta}{ \ft - \eta}+1\right)\mathcal{B}_{m+1} ~\text{by Equation \eqref{Eq:Bll-1}} \\
={}&\frac{\bar{\sigma}^{m+1}\eta-\eta}{  \ft - \eta}\mathcal{B}_{m+1}.
\end{align*}
The equality \eqref{Eq:mathcalBl} follows by induction.
Therefore, we have 
 \begin{align*} 
   \sum_{l=1}^{\infty}\left(\mathcal{B}_l-\frac{t - \eta}{ \ft - \eta}\mathcal{B}_{l-1}\right)= \lim_{m\to \infty}\frac{\bar{\sigma}^{m}\eta-\eta}{  \ft - \eta}\mathcal{B}_{m}=0.
\end{align*}
 This implies
      \begin{align}\label{Eq:twosum}
   -1+\sum_{l=0}^{\infty} \mathcal{B}_l= \sum_{l=1}^{\infty} \mathcal{B}_l=\frac{t - \eta}{ \ft - \eta}\sum_{l=0}^{\infty} \mathcal{B}_l.
     \end{align}
Therefore, Equation \eqref{Eq:LeGDzeta} holds.

\end{proof}

 \begin{notation}\label{No:Upk}
For $k\in \mathbb{Z}$, we introduce the rational functions:
 \[
\Upsilon_k(t) =  
\begin{cases}
    (t-\eta^{(1)})\prod_{i=k}^{-1} \frac{1}{t - \eta^{(-i)}} & \text{for $k\leqslant-1$; } \\
    \prod_{i=-1}^{k-1} (t - \eta^{(-i)}) & \text{for $k\geqslant 0$}.
\end{cases}
\]
 \end{notation}

We are now ready to express the Anderson-type generating function  $ H(U;\ft)$ in terms of the quantities  $\Upsilon_k(t)$.
 
\begin{cor}\label{Cor:HUexp}
   The generating function $ H(U;\ft)$ admits the following expansion:
    \begin{align*}
    H(U;\ft) = \frac{\theta - \eta  }{\theta - \eta^{(1)}  } \left( \Upsilon_0(t)  \triangleright_{\Phi}U \right)  + \sum_{k=1}^{\infty} \left( \Upsilon_k(t) \triangleright_{\Phi}U \right)  \frac{\theta - \eta  }{\theta - \eta^{(1)}  } \prod_{i=1}^k \frac{1}{\ft - \eta^{(-i)}}. 
\end{align*}

\end{cor}
\begin{proof}
By applying Lemma \ref{Lem:1-theeaexp}, we first derive the identity
     \begin{align*}
   \frac{\theta - \eta  }{\theta - \eta^{(1)}  }\frac{t - \eta^{(1)} }{t - \eta  }  (\frac{1}{t - \eta} - \frac{1}{\ft - \eta })^{-1} = \frac{\theta - \eta  }{\theta - \eta^{(1)}  } \Upsilon_0(t)  + \sum_{k=1}^{\infty} \Upsilon_k(t) \frac{\theta - \eta  }{\theta - \eta^{(1)}  } \prod_{i=1}^k \frac{1}{\ft - \eta^{(-i)}}. 
\end{align*}
Substituting this result into the definition of 
$ H(U;\ft)$
 (see Definition \ref{Def:U}) 
 immediately yields the desired expansion.
\end{proof}
It remains to compute the action of $\Upsilon_k(t) $ on $U$ for all integers $k$. The following lemma provides the explicit expression of this action.
\begin{lem}\label{lem:Upsilon}
     We keep the above notations.
    \begin{enumerate}
        \item For arbitrary integers  $k,l$, the following identity holds:
    \begin{align}\label{Eq:Upkre}
         \Upsilon_k (t)\rho(t)^l = \Upsilon_{k +N l }(t).
    \end{align}
    \item For arbitrary integers  $k,l$,  we have 
    \begin{align}\label{Eq:leUpUp}
         \Upsilon_{k}(t) \triangleright_{\Phi} \left(\rho(\theta)^{l}U \right)=\Upsilon_{k+lN}(t) \triangleright_{\Phi} U.
    \end{align}
    \item If $k < 0$, then  $ \Upsilon_{k}(t) \triangleright_{\Phi} \mu 
 $ vanishes for $ \mu \in \tilde{\pi}_{\Phi} A \subseteq \mathbb{C}_\infty $.
 \item For all integers $k$, we have
\begin{align}\label{Eq:UptriPhiU}
    \Upsilon_{k}(t) \triangleright_{\Phi} U = \frac{\theta - \eta^{(1)} }{\theta - \eta} \sum_{n=0}^{\infty}\left(\frac{U}{\Theta^2}\right)^{q^n}  \frac{\Upsilon_k (\theta^{q^n})}{ D_n(\theta ^{q^n} - \eta)(\theta ^{q^n} - \eta^{(1)})}.  
    \end{align}
    \item For all integers $k$, the action of $ \Upsilon_k(t)\triangleright_{\Phi} $ on $ U $ admits the exponential-form expression:
 \begin{align}\label{Eq:Upk<0tri}
    \Upsilon_{k}(t) \triangleright_{\Phi} U =\frac{\theta - \eta^{(1)} }{\theta - \eta} \cdot\frac{1}{ \sqrt[q-1]{\Theta^{\sigma^{1-k}-1}}} \cdot \exp_{(1-k)}\left(   \frac{\theta - \eta}{\theta - \eta^{(1)}} \cdot \sqrt[q-1]{\Theta^{\sigma^{1-k}-1}}
   \cdot \Upsilon_k(\theta)
  \cdot U \right ).
    \end{align}
    \end{enumerate}
\end{lem}

\begin{proof}
 (1)  To prove Equation \eqref{Eq:Upkre}, we may assume without loss of generality that $ l = 1 $
and $ -N < k < -1 $. By Notation \ref{No:Upk}, we compute directly
 \begin{align*}
     \Upsilon_k(t)\rho(t)=&\frac{1}{t-\eta^{(2)}}\cdots \frac{1}{t-\eta^{(-k)}}\left( (t-\eta)(t-\eta^{(-1)})\cdots (t-\eta^{(1-N)})\right)\\
     =&(t-\eta^{(1-N)})(t-\eta^{(-k-N+1)})\cdots (t-\eta^{(1)})(t-\eta)=\Upsilon_{k+N}(t).
 \end{align*}
 
(2) Equation \eqref{Eq:leUpUp} follows from the first assertion and Corollary \ref{lem:gaU}.

(3) Since $ \Upsilon_{k}(t) \in \A\otimes_{\mathbb{F}_q} \mathbb{F}_\eta $,
we write 
\[
    \Upsilon_k(t) = \sum_{i=0}^{N-1} a_i \eta^{(i)}
\]
for some $a_i \in \A $. 
Definition \ref{Def:expgU} implies that 
\[
\Upsilon_k(t) \triangleright_{\Phi}  \mu =\sum_{i=0}^{N-1} \eta^{(i)}\exp_{\Phi}( a_i(\theta) \mu) =0.
\]

(4) 
Equation \eqref{Eq:UptriPhiU} follows directly from Equation \eqref{Eq:gtriPhiU}.

    (5) 
     We write $ k = -m + Nl $ for some integers $l $ and $ m \geqslant 2 $. 
Applying Equation \eqref{Eq:leUpUp}, we obtain
    \begin{align*}
      \Upsilon_{k}(t) \triangleright_{\Phi} U= \Upsilon_{-m}(t)\triangleright_{\Phi} (\rho(\theta)^{l}\cdot U).
    \end{align*}
Substituting $g(t)=\Upsilon_{-m}(t)$ into Equation
\eqref{Eq:PsiPhitri} yields
\begin{align*}
      \Upsilon_{k}(t) \triangleright_{\Phi} U= \frac{\theta - \eta^{(1)} }{\theta - \eta}\frac{\Upsilon_{-m}(t)}{(t - \eta )(t - \eta^{(1)} )} \triangleright_\Psi \frac{\rho(\theta)^{l}\cdot U}{\Theta^2}.
    \end{align*}
 From Notation \ref{No:Upk}, we express 
 \[
 \Upsilon_{-m }(t)=\frac{1}{(t - \eta^{(2)} )(t - \eta^{(3)} )\cdots (t-\eta^{(m)})},
 \]
 which gives
  \begin{align*}
      \Upsilon_{k}(t) \triangleright_{\Phi} U=\frac{\theta - \eta^{(1)} }{\theta - \eta} \frac{1}{(t - \eta )(t - \eta^{(1)} )\cdots (t-\eta^{(m)})} \triangleright_\Psi \frac{\rho(\theta)^{l}\cdot U}{\Theta^2}.
    \end{align*}
  Combining this with Corollary \ref{Cor:ePsie}  implies 
      \begin{align*}
      \Upsilon_{k}(t) \triangleright_{\Phi} U=&  \frac{\theta - \eta^{(1)} }{\theta - \eta} \cdot\frac{1}{ \sqrt[q-1]{\Theta^{\sigma^{m+1}-1}}}\cdot  \exp_{(m+1)}\left(     \frac{  \sqrt[q-1]{\Theta^{\sigma^{m+1}-1}} \cdot \rho(\theta)^{l} }{(\theta - \eta^{(0)})(\theta - \eta^{(1)}) \cdots (\theta - \eta^{(m)})}\cdot\frac{U}{\Theta^2} \right )\\
       =&  \frac{\theta - \eta^{(1)} }{\theta - \eta} \cdot\frac{1}{ \sqrt[q-1]{\Theta^{\sigma^{m+1}-1}}}\cdot  \exp_{(m+1)}\left(  \frac{\theta - \eta}{\theta - \eta^{(1)}} \cdot \sqrt[q-1]{\Theta^{\sigma^{m+1}-1}}
   \cdot \Upsilon_k(\theta)
  \cdot U  \right ).
    \end{align*}
    Notice that the equality $ m+1 =1- k + Nl  $ yields $ \exp_{m+1} (\mu) = \exp_{1-k} (\mu )$.
    Therefore, we arrive at Equation \eqref{Eq:Upk<0tri}.
\end{proof}

We remark that combining Corollary \ref{Cor:HUexp} with Lemma \ref{lem:Upsilon} (5) gives an exponential formula of $ H(U; \ft) $, which is analogous to Equation \eqref{eq:Andersonphi}.

Our next task is to prove the two generating functions  
$G(U;\ft)$ and $H(U;\ft)$ are indeed identical. 
\begin{thm}
     We keep the above notations. The following equality holds:
    \begin{align}\label{Eq:thmEG}
    G(U;\ft)=H(U;\ft).
    \end{align}
\end{thm}
\begin{proof}
 Recall the expression of $G(U;\ft)$ in Equation \eqref{Eq:deG}:
\begin{align*}
 G(U;\ft) =&{} \sum \frac{U^{q^n}}{D_n Q_n(\ft)} \\
 = {}&\sum_{n= 0}^{\infty} \frac{1}{ D_n (\theta^{q^n} - \eta)} \left({ 1- \frac{\theta^{q^n} - \eta}{\ft - \eta}}\right)^{-1}  \left(\frac{U}{\Theta^2} \right)^{q^n}.
 \end{align*}
 Applying Lemma \ref{Lem:1-theeaexp} with $ \bar{\sigma}=\sigma^{-1}$ and $t = \theta^{q^n}$, we get
    \begin{align*}
         \left({ 1- \frac{\theta^{q^n} - \eta}{\ft - \eta}}\right)^{-1}=&{} 1 + \frac{\theta^{q^n} - \eta }{\ft - \eta^{(-1)} }+ \frac{(\theta^{q^n} - \eta )(\theta^{q^n} - \eta^{(-1)})}{(\ft - \eta^{(-1)}) (\ft - \eta^{(-2)}) } \\
         &{} + \frac{(\theta^{q^n} - \eta )(\theta^{q^n} - \eta^{(-1)})(\theta^{q^n} - \eta^{(-2)})}{(\ft - \eta^{(-1)}) (\ft - \eta^{(-2)})(\ft - \eta^{(-3)}) } + \cdots.
    \end{align*}
 Substituting this formula into $G(U;\ft)$, we obtain from Lemma \ref{lem:Upsilon} (4) that 
\begin{align*}
 G(U;\ft) 
 ={}&\sum_{n= 0}^{\infty} \frac{ 1}{ D_n }\frac{1}{\theta^{q^n} - \eta}\left(\frac{U}{\Theta^2} \right)^{q^n}
+\frac{1}{\ft - \eta^{(-1)} }\sum_{n= 0}^{\infty} \frac{ 1}{ D_n }\left(\frac{U}{\Theta^2} \right)^{q^n}
\\
& +\frac{1}{(\ft - \eta^{(-1)}) (\ft - \eta^{(-2)}) }\sum_{n= 0}^{\infty} \frac{  1}{ D_n }(\theta^{q^n} - \eta^{(-1)})\left(\frac{U}{\Theta^2} \right)^{q^n} +\cdots\\
&+\frac{1}{\prod_{i=1}^{k+1}\left(\ft - \eta^{(-i)}\right)}\sum_{n= 0}^{\infty} \frac{ 1}{ D_n }\prod_{i=1}^{k}\left(\theta^{q^n} - \eta^{(-i)}\right) \left(\frac{U}{\Theta^2} \right)^{q^n} +\cdots
\\
 = {}&{\frac{\theta - \eta}{\theta - \eta^{(1)} }}  (\Upsilon_0 \triangleright_{\Phi} U) + {\frac{\theta - \eta}{\theta - \eta^{(1)} }} \sum_{k=1}^{\infty} (\Upsilon_k \triangleright_{\Phi} U) \prod_{i=1}^k \frac{1}{\ft - \eta^{(-i)}}  .
\end{align*}
This coincides with the expression of $H(U;\ft)$ in Corollary \ref{Cor:HUexp}.
\end{proof}

\subsection{Deformation of Logarithm}
The following definition is due to Papanikolas \cite{P08}.
\begin{defn}\label{Defn:deformationLog}
  For $ \xi \in \mathbb{C}_{\infty} $, we define the function 
\begin{equation}\label{Eq:Log}
      \Log (\xi; \ft ) = \xi + \sum_{n=1}^{\infty} \frac{\xi^{q^{n}}}{f^{(1)}(\ft) \cdots f^{(n)}(\ft)}.
\end{equation}
\end{defn}
From Equation \eqref{Eq:logphi}, it is evident that $\log_{\Phi}(\xi) = \Log(\xi; \theta) $. In other words, $ \Log (\xi; \ft ) $ acts as a deformation of $ \log_{\Phi} $ in a neighborhood of $ \ft = \theta $.  
Clearly, it follows from Equation \eqref{Eq:Log} that
\[
    \Log (\xi; \ft )^{(1)} = f^{(1)}(\ft) \Log(\xi;\ft ) - \xi f^{(1)}(\ft),
\]
or equivalently
\[
    \nabla^{f} \left( -\frac{\Log(\xi;\ft )}{f(\ft)} \right) = \xi.  
\]

 The following proposition establishes the connection between $\Log(\xi;\ft)$ and $G(U;\ft)$.
\begin{prop}
Assume that $ \exp_{\Phi} (U) = \xi $. Then
    \begin{equation}\label{eq:ProLfG}
    \Log (\xi; \ft )=-f(\ft)G(U;\ft).
    \end{equation}
\end{prop}

\begin{proof}
  From Proposition \ref{prop:G(1)fG}, we know that $ G(U;\ft)$ satisfies
\begin{equation}\label{eq:nablafG}
\nabla^{f} G(U;\ft)=G^{(1)}(U;\ft)-f(\ft)G(U;\ft)= \exp_{\Phi}(U) = \xi .
\end{equation}
Thus,
    \[
    \nabla^{f} \left( -\frac{\Log(\xi;\ft )}{f(\ft)} \right)-\nabla^{f} G(U;\ft)=0.
    \]
     By Theorem \ref{Thm:omeg0},
    \[
  \frac{\Log(\xi;\ft )}{f(\ft)} +G(U;\ft)\in \K \cdot\omega_{f}(\ft).
    \]
  This function is zero since it is regular at $\ft=\theta$, while $\omega_f(\ft)$ has a pole at $\ft=\theta$.
  Identity~\eqref{eq:ProLfG} follows.
\end{proof}

\begin{example}
    In this example, we directly verify that $\Log (\xi; \ft ) $ verifies the extended version of the logarithmic arithmetic for $N=2$. More precisely,  
\begin{align}\label{Eq:Log2}
\begin{cases}
    \Log(\Phi_{T_0(\ft)}(\xi);\ft) 
    = T_0(\ft) \Log(\xi;\ft)-f(\ft) \left( \Theta^{\sigma}  \xi + \frac{\Theta^{\sigma-1} }{\ft - \eta^{(1)}} \xi+\Theta^{q(\sigma-1)}  \xi^q \right)\\
    \Log(\Phi_{T_1(\ft)}(\xi);\ft)=T_1(\ft) \Log(\xi;\ft)-f(\ft)\left(\theta\Theta^{\sigma}\xi+\eta^{(1)}\frac{\Theta^{\sigma-1}}{\ft-\eta^{(1)}}\xi+\eta\Theta^{q(\sigma-1)}\xi^q\right).
    \end{cases}
\end{align} 


Notice that $f(\theta)=0$. So the equalities yield that $\Log(\xi;\theta)$ equals the logarithm function of $\Phi$. 
\begin{enumerate}
    \item Let us prove the first equality of Equation \eqref{Eq:Log2}. Substituting the expression \eqref{Eq:Phi} of $ \Phi $ into Equation \eqref{Eq:Log}, it follows that
\begin{align*}
    \Log(\Phi_{T_0(\ft)}(\xi);\ft)={}&(\Theta^{\sigma-1}\xi)^{q^2} +(\Theta+\Theta^q)(\Theta^{\sigma-1}\xi)^q +\Theta^{1 + \sigma}\xi\\
    &+\sum_{n=1}^{\infty} \frac{\left( (\Theta^{\sigma-1}\xi)^{q^2} +(\Theta+\Theta^q)(\Theta^{\sigma-1}\xi)^q +\Theta^{1 + \sigma}\xi\right)^{q^{n}}}{f^{(1)}(\ft) \cdots f^{(n)}(\ft)} \\
    ={}&\Theta^{1 + \sigma}\xi-f(\ft)\cdot\Theta^{q(\sigma-1)} \xi^{q}+\sum_{n=1}^{\infty} \frac{ \alpha_{n}(\ft) \xi ^{q^{n}}}{f^{(1)}(\ft) \cdots f^{(n)}(\ft)},
\end{align*}
where $ \alpha_n(\ft)  $ is given by 
\begin{align*}
   &\alpha_{n}(\ft)=\Theta^{(\sigma-1)q^n}\cdot f^{(n)}(\ft) \left(f^{(n-1)}(\ft) +(\Theta+\Theta^q)^{q^{n-1}}\right)+\Theta^{(1 + \sigma)q^n}\\
   &= \Theta^{(\sigma-1)q^n}\cdot \left(\frac{1}{\ft-\eta^{(n)}}-\frac{1}{\theta^{q^n}-\eta^{(n)}}\right)\left(\frac{1}{\ft-\eta^{(n-1)}}+\frac{1}{\theta^{q^n}-\eta^{(n)}}\right)+\Theta^{(1 + \sigma)q^n}\\
   &=T_0(\ft).
\end{align*} 
Therefore, we have
\begin{align*}
    \Log(\Phi_{T_0(\ft)}(\xi);\ft)
    ={}&\Theta^{1 + \sigma}\xi-T_0(\ft)\xi-f(\ft)\Theta^{q(\sigma-1)}\xi^q +T_0(\ft)\Log(\xi;\ft)\\
   ={}& -f(\ft)\left(\Theta^{\sigma}+\frac{\Theta^{\sigma-1}}{\ft-\eta^{(1)}}\right)\xi-f(\ft)\Theta^{q(\sigma-1)}\xi^q+T_0(\ft)\Log(\xi;\ft).
\end{align*}
\item Similarly, the second formula of Equation \eqref{Eq:Log2} can be written as 
\begin{align*}
  \Log(\Phi_{T_1(\ft)}(\xi);\ft)
    ={}&\eta^{(1)}(\Theta^{\sigma-1}\xi)^{q^2} +(\theta\Theta+\eta^{(1)}\Theta^q)(\Theta^{\sigma-1}\xi)^q +\theta\Theta^{1 + \sigma}\xi\\
    &+\sum_{n=1}^{\infty} \frac{\left( \eta^{(1)}(\Theta^{\sigma-1}\xi)^{q^2} +(\theta\Theta+\eta^{(1)}\Theta^q)(\Theta^{\sigma-1}\xi)^q +\theta\Theta^{1 + \sigma}\xi\right)^{q^{n}}}{f^{(1)}(\ft) \cdots f^{(n-2)}(\ft)f^{(n-1)}(\ft)f^{(n)}(\ft)}\\
    ={}&\theta\Theta^{1 + \sigma}\xi-f(\ft)\cdot\eta\cdot\Theta^{q(\sigma-1)}\xi^q+\sum_{n=1}^{\infty} \frac{\beta_n(\ft) \xi ^{q^{n}}}{f^{(1)}(\ft) \cdots f^{(n-2)}(\ft)f^{(n-1)}(\ft)f^{(n)}(\ft)},
\end{align*}
where $ \beta_n(\ft) $ is given by 
\begin{align*}
 \beta_n(\ft) =  &\Theta^{(\sigma-1)q^n}\cdot f^{(n)}(\ft) \left(\eta^{(n-1)}f^{(n-1)}(\ft) +(\theta\Theta+\eta^{(1)}\Theta^q)^{q^{n-1}}\right)+\theta^{q^n}\Theta^{(\sigma+1) q^n}\\
    ={}&  \Theta^{(\sigma-1)q^n}\cdot \left(\frac{1}{\ft-\eta^{(n)}}-\frac{1}{\theta^{q^n}-\eta^{(n)}}\right)\cdot\left(\frac{\eta^{(n-1)}}{\ft-\eta^{(n-1)}}+\frac{\theta^{q^n}}{\theta^{q^n}-\eta^{(n)}}\right)+\theta^{q^n}\Theta^{(\sigma+1) q^n} .
\end{align*} 
After careful calculation, we obtain $\beta_n(\ft) =T_1(\ft) $.
Therefore, we have
\begin{align*}
    \Log(\Phi_{T_1(\ft)}(\xi);\ft)
    = -f(\ft)\left(\theta\Theta^{\sigma}+\eta^{(1)}\frac{\Theta^{\sigma-1}}{\ft-\eta^{(1)}}\right)\xi-f(\ft)\cdot\eta\cdot\Theta^{q(\sigma-1)}\xi^q+T_1(\ft)\Log(\xi;\ft).
\end{align*}
\end{enumerate}
\end{example}

 \bibliographystyle{amsplain}
 \bibliography{paper}
\end{document}

%% file: graph.tex
\begin{tikzpicture}[
    dot/.style={circle,fill,inner sep=1.5pt},
    >=stealth,
    scale=1
]
\def\yK{0}    
\def\yH{2}    
\def\yHp{4}   

\def\xP0{3}  
\def\xPinf{9}

\def\xQ0{\xP0}
\def\xQinf{\xPinf}

\def\xRzero {\xQ0}
\def\xRone{2}
\def\xRtwo{4}

\def\xRinfmuzero{8.5}
\def\xRinfmu{9.5}
\def\xRinfmuExA{7.5}
\def\xRinfmuExB{10.5}

\def\xleftHp{0}
\def\xmidleftHp{1}
\def\xmidrightHp{3}
\def\xrightHp1{11}
\def\xrightHpEdge{12}

\draw[dashed] (\xleftHp,\yK) -- (\xrightHpEdge,\yK) node[right] {$K$};
\draw[dashed] (\xleftHp,\yH) -- (\xrightHpEdge,\yH) node[right] {$H$};
\draw[dashed] (\xleftHp,\yHp) -- (\xrightHpEdge,\yHp) node[right] {$H^+$};

\node[dot,label=below:$P_0$] (P0) at (\xP0,\yK) {};
\node[dot,label=below:$P_\infty$] (Pinf) at (\xPinf,\yK) {};

\node[dot,label=below left:$Q_0$] (Q0) at (\xQ0,\yH) {};
\node[dot,label=below left:$Q_\infty$] (Qinf) at (\xQinf,\yH) {};

\node[dot,label=above:$R_0$] (R0) at (\xRzero, \yHp) {};
\node[dot] (R1) at (\xRone ,\yHp) {};
\node[dot] (R2) at (\xRtwo ,\yHp) {};


\node[dot,label=above:$ R_\infty^{(\mu_0)} $] (Rinfmu0) at (\xRinfmuzero,\yHp) {};
\node[dot,label=above:$R_\infty^{(\mu)}$] (Rinfmu) at (\xRinfmu,\yHp) {};
\node[dot] (RinfmuExA) at (\xRinfmuExA,\yHp) {};
\node[dot] (RinfmuExB) at (\xRinfmuExB,\yHp) {};

\draw (P0) -- (Q0) -- (R0);
\draw (R1) -- (Q0);
\draw (R2) -- (Q0);
\draw (Pinf) -- (Qinf);


\draw (Rinfmu0) -- (Qinf);
\draw (Rinfmu) -- (Qinf);
\draw (RinfmuExA) -- (Qinf);
\draw (RinfmuExB) -- (Qinf);

\end{tikzpicture}
 